\documentclass{article}

\PassOptionsToPackage{square, sort, numbers, compress}{natbib}
   \usepackage[preprint]{neurips_2024}

\usepackage[utf8]{inputenc} % allow utf-8 input
\usepackage[T1]{fontenc}    % use 8-bit T1 fonts
\usepackage{hyperref}       % hyperlinks
\usepackage{url}            % simple URL typesetting
\usepackage{booktabs}       % professional-quality tables
\usepackage{amsfonts}       % blackboard math symbols
\usepackage{nicefrac}       % compact symbols for 1/2, etc.
\usepackage{microtype}      % microtypography
\usepackage{xcolor}         % colors

\usepackage{graphicx} % Required for inserting images
\usepackage{amsmath}
\usepackage{amssymb}
\usepackage{amsthm}
\usepackage{xcolor}
\usepackage{algorithm}
\usepackage{algpseudocode}
\usepackage{titlesec}
\usepackage{appendix}
\usepackage[font=small,labelfont=bf,tableposition=top]{caption}

\DeclareCaptionLabelFormat{andtable}{#1~#2  \&  \tablename~\thetable}
\DeclareMathOperator*{\argmin}{arg\,min}

\title{Online Nonconvex Bilevel Optimization with Bregman Divergences}

\author{%
  Jason Bohne \\
  Department of Applied Mathematics and Statistics\\
  Stony Brook University\\
  Stony Brook, NY 11794\\
  \texttt{jason.bohne@stonybrook.edu} \\
   \And David Rosenberg\\
   Bloomberg \\
   Toronto, ON M5J 2S1\\
   \texttt{drosenberg44@bloomberg.net}\\
   \And Gary Kazantsev\\
   Bloomberg \\
   New York City, NY 10022\\
   \texttt{gkazantsev@bloomberg.net}
   \And
  Pawe\l\ Polak \\
  Department of Applied Mathematics and Statistics\\
   Institute for Advanced Computational Science \\
  Stony Brook University\\
    Stony Brook, NY 11794\\
  \texttt{pawel.polak@stonybrook.edu} \\
}

\newtheorem{theorem}{Theorem}[section]

\newtheorem{lemma}[theorem]{Lemma}

\newtheoremstyle{assumptionstyle} % Name
  {3pt} % Space above
  {3pt} % Space below
  {\itshape} % Body font
  {} % Indent amount
  {\bfseries} % Theorem head font
  {.} % Punctuation after theorem head
  {.5em} % Space after theorem head
  {} % Theorem head spec

\theoremstyle{assumptionstyle}
\newtheorem{assumption}{Assumption}
\begin{document}

\maketitle

\begin{abstract}
Bilevel optimization methods are increasingly relevant within machine learning, especially for tasks such as hyperparameter optimization and meta-learning. Compared to the offline setting, online bilevel optimization (OBO) offers a more dynamic framework by accommodating time-varying functions and sequentially arriving data. This study addresses the online nonconvex-strongly convex bilevel optimization problem. In deterministic settings, we introduce a novel online Bregman bilevel optimizer (OBBO) that utilizes adaptive Bregman divergences. We demonstrate that OBBO enhances the known sublinear rates for bilevel local regret  through a novel hypergradient error decomposition that adapts to the underlying geometry of the problem. In stochastic contexts, we introduce the first stochastic online bilevel optimizer (SOBBO), which employs a window averaging method for updating outer-level variables using a weighted average of recent stochastic approximations of hypergradients. This approach not only achieves sublinear rates of bilevel local regret  but also serves as an effective variance reduction strategy, obviating the need for additional stochastic gradient samples at each timestep. Experiments on online hyperparameter optimization and online meta-learning highlight the superior performance, efficiency, and adaptability of our Bregman-based algorithms compared to established online and offline bilevel benchmarks.
\end{abstract}

\section{Introduction}
Online bilevel optimization (OBO) is a recently introduced strategy that complements traditional offline bilevel optimization approaches. It is tailored for dynamic environments encountered in machine learning tasks such as online hyperparameter optimization \cite{sobow}, online meta-learning \cite{tarzanagh2024online}, domain adaptation \cite{grangier2023adaptive},  and dataset augmentation \cite{mounsaveng2020learning}. The deterministic OBO problem, for parameters $\boldsymbol{\lambda} \in \mathcal{X} \subseteq \mathbb{R}^{d_1}$, $\boldsymbol{\beta} \in \mathbb{R}^{d_2}$, and the online running index $\forall t\in[1,T]$, is defined as:
\begin{align}\label{eq:deterministic_obo}
    &\widehat{\boldsymbol{\lambda}}_t \in \argmin_{\boldsymbol{\lambda} \in \mathcal{X} \subseteq \mathbb{R}^{d_1}} \left\{f_t\left(\boldsymbol{\lambda},\widehat{\boldsymbol{\beta}}_t(\boldsymbol{\lambda})\right) + h(\boldsymbol{\lambda})\right\},
\quad 
    \widehat{\boldsymbol{\beta}}_t(\boldsymbol{\lambda}) \in \argmin_{\boldsymbol{\beta} \in \mathbb{R}^{d_2}} g_t(\boldsymbol{\lambda},\boldsymbol{\beta}),
\end{align}
where $F_t(\boldsymbol{\lambda})\triangleq f_t\left(\boldsymbol{\lambda},\widehat{\boldsymbol{\beta}}_t(\boldsymbol{\lambda})\right)$ is a nonconvex and smooth outer objective function, $h(\boldsymbol{\lambda})$ is a convex and potentially  nonsmooth regularization term, and $g_t(\boldsymbol{\lambda},\boldsymbol{\beta})$ is a smooth inner objective function that is $\mu_g$-strongly convex in $\boldsymbol{\beta}$.

Single-level online optimization problems are often formulated with time-varying objectives that are only accessible through a stochastic oracle, for example, a stochastic gradient oracle in the online learning of an LSTM in \cite{aydore2019dynamic} and in the online learning of a GAN  in \cite{hazan2017efficient}. Similarly,  many machine learning tasks in offline bilevel optimization are structured using stochastic loss functions to capture the inherent randomness and perturbations from using sample batches. Examples include the meta-learning task of \cite{ji2021bilevel} and the data hyper-cleaning task as in \cite{huang2022enhanced}. This motivates us  to introduce the first stochastic online bilevel optimization problem:
\begin{align}\label{eq:stochastic_obo}
    &\widehat{\boldsymbol{\lambda}}_t \in \argmin_{\boldsymbol{\lambda} \in \mathcal{X} \subseteq \mathbb{R}^{d_1}} \left\{\mathbb{E}_{\epsilon}\left[f_t(\boldsymbol{\lambda},\widehat{\boldsymbol{\beta}}_t(\boldsymbol{\lambda}), \epsilon)\right] + h(\boldsymbol{\lambda})\right\},
\quad
    &\widehat{\boldsymbol{\beta}}_t(\boldsymbol{\lambda}) \in \argmin_{\boldsymbol{\beta} \in \mathbb{R}^{d_2}} \mathbb{E}_{\zeta}\left[g_t(\boldsymbol{\lambda},\boldsymbol{\beta}, \zeta)\right]
\end{align}
where $F_{t}(\boldsymbol{\lambda})\triangleq \mathbb{E}_{\epsilon}\left[f_t(\boldsymbol{\lambda}, \widehat{\boldsymbol{\beta}}_t(\boldsymbol{\lambda}), \epsilon)\right]$ is a nonconvex and smooth outer objective function, $h(\boldsymbol{\lambda})$ is a convex and potentially nonsmooth regularization term, and $\mathbb{E}_{\zeta}\left[g_t(\boldsymbol{\lambda}, \boldsymbol{\beta}, \zeta)\right]$ is a smooth inner objective function that is $\mu_g$-strongly convex in $\boldsymbol{\beta}$. Further, $\epsilon$ and $\zeta$ are independent and identically distributed random variables from error distributions $D$ and $D'$, respectively.

An OBO framework is necessary in settings where both the inner and outer objective functions can change over time and data can arrive sequentially in an online manner. One current state-of-the-art algorithm is online alternating gradient descent (OAGD) proposed by \cite{tarzanagh2024online}, which utilizes an alternating gradient descent step to update inner and outer level variables. While OAGD achieves a sublinear bilevel local regret, previous objective functions are required to evaluate the current iterates, resulting in a computationally expensive scheme for the outer level variable gradient step. Alternatively, the single-loop online bilevel optimizer with window averaging (SOBOW) introduced by \cite{sobow} offers a computational improvement compared to OAGD through a clever decomposition of the error of the  hypergradient and a single loop iteration. 

However, while these methods significantly contribute to the development of OBO, they are limited in several ways: (i) they assume smooth outer level objectives, excluding OBO problems where the outer level objective is separable into a smooth and nonsmooth term; (ii) they use Euclidean proximal operators, preventing the application of more sophisticated gradient step techniques utilizing Bregman divergences as explored in the offline bilevel optimization of \cite{huang2022enhanced}; and (iii) they restrict the theoretical discussion to a deterministic setting, whereas stochasticity is commonly explored and assumed in both single level online optimization \cite{hallak2021regret} and offline bilevel optimization \cite{ji2021bilevel}.

Addressing these limitations, our contributions to OBO include introducing a novel family of online Bregman bilevel optimizers (OBBO) that utilize Bregman divergences to solve a larger class of OBO problems. These problems include cases where the outer-level objective is a nonconvex composite function separable into a smooth and potentially nonsmooth term, and the inner-level objective is strongly convex. Due to the application of Bregman divergences to our novel hypergradient error decomposition, we show that OBBO improves the previously best known sublinear rate of bilevel local regret.  Further, our  hypergradient error decomposition  can be a theoretical interest of its own, as we provide a decomposition not expressed in  previous works (\cite{tarzanagh2024online,sobow}) which depend on the variation of outer level variables. Additionally, we formulate the first stochastic OBO problem and provide an algorithm (SOBBO) that achieves a sublinear rate of bilevel local regret. We demonstrate the computational efficiency of SOBBO compared to offline stochastic bilevel algorithms, as it incorporates a variance reduction technique without requiring additional gradient samples.

\section{Related Work}

Online optimization is a well-developed area of research that extends the offline optimization problem to an objective function that can change over time. For a convex time-varying objective, gradient-based methods were proposed within \cite{zinkevich2003} and achieved a sublinear rate of static regret. Improved rates of static regret  are provided in \cite{mcmahan2010adaptive} through utilizing a strongly-convex adaptive regularization function and are further generalized to a nonsmooth regularization term within \cite{pmlr-v15-mcmahan11b}. In non-stationary  and convex environments as in \cite{hall2013dynamical}, \cite{zhao2020dynamic}, and   \cite{besbes2015non}, the static regret performance metric is too optimistic, and instead, a dynamic regret measure should be used to capture the changing comparator sequence. For a time-varying and nonconvex objective function, there has been recent interest in online optimization algorithms that achieve sublinear local regret, a stationary measure introduced within \cite{hazan2017efficient}. Recent work extends local regret  to non-stationary environments, as in \cite{aydore2019dynamic} and stochastic gradient oracles in \cite{hallak2021regret}. The use of Bregman divergences has notably improved rates of regret  in online optimization, as demonstrated by the seminal Adagrad algorithm \cite{JMLR:v12:duchi11a}, and further developed in implicit online learning algorithms in \cite{Kulis2010ImplicitOL}.

Offline bilevel optimization has a rich history with its introduction in  constrained optimization  of \cite{mcgill}. For applications in modern-day machine learning, gradient-based optimization methods are often preferred due to their ease of implementation and scalability. One class of methods is approximate implicit differentiation techniques, detailed in \cite{pedregosa2016hyperparameter} and \cite{lorraine2020optimizing}, which construct a hypergradient approximation through a conjugate gradient or fixed point approach. Another common method is to use iterative differentiation techniques and their efficient computational implementations to form an approximate hypergradient (see e.g., \cite{maclaurin2015gradient} and \cite{franceschi2017forward}). Convergence improvements are achieved in the Bio-BreD algorithm of \cite{huang2022enhanced} through incorporating Bregman divergences in the outer level gradient step. Moreover, implications of the Bregman-based optimizers of \cite{huang2022enhanced} allow for integration of adaptive learning rates and momentum in offline bilevel optimization, such as in the BiAdam framework of \cite{huang2021biadam}. Due to the inherent uncertainty of utilizing sample batches, stochastic formulations of offline bilevel optimization are of recent interest. A stochastic approximation algorithm is introduced in \cite{ghadimi2018approximation} that utilizes a single stochastic gradient sample at each iteration. Convergence rates are improved in Stoc-BiO within \cite{ji2021bilevel} through an improved stochastic hypergradient estimator. Sample-efficient stochastic algorithms of SUSTAIN and STABLE are discussed in \cite{khanduri2021near} and \cite{pmlr-v151-chen22e} by employing momentum-based variance reduction recurrences.

OBO is underdeveloped relative to single-level online optimization and offline bilevel optimization. The online alternating gradient descent algorithm (OAGD) of \cite{tarzanagh2024online} is shown to achieve  sublinear bilevel local regret  for a nonconvex time-varying objective. The single-loop online bilevel optimizer with window averaging (SOBOW) of \cite{sobow} offers a  computational improvement while maintaining a sublinear rate. The general problem formulation of an OBO framework fits time-varying machine learning tasks not within the constraints of offline bilevel optimization, such as learning optimal dataset augmentations in \cite{mounsaveng2020learning} and online adaptation of pre-trained distributions of \cite{grangier2023adaptive}.

\section{Preliminaries}
\subsection{Notations and Assumptions}
Let $\left\|\cdot\right\|$ denote the $\ell_2$ norm for vectors and the spectral norm for matrices, with $\langle \boldsymbol{\beta}_1,\boldsymbol{\beta}_2\rangle$ denoting the inner product between $\boldsymbol{\beta}_1$ and $\boldsymbol{\beta}_2$. For  a function $g_t(\boldsymbol{\lambda},\boldsymbol{\beta})$ and its stochastic variation $g_t(\boldsymbol{\lambda},\boldsymbol{\beta},\zeta)$  we denote the gradient  as $\nabla g_t(\boldsymbol{\lambda},\boldsymbol{\beta})$ and $\nabla g_t(\boldsymbol{\lambda},\boldsymbol{\beta},\zeta) $ respectively, with partial derivatives denoted, for example with respect to $\boldsymbol{\lambda}$,  as  $\nabla_{\boldsymbol{\lambda}} g_t(\boldsymbol{\lambda},\boldsymbol{\beta})$ and $\nabla_{\boldsymbol{\lambda}} g_t(\boldsymbol{\lambda},\boldsymbol{\beta},\zeta)$ respectively. For the deterministic OBO problem, we make the following assumptions that are standard in  OBO \cite{tarzanagh2024online,sobow},  whereas, in the stochastic setting, our assumptions are typical in offline stochastic bilevel optimization of \cite{ghadimi2018approximation,ji2021bilevel}.
\begin{assumption}\label{assump:rel_smoothness} 
For  all $ t\in[1,\ldots,T]$ and $\forall \boldsymbol{\lambda} \in\mathcal{X}$, $\forall\boldsymbol{\beta}\in\mathbb{R}^{d_2}$, we have the following:\\
A1. $\ f_t(\boldsymbol{\lambda},\boldsymbol{\beta}) $ and $\nabla f_t(\boldsymbol{\lambda},\boldsymbol{\beta})  $ are respectively $\ell_{f,0}$-Lipschitz and $\ell_{f,1}$-Lipschitz continuous. \\A2. $\ \nabla g_t(\boldsymbol{\lambda},\boldsymbol{\beta})  $ is $\ell_{g,1}$-Lipschitz continuous.\\
A3. $\ \nabla^2_{\boldsymbol{\lambda},\boldsymbol{\beta}} g_t(\boldsymbol{\lambda},\boldsymbol{\beta})  $ and $\nabla^2_{\boldsymbol{\beta},\boldsymbol{\beta}} g_t(\boldsymbol{\lambda},\boldsymbol{\beta}) $  are both $\ell_{g,2}$-Lipschitz continuous. \\
In the stochastic case  A1.-A3. hold for $f_t(\boldsymbol{\lambda},\boldsymbol{\beta},\epsilon)$ and $g_t(\boldsymbol{\lambda},\boldsymbol{\beta},\zeta)$ for any $\epsilon\sim D$ and $\zeta\sim D'$.
\end{assumption}
\begin{assumption}\label{assump:strong_convex_g_t} $g_t(\boldsymbol{\lambda},\boldsymbol{\beta})$ and its stochastic variation $g_t(\boldsymbol{\lambda},\boldsymbol{\beta},\zeta)$ are $\mu_g$-strongly convex functions in $\boldsymbol{\beta}$ $ \forall t \in [1,T]$  and for all $\boldsymbol{\lambda} \in \mathcal{X}$.
\end{assumption}
All the above smoothness assumptions are standard in OBO  \cite{tarzanagh2024online,sobow} and are required for our convergence analysis. They also imply that the condition number $\kappa_g:=\frac{\ell_{g,1}}{\mu_g}\geq 1$. The next Assumption \ref{assump:unbiased_finite_var} is crucial in the stochastic setting for ensuring the reliability of stochastic optimizers, maintaining that gradient steps converge accurately while managing  stochastic variability \cite{huang2022enhanced}.

\begin{assumption}\label{assump:unbiased_finite_var} 
For all $t \in[1,T]$ and  any $\boldsymbol{\lambda}\in \mathcal{X}$ and $\boldsymbol{\beta}\in\mathbb{R}^{d_2}$ the stochastic partial derivative of  $\nabla_{\boldsymbol{\beta}} g_t(\boldsymbol{\lambda},\boldsymbol{\beta},\zeta)$ satisfies $\mathbb{E}\left[\nabla_{\boldsymbol{\beta}} g_t(\boldsymbol{\lambda},\boldsymbol{\beta},\zeta)\right]=\nabla_{\boldsymbol{\beta}} g_t(\boldsymbol{\lambda},\boldsymbol{\beta}) $ with finite variance $\sigma^2_{g_{\boldsymbol{\beta}}}$. Further the estimated stochastic partial derivative  of $\widetilde{\nabla}f_t(\boldsymbol{\lambda},\boldsymbol{\beta},\mathcal{E})$  has finite variance $\sigma^2_f$.
\end{assumption}
The next two assumptions are common in online nonconvex optimization, where decision variables and a time-varying objective function have natural bounds, see \cite{hazan2017efficient,aydore2019dynamic}. 
\begin{assumption}\label{assump:bounded_hyperparm} The set $\mathcal{X}\subseteq \mathbb{R}^{d_1}$ is closed, convex, and bounded s.t. $\forall \ \boldsymbol{\lambda}_1,\boldsymbol{\lambda}_2 \in \mathcal{X}$,  $\left\|\boldsymbol{\lambda}_1-\boldsymbol{\lambda}_2\right\|\leq S$. 
\end{assumption}
\begin{assumption}\label{assump:bounded_F_t} For all $t\in[1,\ldots,T]$ it holds for finite $Q\in\mathbb{R}$ that $\sup_{\boldsymbol{\lambda}\in \mathcal{X}}\left|F_t(\boldsymbol{\lambda})\right|\leq Q$.
\end{assumption}
The next Assumption \ref{assump:continuity_phi_t} on the continuity and $\rho$-strong convexity of the distance generating function $\phi_t$  is required for constructing well-defined Bregman divergences, \cite{huang2022enhanced,huang2022bregman}. 
\begin{assumption}\label{assump:continuity_phi_t} For all $t\in[1,T]$ the distance generating function $\phi_t(\boldsymbol{\lambda})$ is  continuously differentiable and $\rho$-strongly convex  with respect to $\left\|\cdot\right\|$.
\end{assumption}

\subsection{Bregman Proximal Gradient}\label{defn:bregman_divergence} 
Originally formulated within \cite{BREGMAN1967200}, a Bregman divergence $ \mathcal{D}_{\phi}(\cdot, \cdot)$ offers  a generalization to the squared Euclidean distance and is defined for a continuously differentiable and $\rho$-strongly convex function $\phi(\boldsymbol{\lambda})$ with respect to  $\left\|\cdot\right\|$,  for all $\boldsymbol{\lambda}_1, \boldsymbol{\lambda}_2 \in \mathcal{X}$ as:
\begin{align}\label{eq:bregman_divergence}
    \mathcal{D}_{\phi}(\boldsymbol{\lambda}_2, \boldsymbol{\lambda}_1) := \phi(\boldsymbol{\lambda}_2) - \phi(\boldsymbol{\lambda}_1) - \langle \nabla \phi(\boldsymbol{\lambda}_1), \boldsymbol{\lambda}_2 - \boldsymbol{\lambda}_1 \rangle.
\end{align}
Given a Bregman divergence $\mathcal{D}_{\phi}(\cdot, \cdot)$  defined in \eqref{eq:bregman_divergence}, the proximal gradient step is stated as
\begin{align}\label{eq:outer_step_minimization}
\boldsymbol{\lambda}^{+} = \argmin_{\boldsymbol{\lambda} \in \mathcal{X}} \left\{\langle \boldsymbol{q}, \boldsymbol{\lambda} \rangle + h(\boldsymbol{\lambda}) + \frac{1}{\alpha} \mathcal{D}_{\phi}(\boldsymbol{\lambda}, \boldsymbol{u})\right\},
\end{align}
where  $\phi(\boldsymbol{\lambda})$ is a continuously differentiable and $\rho$-strongly convex function, $h(\boldsymbol{\lambda})$ is a convex and potentially nonsmooth regularization term, $\alpha>0$ is a step size, and $\boldsymbol{q},\boldsymbol{u}\in\mathbb{R}^{d_1}$ are the estimate of the gradient, and current reference point, respectively. Proximal gradient methods in offline bilevel optimization have been shown to improve convergence rates as in the Bio-BreD algorithm of \cite{huang2022enhanced}. Special cases of the gradient update in \eqref{eq:outer_step_minimization} include  projected gradient descent ($\phi(\boldsymbol{\lambda})=\frac{1}{2}\left\|\boldsymbol{\lambda}\right\|^2$, $\mathcal{X}\subseteq\mathbb{R}^{d_1}$, and $h(\boldsymbol{\lambda})=0$), as well as proximal gradient descent ($\phi(\boldsymbol{\lambda})=\frac{1}{2}\left\|\boldsymbol{\lambda}\right\|^2$ and $\mathcal{X}=\mathbb{R}^{d_1}$).  The aforementioned gradient step in \eqref{eq:outer_step_minimization}  can be further extended to a time-varying distance generating function, e.g., $\phi_t(\boldsymbol{\lambda})=\frac{1}{2}\boldsymbol{\lambda}^T\mathbf{H}_t\boldsymbol{\lambda}$ with an adaptive matrix $\mathbf{H}_t$, resulting in an adaptive proximal gradient method with similarities to  Adagrad from \cite{JMLR:v12:duchi11a} and Super-Adam of \cite{huang2021super}.

The proximal gradient step of \eqref{eq:outer_step_minimization} has led to the introduction of a generalized projection from \cite{ghadimi2016mini} defined for a step size $\alpha>0$, $\boldsymbol{q}\in\mathbb{R}^{d_1}$, and $\boldsymbol{u} \in \mathcal{X}$ as
\begin{align}\label{defn:generalized_gradient}
\mathcal{G}_{\mathcal{X}}(\boldsymbol{u},\boldsymbol{q},\alpha) := \frac{1}{\alpha}\left(\boldsymbol{u} - \boldsymbol{\lambda}^+\right).
\end{align} 
 $\mathcal{G}_{\mathcal{X}}(\boldsymbol{\lambda},\nabla f_t(\boldsymbol{\lambda}),\alpha)$ acts as a generalized gradient that simplifies  to $\nabla f_t(\boldsymbol{\lambda})$ if $\mathcal{X}=\mathbb{R}^{d_1}$ and $h(\boldsymbol{\lambda})=0$.

\subsection{Bilevel Local Regret  }
Bilevel local regret is a stationary metric for online bilevel optimization \cite{tarzanagh2024online,sobow} that extends the single-level local regret measure from \cite{hazan2017efficient}. The work of \cite{sobow} in particular defines the bilevel local regret for a window length $w\geq 1$ and a sequence  $\{\boldsymbol{\lambda}_t\}_{t=1}^T$ as
\begin{align}\label{eq:det_sobow_local_regret}
    BLR_{w}(T) := \sum_{t=1}^T \left\|\nabla F_{t,w}(\boldsymbol{\lambda}_t)\right\|^2, \quad  F_{t,w}(\boldsymbol{\lambda}_t) := \frac{1}{w}\sum_{i=0}^{w-1} F_{t-i}(\boldsymbol{\lambda}_{t-i}),
\end{align}
where for simplicity  we have defined $F_{t,w}(\boldsymbol{\lambda}_ t)$ as a time-smoothed  outer level objective with $F_t = 0 \ \forall t \leq 0$. Note, that the time-smoothed outer-level objective is defined as the weighted average of previous objective functions evaluated at the respective outer-level variables. Computationally efficient algorithms can be derived under this definition of bilevel local regret as only the previous evaluations are required, which can be stored in memory, instead of the  objective functions  such as OAGD. A similar notion of local regret has been introduced by \cite{aydore2019dynamic} in online single-level optimization under nonstationary environments. To incorporate the generalized projection of \eqref{defn:generalized_gradient},  we introduce a new bilevel local regret for a window length $w\geq 1$ and sequence   $\{\boldsymbol{\lambda}_t\}_{t=1}^T$  as
\begin{align}\label{eq:det_local_regret}
    BLR_{w}(T) := \sum_{t=1}^T \left\|\mathcal{G}_{\mathcal{X}}(\boldsymbol{\lambda}_t,\nabla F_{t,w}(\boldsymbol{\lambda}_t),\alpha)\right\|^2
\end{align}
In the setting where $\mathcal{X} = \mathbb{R}^{d_1}$, $h(\boldsymbol{\lambda}) = 0$,  and $\phi_t(\boldsymbol{\lambda})=\phi(\boldsymbol{\lambda}) = \frac{1}{2}\left\|\boldsymbol{\lambda}\right\|^2$, our variation of local regret simplifies to the regret measure of \eqref{eq:det_sobow_local_regret}, the metric used in the analysis of SOBOW in  \cite{sobow}.  While novel to OBO, similar convergence metrics are considered in the analysis of offline bilevel algorithms in \cite{ghadimi2018approximation} and \cite{huang2022enhanced}. However, our definition offers an important generalization of bilevel local regret when an adaptive distance generating function $\phi_t(\boldsymbol{\lambda})$ or a non-zero regularization term $h(\boldsymbol{\lambda})$ is present.  

Proposition 1 in \cite{besbes2015non} shows that in nonstationary environments, there always exists a sequence of well-behaved loss functions for which sublinear regret cannot be achieved. Hence, to derive useful regret bounds of online algorithms, further regularity constraints must be imposed on the sequence, such as a sublinear path variation \cite{yang2016tracking}, function variation \cite{besbes2015non}, or gradient variation \cite{pmlr-v23-chiang12}. In nonstationary OBO one proposed regularity constraint is the $p$-th order inner level path variation of optimal decisions from \cite{tarzanagh2024online}, denoted as $H_{p,T}$. A regularity metric on the $p$-th order variation of the evaluations of the outer level objective function across time is suggested by \cite{sobow} and denoted as $V_{p,T}$.
\begin{align}\label{eq:inner_level_variation}
H_{p,T}:=\sum_{t=2}^T\sup_{\boldsymbol{\lambda}\in \mathcal{X}}\left\|\widehat{\boldsymbol{\beta}}_{t-1}(\boldsymbol{\lambda})-\widehat{\boldsymbol{\beta}}_{t}(\boldsymbol{\lambda})\right\|^p, \quad V_{p,T}:=\sum_{t=1}^{T} \sup_{\boldsymbol{\lambda} \in \mathcal{X}}\left|F_{t+1}\left(\boldsymbol{\lambda}\right)-F_{t}\left(\boldsymbol{\lambda}\right)\right|^p
\end{align}
Note  the latter regularity metric, $V_{p,T}$, tracks how  the optimal outer level variable, which is fixed for a given $t\in[1, T]$ ,  can vary over time.  We will be utilizing the regularity metrics of second-order inner-level path variation, $H_{2, T}$,  and first-order variation of the evaluations of the outer level objective, $V_{1, T}$, given in \eqref{eq:inner_level_variation}, and impose a sublinear constraint, that is $H_{2, T}=o(T)$ and $V_{1, T}=o(T)$.

\section{Online Bregman Bilevel Optimizers}
\subsection{Deterministic Algorithm (\textbf{OBBO})}
We begin our section on the deterministic online Bregman bilevel optimizer OBBO with a few useful lemmas on the gradient of the outer level objective $\nabla F_t(\boldsymbol{\lambda})$,  which is often referred as the hypergradient \cite{ghadimi2018approximation}. The first lemma expands the hypergradient with the chain rule followed by an implicit function theorem to get the second equality.
\begin{lemma}\label{lem:hypergrad_estimator_itd} (Lemma 2.1 in \cite{ghadimi2018approximation})
Under Assumption A and C, we have $\forall \boldsymbol{\lambda}\in\mathcal{X}$, $\forall t\in[1,T]$ 
\begin{align}\label{eq:true_hg_val}
    \nabla F_t(\boldsymbol{\lambda})= \nabla_{\boldsymbol{\lambda}} f_t(\boldsymbol{\lambda},\widehat{\boldsymbol{\beta}}_t(\boldsymbol{\lambda}))+ \nabla\widehat{\boldsymbol{\beta}}_t(\boldsymbol{\lambda})\nabla_{\boldsymbol{\beta}}f_t(\boldsymbol{\lambda},\widehat{\boldsymbol{\beta}}_t(\boldsymbol{\lambda}))\nonumber\\=\nabla_{\boldsymbol{\lambda}} f_t(\boldsymbol{\lambda},\widehat{\boldsymbol{\beta}}_t(\boldsymbol{\lambda})) - \nabla^2_{\boldsymbol{\lambda},\boldsymbol{\beta}} g_t(\boldsymbol{\lambda},\widehat{\boldsymbol{\beta}}_t(\boldsymbol{\lambda}))\left(\nabla^2_{\boldsymbol{\beta},\boldsymbol{\beta}}g_t(\boldsymbol{\lambda},\widehat{\boldsymbol{\beta}}_t(\boldsymbol{\lambda}))\right)^{-1}\nabla_{\boldsymbol{\beta}}f(\boldsymbol{\lambda},\widehat{\boldsymbol{\beta}}_t(\boldsymbol{\lambda})).
\end{align}\end{lemma}
The gradient decomposition of Lemma \ref{lem:hypergrad_estimator_itd} is a common expansion in bilevel optimization that utilizes the smoothness and strong convexity requirements of Assumption \ref{assump:rel_smoothness} and \ref{assump:strong_convex_g_t}. However the computation of $\nabla F_t(\boldsymbol{\lambda})$ requires knowledge of the optimal inner level variables $\widehat{\boldsymbol{\beta}}_t(\boldsymbol{\lambda})$, which are typically inaccessible. This has motivated the construction of computationally efficient gradient estimators for $\nabla F_t(\boldsymbol{\lambda})$ that for a fixed $\boldsymbol{\lambda}\in\mathcal{X}$ and $\boldsymbol{\beta}\in\mathbb{R}^{d_2}$ can provide an approximation to $\nabla F_t(\boldsymbol{\lambda})$ without access to $\widehat{\boldsymbol{\beta}}_t(\boldsymbol{\lambda})$. One such example, is the gradient estimator proposed within the work of \cite{ghadimi2018approximation}, which provides an approximation of $\nabla F_t(\boldsymbol{\lambda})$ for a fixed $\boldsymbol{\lambda}\in\mathcal{X}$ and $\boldsymbol{\beta}\in\mathbb{R}^{d_2}$ as
    $\widetilde{\nabla} f_t(\boldsymbol{\lambda},\boldsymbol{\beta}):=\nabla_{\boldsymbol{\lambda}} f_t(\boldsymbol{\lambda},\boldsymbol{\beta}) -\nabla^2_{\boldsymbol{\lambda},\boldsymbol{\beta}} g_t(\boldsymbol{\lambda},\boldsymbol{\beta})\left(\nabla^2_{\boldsymbol{\beta},\boldsymbol{\beta}}g_t(\boldsymbol{\lambda},\boldsymbol{\beta})\right)^{-1}\nabla_{\boldsymbol{\beta}}f(\boldsymbol{\lambda},\boldsymbol{\beta})$. Another popular approach is to utilize iterative differentiation techniques for gradient estimation in offline bilevel optimization, see the work of \cite{maclaurin2015gradient} as well as  \cite{ji2021bilevel}, and hence, we examine our OBBO algorithm in this case\footnote{While we choose the hypergradient estimate of $\frac{\partial f_t(\boldsymbol{\lambda}_t,\boldsymbol{\omega}^K_t)}{\partial \boldsymbol{\lambda}} $ for  our analysis of OBBO, other hypergradient estimates can be used. }. Through the iterative gradient descent updates of   OBBO  we analyze the estimate of  $\widetilde{\nabla} f_t(\boldsymbol{\lambda}_t,\boldsymbol{\omega}^K_t):=\frac{\partial f_t(\boldsymbol{\lambda}_t,\boldsymbol{\omega}^K_t)}{\partial \boldsymbol{\lambda}}$, that has an analytical form in Lemma \ref{lem:ji_2021}. 
\begin{lemma}\label{lem:ji_2021}(Proposition 2 in \cite{ji2021bilevel}) The partial  $\frac{\partial f_t(\boldsymbol{\lambda}_t,\boldsymbol{\omega}^K_t)}{\partial \boldsymbol{\lambda}}$ takes an analytical form of $\frac{\partial f_t(\boldsymbol{\lambda}_t,\boldsymbol{\omega}^K_t)}{\partial \boldsymbol{\lambda}}=$
\begin{align}\label{eq:analytical_backprop}
\nabla_{\boldsymbol{\lambda}}f_t\left(\boldsymbol{\lambda}_t,\boldsymbol{\omega}^K_t\right)-\eta \sum_{k=0}^{K-1}\nabla^2_{\boldsymbol{\lambda},\boldsymbol{\omega}}g_t\left(\boldsymbol{\lambda}_t,\boldsymbol{\omega}^k_t\right)H_{\boldsymbol{\omega},\boldsymbol{\omega}}\nabla_{\boldsymbol{\omega}}f_t\left(\boldsymbol{\lambda}_t,\boldsymbol{\omega}^K_t\right),
\end{align}
\end{lemma}
where $H_{\boldsymbol{\omega},\boldsymbol{\omega}}:=\prod_{j=k+1}^{K-1}\left(I_{d_2}-\eta\nabla^2_{\boldsymbol{\omega},\boldsymbol{\omega}}g_t\left(\boldsymbol{\lambda}_t,\boldsymbol{\omega}^j_t\right)\right)$,  the $d_2$-identity matrix is denoted  $I_{d_2}$, with  $\eta>0$ and $K$  as the step size and number of iterations for the inner loop.  In practice, automatic differentiation is applied to compute second-order derivatives in \eqref{eq:analytical_backprop}, e.g.,  PyTorch of \cite{paszke2019pytorch}. 

\begin{algorithm}[]
\begin{algorithmic}
\caption{OBBO Deterministic Online Bregman Bilevel Optimizer }\label{alg:deterministic_online_bregman}
\Require $T,K$; stepsizes $\alpha,\eta>0$;   distance generating functions $\phi_t(\boldsymbol{\lambda}): \mathcal{X}   \mapsto \mathbb{R}$; window  $w\geq 1$. \vspace{0.8mm}
\vspace{0.6mm}\\
\textbf{Initialize} $\boldsymbol{\beta}_{1} \in \mathbb{R}^{d_2} $ and $\boldsymbol{\lambda}_1 \in \mathcal{X}$
\For{$t=1, \ldots, T$}
\vspace{0.6mm}
\State Retrieve information about $f_t$ and $g_t$
\vspace{0.6mm}
\State $\boldsymbol{\omega}^0_t \gets \boldsymbol{\beta}_t$ 
\vspace{0.6mm}
\For{$k=1, \ldots, K$}
\vspace{0.6mm}
\State $\boldsymbol{\omega}_t^{k} \gets \boldsymbol{\omega}_t^{k-1}-\eta\nabla_{\boldsymbol{\omega}} g_t(\boldsymbol{\lambda}_t,\boldsymbol{\omega}_t^{k-1}) $ 
\EndFor
\vspace{0.6mm}
\State  Get $\widetilde{\nabla} f_t(\boldsymbol{\lambda}_t,\boldsymbol{\omega}^K_t):=\frac{\partial f_t(\boldsymbol{\lambda}_t,\boldsymbol{\omega}^K_t)}{\partial \boldsymbol{\lambda}} $ from \eqref{eq:analytical_backprop} and store in memory
\vspace{0.6mm}
\State Compute $\widetilde{\nabla} f_{t,w} (\boldsymbol{\lambda}_t,\boldsymbol{\beta}_{t+1} )$ from \eqref{eq:time_smooth_hypergradient} for $\boldsymbol{\beta}_{t+1}=\boldsymbol{\omega}^K_t$
\vspace{0.6mm}
\State $\boldsymbol{\lambda}_{t+1} \gets \argmin_{\boldsymbol{\lambda} \in \mathcal{X}} \left\{\left\langle \widetilde{\nabla} f_{t,w} (\boldsymbol{\lambda}_t,\boldsymbol{\beta}_{t+1} ), \boldsymbol{\lambda} \right\rangle + h(\boldsymbol{\lambda}) + \frac{1}{\alpha} \mathcal{D}_{\phi_t}(\boldsymbol{\lambda}, \boldsymbol{\lambda}_t)\right\}$
\vspace{0.6mm}
\EndFor\\
\vspace{0.6mm}
\Return $\boldsymbol{\lambda}_{T+1},\boldsymbol{\beta}_{T+1}$
\end{algorithmic}
\end{algorithm}

In Theorem 2.7 of \cite{hazan2017efficient}, time smoothing is shown to be a necessary component of the gradient step for an online algorithm to achieve a sublinear rate of bilevel local regret. Following such, we introduce dynamic time-smoothing into our OBBO algorithm by first defining the  gradient estimator $\widetilde{\nabla}f_{t,w}(\boldsymbol{\lambda}_t,\boldsymbol{\beta_{t+1}})$ for all $t\in[1,T]$ and a window size $w\geq 1$ as 
\begin{align}\label{eq:time_smooth_hypergradient}
    \widetilde{\nabla}f_{t,w}(\boldsymbol{\lambda}_t,\boldsymbol{\beta_{t+1}}):=\frac{1}{w}\sum_{i=0}^{w-1}
\widetilde{\nabla}f_{t-i}(\boldsymbol{\lambda}_{t-i},\boldsymbol{\beta}_{t+1-i}) 
\end{align}
 with the convention that $f_t=0 \  \forall t\leq 0$. Similar to the window averaged decision update in SOBOW of \cite{sobow}, the gradient estimator in \eqref{eq:time_smooth_hypergradient} is a weighted average of the previous evaluations, which can be stored in memory. This offers a computational improvement relative to OAGD, as only the previous evaluations are required instead of the time-varying objective functions.

\subsection{Stochastic Algorithm (\textbf{SOBBO})}
We start the section on our stochastic online Bregman bilevel optimizer SOBBO with a discussion of a stochastic gradient estimator common in offline bilevel optimization, see  \cite{ghadimi2018approximation} and \cite{huang2022enhanced}. For a sample  upper bound of $m$ and independent   $\mathcal{E}_t=\{\epsilon_t,\zeta_t^0,\ldots,\zeta_t^{m-1}\}$, the stochastic gradient of $\widetilde{\nabla}f_t(\boldsymbol{\lambda}_t,\boldsymbol{\beta}_{t+1},\mathcal{E}_t)$ provides an estimate of $\widetilde{\nabla}f_t(\boldsymbol{\lambda}_t,\boldsymbol{\beta}_{t+1})$  and is constructed $\forall t\in[1,T]$ as
\begin{align}\label{eq:stochastic_gradient_sample}
\widetilde{\nabla}f_t(\boldsymbol{\lambda}_t,\boldsymbol{\beta}_{t+1},\mathcal{E}_t):=\nabla_{\boldsymbol{\lambda}}f_t(\boldsymbol{\lambda}_t,\boldsymbol{\beta}_{t+1},\epsilon_t)-\nabla_{\boldsymbol{\lambda},\boldsymbol{\beta}}^2g_t(\boldsymbol{\lambda}_t,\boldsymbol{\beta}_{t+1},\zeta_t^0)\nonumber\\\times\left[\frac{m}{\ell_{g,1}}\prod_{j=1}^{\widetilde{m}}\left(I_{d_2}-\frac{1}{\ell_{g,1}}\nabla_{\boldsymbol{\beta}}^2g_t(\boldsymbol{\lambda}_t,\boldsymbol{\beta}_{t+1},\zeta_t^j)\right)\right]\nabla_{\boldsymbol{\beta}}f_t(\boldsymbol{\lambda}_t,\boldsymbol{\beta}_{t+1},\epsilon_t),
\end{align}
where $\widetilde{m}\sim\mathcal{U}(0,1,\ldots,m-1)$ and $\prod_{j=1}^{m=0}(\cdot)=I_{d_2}$. As in \eqref{eq:time_smooth_hypergradient}, we introduce dynamic time-smoothing into our SOBBO algorithm with a  stochastic gradient estimator $\widetilde{\nabla}f_{t,w}(\boldsymbol{\lambda}_t,\boldsymbol{\beta}_{t+1},\mathcal{Z}_{t,w})$ defined for all $t\in[1,T]$, window size $w\geq 1$, and independent samples $\mathcal{Z}_{t,w}=\{\mathcal{E}_{t-i}\}_{i=0}^{w-1}$ as
\begin{align}\label{eq:stoch_time_smooth_hypergradient}
\widetilde{\nabla}f_{t,w}(\boldsymbol{\lambda}_t,\boldsymbol{\beta}_{t+1},\mathcal{Z}_{t,w}):=\frac{1}{w}\sum_{i=0}^{w-1}
\widetilde{\nabla}f_{t-i}(\boldsymbol{\lambda}_{t-i},\boldsymbol{\beta}_{t+1-i},\mathcal{E}_{t-i}) , \quad f_t=0\ \forall t\leq 0
\end{align}

\begin{algorithm}[]
\begin{algorithmic}
\caption{ SOBBO Stochastic Online Bregman Bilevel Optimizer }\label{alg:stochastic_online_bregman}
\Require $T,K,m,s$; $\alpha,\eta>0$;   distance generating function $\phi_t(\boldsymbol{\lambda}): \mathcal{X}   \mapsto \mathbb{R}$; window  $w\geq 1$. \vspace{0.8mm}
\vspace{0.6mm}\\
\textbf{Initialize} $\boldsymbol{\beta}_{1} \in \mathbb{R}^{d_2} $ and $\boldsymbol{\lambda}_1 \in \mathcal{X}$
\vspace{0.6mm}
\For{$t=1, \ldots, T$}
\vspace{0.6mm}
\State Retrieve information about $f_t$ and $g_t$
\vspace{0.6mm}
\State $\boldsymbol{\omega}^0_t \gets \boldsymbol{\beta}_t$ 
\vspace{0.6mm}
\For{$k=1, \ldots, K$}
\vspace{0.6mm}
\State Draw $s$ independent samples of $\zeta$, Set $\bar{\zeta}_{t,k}:=\left\{\zeta^{k-1}_{t,i}\right\}_{i=1}^s$
\vspace{0.6mm}
\State $\boldsymbol{\omega}_t^{k} \gets \boldsymbol{\omega}_t^{k-1}-\eta\nabla_{\boldsymbol{\omega}} g_t(\boldsymbol{\lambda}_t,\boldsymbol{\omega}_t^{k-1},\bar{\zeta}_{t,k}) $ 
\EndFor
\vspace{0.6mm}
\State $\boldsymbol{\beta}_{t+1}\gets\boldsymbol{\omega}_t^{K}$
\vspace{0.6mm}
\State  Draw $m$ independent samples of $\zeta$, Set $\mathcal{E}_t=\left\{\epsilon_t,\zeta_t^0,\ldots,\zeta_t^{m-1}\right\}$
\vspace{0.6mm}
\State Get $\widetilde{\nabla} f_{t}(\boldsymbol{\lambda}_t,\boldsymbol{\beta}_{t+1},\mathcal{E}_t )$ from \eqref{eq:stochastic_gradient_sample} and store in memory
\vspace{0.6mm}
\State Compute $\widetilde{\nabla} f_{t,\boldsymbol{w}} (\boldsymbol{\lambda}_t,\boldsymbol{\beta}_{t+1},\mathcal{Z}_{t,w})$ from \eqref{eq:stoch_time_smooth_hypergradient}
\vspace{0.6mm}
\State $\boldsymbol{\lambda}_{t+1} \gets \argmin_{\boldsymbol{\lambda} \in \mathcal{X}} \left\{\left\langle \widetilde{\nabla} f_{t,\boldsymbol{w}} (\boldsymbol{\lambda}_t,\boldsymbol{\beta}_{t+1},\mathcal{Z}_{t,w}), \boldsymbol{\lambda} \right\rangle + h(\boldsymbol{\lambda}) + \frac{1}{\alpha} \mathcal{D}_{\phi_t}(\boldsymbol{\lambda}, \boldsymbol{\lambda}_t)\right\}$
\vspace{0.6mm}
\EndFor\\
\vspace{0.6mm}
\Return $\boldsymbol{\lambda}_{T+1},\boldsymbol{\beta}_{T+1}$
\end{algorithmic}
\end{algorithm}

\section{Bilevel Local Regret Minimization}
\subsection{\textbf{OBBO} Regret}
Our objective is to show that the deterministic OBBO algorithm of Algorithm \ref{alg:deterministic_online_bregman} achieves a sublinear rate of bilevel local regret. We first provide a lemma that introduces a novel decomposition of the hypergradient error  incurred by OBBO. Namely the hypergradient error is the cost of using the approximate gradient $\widetilde{\nabla} f_t(\boldsymbol{\lambda}_t,\boldsymbol{\omega}^K_t)$ evaluated at $\boldsymbol{\lambda}_t\in\mathcal{X}$ and $\boldsymbol{\beta}_{t+1}=\boldsymbol{\omega}^K_t \in \mathbb{R}^{d_2}$ instead of optimal $\widehat{\boldsymbol{\beta}}_t(\boldsymbol{\lambda})$ and can be separated into four terms. Three terms of our decomposition are discounted variations, parameterized by  decay factor $\nu<1$, of the (i) cumulative time-smoothed hypergradient error; (ii) bilevel local regret; and (iii) cumulative differences of optimal inner-level variables, with the most recent terms in the aforementioned summations with the largest weights. The $\delta_t$ term is composed of an initial error characterizing the original state of the system and a smoothness term from the inner-level objective. Next, we discuss the novelty and practicality of this decomposition.
\begin{lemma}\label{lem:deterministic_hypergradient_error_paper}
    Suppose Assumptions \ref{assump:rel_smoothness}, \ref{assump:strong_convex_g_t}, \ref{assump:bounded_hyperparm}, and \ref{assump:continuity_phi_t}.   Let the inner step size of $\eta < \min{\left(\frac{1}{\ell_{g,1}},\frac{1}{\mu_g}\right)}$ and inner iteration count $K\geq 1$. Then  $\forall t\in[1,T]$  the   hypergradient error from  OBBO  satisfies
    \begin{align}
 \left\|\frac{\partial f_t(\boldsymbol{\lambda}_t,\boldsymbol{\omega}^K_t)}{\partial \boldsymbol{\lambda}}-\nabla F_t(\boldsymbol{\lambda}_t)\right\|^2 \leq \delta_t+A\sum_{j=0}^{t-2}\nu^{j} \left\|\frac{\partial f_{t-1-j,w}(\boldsymbol{\lambda}_{t-1-j},\boldsymbol{\omega}^K_{t-1-j})}{\partial \boldsymbol{\lambda}}-\nabla F_{t-1-j,w}(\boldsymbol{\lambda}_{t-1-j})\right\|^2\nonumber\\+B\sum_{j=0}^{t-2}\nu^{j} \left\|\mathcal{G}_{\mathcal{X}}\left(\boldsymbol{\lambda}_{t-1-j},\nabla F_{t-1-j,w}(\boldsymbol{\lambda}_{t-1-j}),\alpha\right)\right\|^2 +C\sum_{j=0}^{t-2}\nu^{j}\left\|\widehat{\boldsymbol{\beta}}_{t-j}(\boldsymbol{\lambda}_{t-1-j})-\widehat{\boldsymbol{\beta}}_{t-1-j}(\boldsymbol{\lambda}_{t-1-j})\right\|^2,\nonumber
    \end{align}
    where   $F_{t,w}(\boldsymbol{\lambda})\triangleq f_{t,w}\left(\boldsymbol{\lambda},\widehat{\boldsymbol{\beta}}_t(\boldsymbol{\lambda})\right)$ is the time-smoothed  objective, with decay parameter $\nu$, initial error term   $\delta_t=3L^2_3(1-\eta \mu_g)^{2K} +3L_{\boldsymbol{\beta}}\nu^{t-1}\Delta_{\boldsymbol{\beta}}$, and   constants  $A,B,C$   defined in Lemma \ref{lem:deterministic_hypergradient_error}.
\end{lemma}
The hypergradient error decomposition introduced with OBBO is novel, as the upper bound is expressed in terms of the previous bilevel local regret and time-smoothed hypergradient error. In contrast, the expansion in SOBOW, see Theorem 5.6 of \cite{sobow}, is limited to the cumulative difference of outer-level variables. Our decomposition provides an upper bound for the hypergradient error at time $t$ by the cumulative hypergradient error and bilevel local regret up to $t-1$, demonstrating how errors propagate over time. This result, derived from Bregman divergences, is essential for the improved sublinear rate discussed next.
\begin{theorem}\label{thrm:non_convex_regret_paper}
 Suppose Assumptions \ref{assump:rel_smoothness}, \ref{assump:strong_convex_g_t}, \ref{assump:bounded_hyperparm}, \ref{assump:bounded_F_t}, and \ref{assump:continuity_phi_t}. Let inner step size of $\eta < \min{\left(\frac{1}{\ell_{g,1}},\frac{1}{\mu_g}\right)}$, outer step size of $  \alpha\leq \min{\{\frac{3\rho}{4\ell_{F,1}},\frac{\rho\sqrt{(1-\nu)}}{\kappa_g\sqrt{108C_{\mu_g}L_{\boldsymbol{\beta}}}}\}}$, and inner iteration count $K=\frac{\log{(T)}}{\log{\left((1-\eta\mu_g)^{-1}\right)}}+1$. Then the bilevel local regret of our OBBO algorithm  satisfies 
    \begin{align}
            BLR_{w}(T) := \sum_{t=1}^T \left\|\mathcal{G}_{\mathcal{X}}(\boldsymbol{\lambda}_t,\nabla F_{t,w}(\boldsymbol{\lambda}_t),\alpha)\right\|^2\leq O\left(\frac{T}{w}+V_{1,T}+\kappa_g^2H_{2,T}\right)
\end{align}
\end{theorem}
Analogously to OBO (\cite{sobow}), we are interested in sublinear comparator sequences, e.g.,  $V_{1,T}=o(T)$ and $H_{2,T}=o(T)$. It is a weak assumption that still allows for the amount of nonstationarity to grow up to a rate of time itself. Using the above assumption with a properly selected sublinear window size $w=o(T)$ results in the sublinear rate of bilevel local regret presented in Theorem \ref{thrm:non_convex_regret_paper}. 

\subsection{\textbf{SOBBO} Regret}
To study the bilevel local regret in the stochastic framework, we present a lemma that provides a decomposition of the expected hypergradient error of our SOBBO algorithm. The decomposition  includes discounted variations of 1) expected previous bilevel local regret, 2) expected  time-smoothed hypergradient error, 3) expected cumulative differences of optimal inner level variables, and an additional term arising due to the variance $\sigma^2_{g_{\boldsymbol{\beta}}}$ of the inner level objective stochastic gradients. In the following paragraph, we highlight the novelty and practicality of our result.
\begin{lemma}\label{lem:stochastic_cumulative_hypergradient_error_paper}
 Suppose Assumptions \ref{assump:rel_smoothness}, \ref{assump:strong_convex_g_t}, \ref{assump:unbiased_finite_var}, \ref{assump:bounded_hyperparm}, and \ref{assump:continuity_phi_t}. Let the inner step size of $\eta\leq\frac{2}{\ell_{g,1}+\mu_g}$ and inner iteration count $K\geq 1$. Then  $\forall t\in[1,T]$  the expected  hypergradient error from  SOBBO satisfies 
\begin{align}
\mathbb{E}_{\bar{\zeta}_{t,K+1}}\left[\left\|\widetilde{\nabla}f_{t}(\boldsymbol{\lambda}_t,\boldsymbol{\beta}_{t+1})-\nabla F_{t}\left(\boldsymbol{\lambda}_t\right)\right\|^2\right] \leq \delta_t+A\sum_{j=0}^{t-2}\nu^{j+1}\left\|\mathcal{G}_{\mathcal{X}}(\boldsymbol{\lambda}_{t-1-j},\nabla F_{t-1-j,w}(\boldsymbol{\lambda}_{t-1-j}),\alpha)\right\|^2\nonumber\\+B\sum_{j=0}^{t-2}\nu^{j+1}\left\|\widetilde{\nabla}f_{t-1-j,w}(\boldsymbol{\lambda}_{t-1-j},\boldsymbol{\beta}_{t-j},\mathcal{Z}_{t-1-j,w})-\nabla F_{t-1-j,w}(\boldsymbol{\lambda}_{t-1-j})\right\|^2\nonumber\\+C\sum_{j=0}^{t-2}\nu^{j+1}\left[\left\|\widehat{\boldsymbol{\beta}}_{t-j}(\boldsymbol{\lambda}_{t-1-j})-\widehat{\boldsymbol{\beta}}_{t-1-j}(\boldsymbol{\lambda}_{t-1-j})\right\|^2\right]+\frac{D\sigma^2_{g_{\boldsymbol{\beta}}}}{s}.\nonumber 
\end{align}
        where $F_{t,w}(\boldsymbol{\lambda})\triangleq \mathbb{E}_{\epsilon}\left[f_{t,w}\left(\boldsymbol{\lambda},\widehat{\boldsymbol{\beta}}_t(\boldsymbol{\lambda})\right)\right]$ is the time-smoothed  objective  with  decay parameter $\nu$, initial error  $ \delta_t:=3L^2_3(1-\eta \mu_g)^{2K} +3L_{\boldsymbol{\beta}}\nu^{t-1}\Delta_{\boldsymbol{\beta}}$ and constants  $A,B,C,D$  defined in Lemma \ref{lem:stochastic_hypergradient_error}.
\end{lemma}
The decomposition of expected hypergradient error in Lemma \ref{lem:stochastic_cumulative_hypergradient_error_paper} for stochastic OBO problems is, to our knowledge, the first and generalizes from the deterministic setting to finite variances. This decomposition shows that the expected hypergradient error at time $t$ is upper bounded by terms from the deterministic setting and a variance term $\sigma^2_{g_{\boldsymbol{\beta}}}$, motivating the use of variance reduction techniques. The subsequent theorem uses this result to derive a sublinear rate of bilevel local regret.
\begin{theorem}\label{thrm:sobbo_regret_minimization_paper}
  Suppose Assumptions \ref{assump:rel_smoothness}, \ref{assump:strong_convex_g_t}, \ref{assump:unbiased_finite_var}, \ref{assump:bounded_hyperparm}, 
  \ref{assump:bounded_F_t}, and \ref{assump:continuity_phi_t}. Let the inner step size of $\eta\leq\frac{2}{\ell_{g,1}+\mu_g}$, inner iteration count of $K\geq 1$, outer step size of $\alpha\leq \min\{\frac{3\rho}{4\ell_{F,1}},\frac{\rho\sqrt{(1-\nu)}}{\kappa^2_g\sqrt{72C_{\mu_g}}}\} $, and batch sizes of $s=w$ and $ m=\log{(w)}/\log{\left(1-\frac{\mu_g}{\ell_{g,1}}\right)}+1$. Then the  bilevel local regret of  SOBBO satisfies
    \begin{align}\label{eq:sublinear_sobbo}
    BLR_w(T)\leq O\left(\frac{T}{w}\left(1+\kappa^2_g+\sigma^2_f+\kappa^2_g\sigma^2_{g_{\boldsymbol{\beta}}}\right)+V_{1,T}+\kappa^2_gH_{2,T}\right)
\end{align}
\end{theorem} As in the deterministic setting, we consider sublinear comparator sequences, e.g.,  $V_{1,T}=o(T)$ and $H_{2,T}=o(T)$. Given this assumption and properly chosen window and batch size such that  $w=o(T)$ and $s=o(T)$, \eqref{eq:sublinear_sobbo} results in a sublinear rate of bilevel local regret. Compared to the deterministic result, our derived rate in Theorem \ref{thrm:sobbo_regret_minimization_paper} has three additional terms, (i) $\kappa^2_g$ from the bias of the stochastic gradient estimator of Lemma \eqref{lem:stochastic_gradient_estimate_bias}; (ii) $\sigma^2_f$ from the finite variance (Assumption \ref{assump:unbiased_finite_var}) of the stochastic gradients of the outer level objective; and (iii) $\kappa^2_g\sigma^2_{g_{\boldsymbol{\beta}}}$ from the finite variance (Assumption \ref{assump:unbiased_finite_var}) of stochastic gradients of the inner level objective.  We  remark our deterministic rate  is a special case of  the rate of Theorem \ref{thrm:sobbo_regret_minimization_paper} for zero variance, i.e., $\sigma^2_f=0, \sigma^2_{g_{\boldsymbol{\beta}}}=0$.
\section{Experimental Results}
\begin{figure}
  \begin{minipage}[c]{.66\linewidth}
    \centering
\includegraphics[width=.5\textwidth]{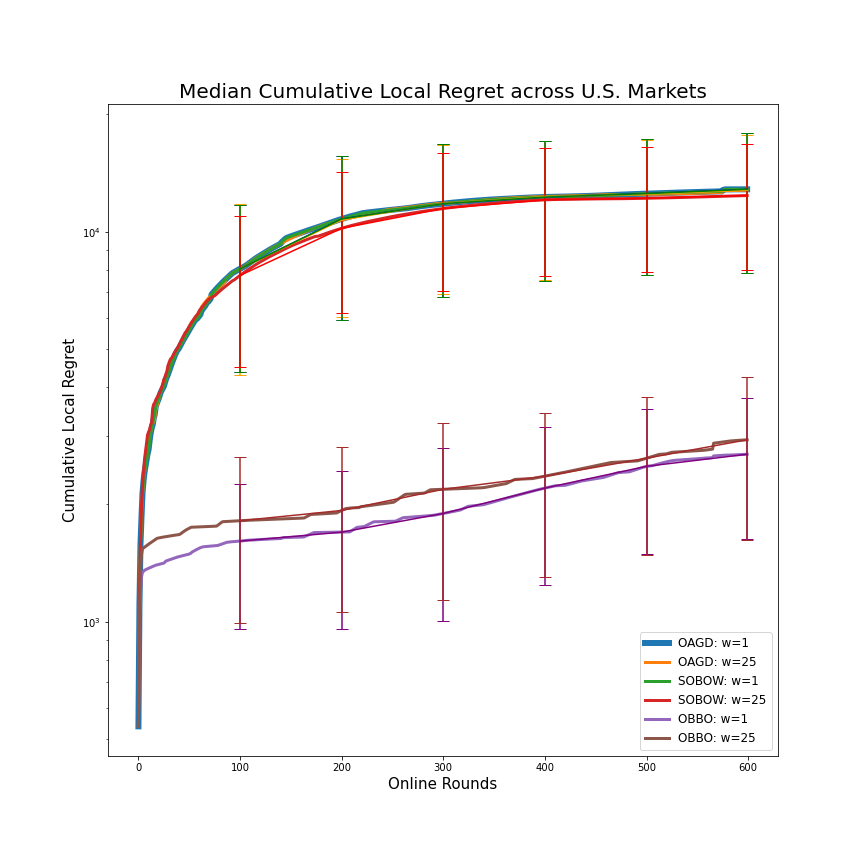}\includegraphics[width=.5\textwidth]{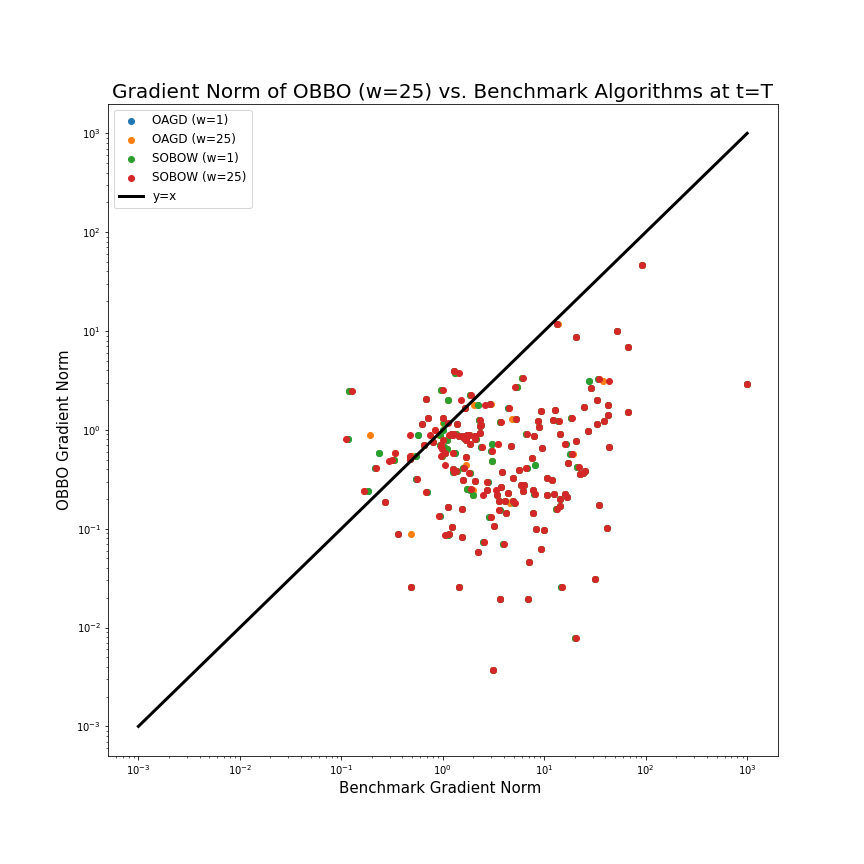}
    \captionof{figure}{\textbf{Left Panel}: Median cumulative local regret of OBBO vs. benchmark algorithms and median deviation bars  plotted every 100 rounds. \textbf{Right Panel}: Gradient norm of OBBO (w=25) vs. benchmark algorithms  at  $t=T$ with $y=x$ line plotted to visualize the  improvement OBBO offers in achieving a solution with smaller gradient norm.}\label{fig:paper_mmn_regret}
  \end{minipage}\hfill
  \begin{minipage}[c]{.3\linewidth}
    \centering
    \begin{tabular}{ p{1cm} p{1cm}p{1.5cm} }
       & Median Loss & Median Deviation  \\
      \hline
       OBBO & \textbf{0.205}   & 0.150  \\
            \hline
      OAGD & 0.265  &  0.209 \\
            \hline
     SOBOW &  0.273 & 0.215  \\
        \hline
    \end{tabular}
    \captionof{table}{Median test loss statistics for U.S. markets (w=25). }\label{tab:table_1}
        \centering
   \begin{tabular}{ p{1cm} p{1cm}p{1.5cm}}
       & Mean Loss & Standard Error  \\
      \hline
       OBBO & \textbf{0.661} &  0.055 \\
            \hline
      OAGD &  0.707  & 0.053  \\
            \hline
      SOBOW & 0.689  &  0.053 \\

        \hline
    \end{tabular}
    \captionof{table}{Mean test loss statistics for U.S. markets (w=25).  }\label{tab:table_2}
  \end{minipage}
\end{figure}

We conduct two experiments to demonstrate the superior performance and efficiency of our algorithms relative to the online bilevel benchmarks of OAGD (\cite{tarzanagh2024online}) and SOBOW (\cite{sobow}). For our algorithms, we choose the  reference function of  $\phi_t(\boldsymbol{\lambda})=\frac{1}{2}\boldsymbol{\lambda}^T\mathbf{H}_t\boldsymbol{\lambda}$ such that $\mathbf{H}_t$ is an adaptive matrix of averaged gradient squares with coefficient 0.9, commonly applied in prior works (\cite{huang2022enhanced},\cite{huang2021super}). Further details and results for both experiments are available in Appendix  \ref{sec:appendix_D}.

\textbf{Online Hyperparameter Optimization:} Hyperparameter optimization has often been formulated as a bilevel optimization as in \cite{pedregosa2016hyperparameter} and \cite{lorraine2020optimizing}. In hyperparameter optimization, the goal is to find optimal hyperparameter values on a validation dataset for optimal parameter values on a training dataset. Specifically, we consider an \emph{online} hyperparameter optimization where the underlying data distribution can vary across time. Compared to the offline case, an online framework captures a larger class of hyperparameter optimization problems (e.g., nonstationarity in optimal hyperparameters).

In online hyperparameter optimization, at each time $t$, new data samples split into a training and validation set, that is $D_t:=\{D_t^{tr},D_t^{val}\}$, arrive from a potentially new distribution. The inner objective of the online hyperparameter optimization is a regularized training loss on $D_t^{tr}$ of the form  $\sum_{\boldsymbol{x} \in D^{tr}_t}L(\boldsymbol{\beta},\boldsymbol{x})+\Omega(\boldsymbol{\lambda},\boldsymbol{\beta})$ for a loss function $L(\boldsymbol{\beta},\boldsymbol{x})$ evaluated across samples $\boldsymbol{x}\in D_t^{tr}$ for parameters $\boldsymbol{\beta}$ and the regularization function of $\Omega(\boldsymbol{\lambda},\boldsymbol{\beta})$. Given the optimal parameters $\widehat{\boldsymbol{\beta}}_t(\boldsymbol{\lambda})$ from the inner optimization, the outer objective is the validation loss on $D_t^{val}$ of the form $\sum_{\boldsymbol{x} \in D^{val}}L(\widehat{\boldsymbol{\beta}}_t(\boldsymbol{\lambda}),\boldsymbol{x})$.

We perform online regression on a Market Impact dataset comprising time series data from the components of the S\&P 500 index. Each time series corresponds to periods annotated by experts as significant market impact events. Our underlying model is a smoothing spline of linear order as in \cite{wahba1978},  where the inner level variables are  B-spline coefficients  and the outer level variable is a positive regularization hyperparameter, $\lambda \in \mathbb{R}^+$, respectively fitted on the training and validation datasets. The simplicity of such a model allows us to utilize closed-form hypergradients, obviating the need for an inner gradient descent loop. For all algorithms and window configurations, the outer learning rate is  set at $\alpha=0.001$. New data samples from the corresponding time series arrive to each model in sequential order in a batch size of 1. Further details can be found in Appendix \ref{sec:appendix_D}. 

The left panel of Figure \ref{fig:paper_mmn_regret} demonstrates significant improvements in terms of the  median cumulative local regret achieved by OBBO relative to  benchmark algorithms of OAGD and SOBOW across samples from the dataset (averaged across 3 random seeds). Further, the right panel of Figure \ref{fig:paper_mmn_regret}, shows that  OBBO achieves a smaller gradient norm at $t=T$ relative to benchmarks across individual market events. The descriptive statistics of forecasting loss aggregated across samples from the dataset are in Table \ref{tab:table_1} and \ref{tab:table_2} and highlight the improvement achieved by OBBO relative to the benchmarks on the test data in terms of median- and mean- of the mean squared error, respectively. 

\begin{figure}[]
  \begin{minipage}[c]{\linewidth}
    \centering
\includegraphics[width=.33\textwidth,,height=47mm]{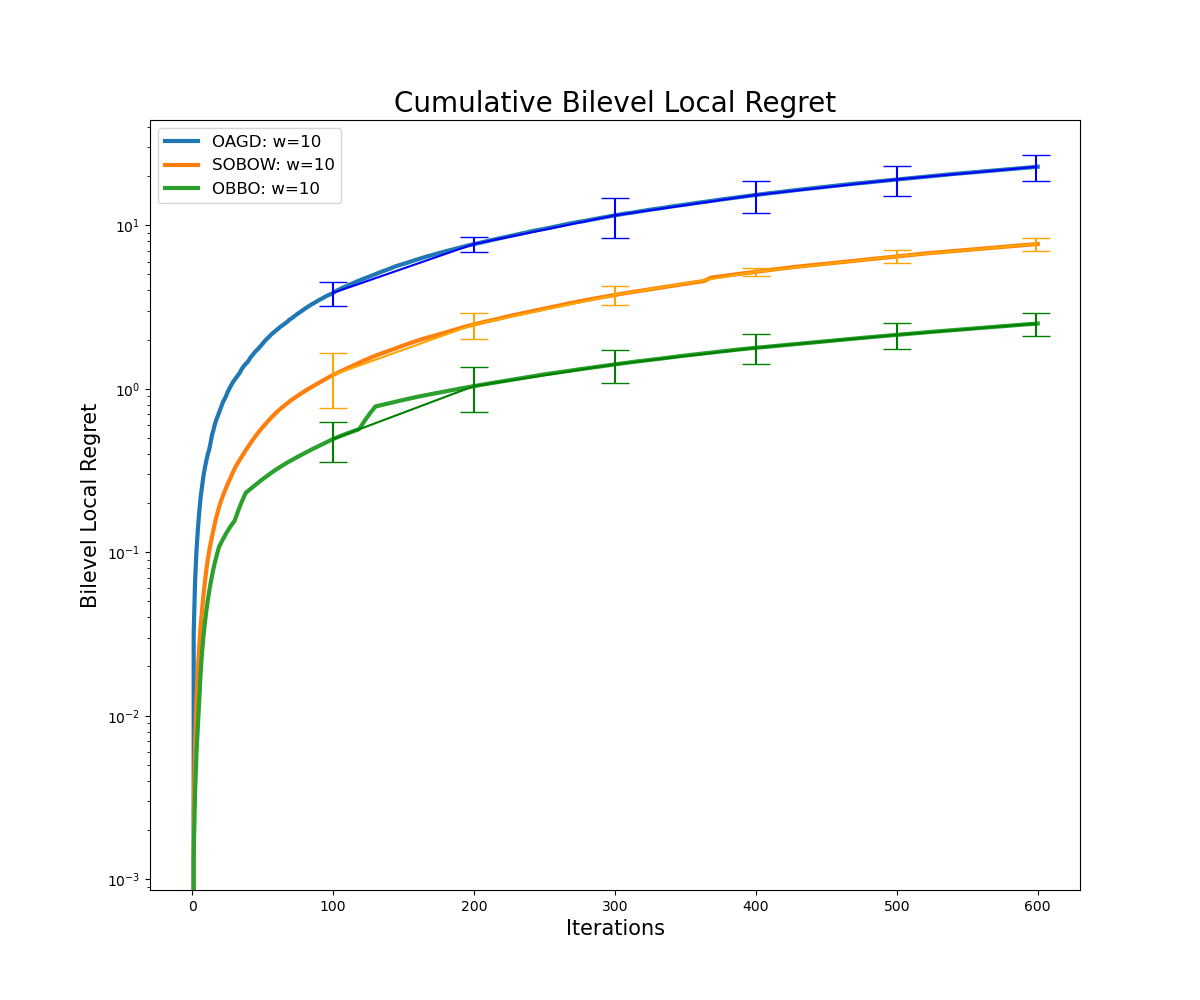}\includegraphics[width=.33\textwidth,height=45mm]{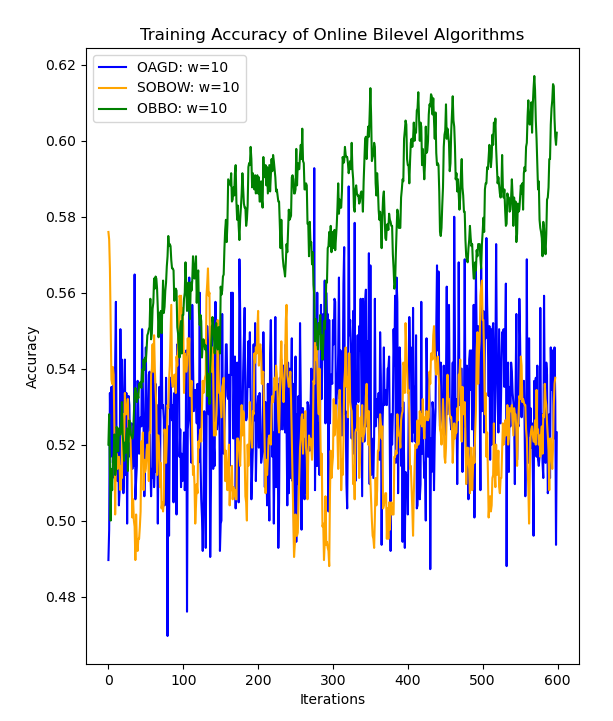}\includegraphics[width=.33\textwidth,height=45mm]{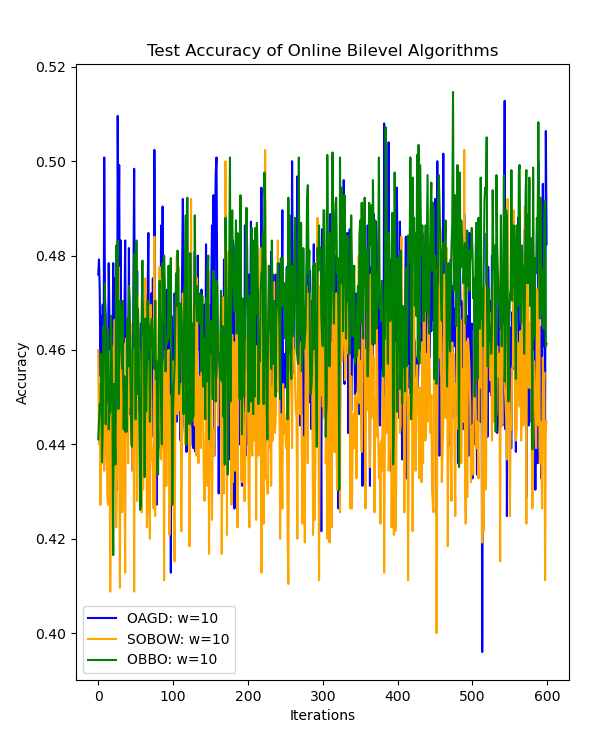} \captionof{figure}{Online meta-learning task on FC100 dataset. \textbf{Left Panel}: Improvement with OBBO on  cumulative bilevel local regret.  \textbf{Middle  Panel}: Higher training accuracy  with OBBO. \textbf{Right Panel}: Test accuracy: OBBO outperforms SOBOW while achieving OAGD performance with 10x ($w=10$) computationally cheaper update.}\label{fig:paper_meta_learning}
\end{minipage}
  \end{figure}

\textbf{Online Meta-Learning:} Meta-learning is frequently formulated as a bilevel optimization problem, see, e.g., \cite{finn2017model} and \cite{rajeswaran2019meta}. The objective is to learn optimal meta-parameters in the outer optimization, which, after adaptation—typically via a gradient descent step—yield optimal parameters for the inner optimization. We consider the online meta-learning framework proposed in \cite{finn2019online}, which extends the traditional offline formulation to accommodate non-stationary task distributions in the online setting.

The online meta-learning problem is setup where at each time $t$, a new task $D_t$ composed of training and validation samples, that is $D_t:\{D_t^{tr},D_t^{val}\}$, arrives from a potentially nonstationary distribution. The goal of the learner in the inner objective is to learn task specific parameters $\widehat{\boldsymbol{\beta}}_t(\boldsymbol{\lambda}) $ for task $D_t$ through the minimization of  $\sum_{\boldsymbol{x} \in D_t^{tr}} L(\boldsymbol{\beta},\boldsymbol{x}) +\frac{\gamma}{2} \left\|\boldsymbol{\lambda}-\boldsymbol{\beta}\right\|^2$ for $\boldsymbol{\beta}$. Note the learner requires a loss function $L(\boldsymbol{\beta},\boldsymbol{x})$, fixed meta-learned parameters $\boldsymbol{\lambda}$, and regularization constant $\gamma$ that can also be learned. The meta-learner aims to learn optimal meta-parameters in the outer objective, that is $\sum_{\boldsymbol{x} \in D_t^{val}} L(\widehat{\boldsymbol{\beta}}_t(\boldsymbol{\lambda}),\boldsymbol{x})$, such that after task adaptation the task-specific parameters will be optimal.  

To test the performance and efficiency of OBBO in online meta-learning, we consider a 5-way 5-shot classification task on the FC100 dataset from \cite{oreshkin2018tadam}. Following the online bilevel optimization work of \cite{tarzanagh2024online}, we consider a 4-layer convolutional neural network as the underlying model with all other experimental details and hyperparameters configured as in \cite{tarzanagh2024online}. In particular, at each time $t$, a new task composed of 25 training and 25 validation samples arrives for which the model respectively undergoes task-specific updates and evaluation on. Inner and outer learning rates are respectively $\eta=0.1$ and $\alpha=1e-4$ for all benchmark algorithms with further remarks in Appendix \ref{sec:appendix_D}. 

In the left panel of Figure \ref{fig:paper_meta_learning}, note the improvement in cumulative bilevel local regret with OBBO relative to benchmark algorithms of OAGD and SOBOW across samples from the FC100 dataset (averaged across 5 random seeds). In the middle  panel of Figure \ref{fig:paper_meta_learning}, we report higher training accuracy achieved with OBBO. In the right panel of Figure \ref{fig:paper_meta_learning}, OBBO outperforms  test accuracy relative to SOBOW while achieving OAGD performance with a 10x ($w=10$) computationally cheaper update.

\section{Conclusions}
 In this work, we develop a novel  family of algorithms parameterized by Bregman divergences. We prove that  in addition to achieving  sublinear bilevel local regret,  our algorithms  offer improvements in convergence rates by adapting to the underlying geometry of the problem. Furthermore, we  introduce the first stochastic online bilevel optimizer and show that by utilizing a weighted average of recent stochastic approximated hypergradients, our algorithm achieves a  sublinear rate of bilevel local regret in a sample efficient manner. Empirical results on machine learning tasks such as online hyperparameter optimization and online meta-learning validate our theoretical contributions and the practicality of our proposed methods. We remark that research in online bilevel optimization is  underdeveloped relative to its offline counterpart. Many open problems remain such as the existence of lower bounds for bilevel local regret in the deterministic and stochastic setting. As a comparison, we note a few areas of interest in offline bilevel optimization that have not been explored in the online setting; that is the design of adaptive-momentum based bilevel algorithms in \cite{huang2021biadam} as well as the construction of Hessian free bilevel algorithms in \cite{sow2022convergence}. 

\bibliographystyle{abbrv}
\bibliography{ref}

% APPENDIX
\newpage
\appendix

\section{Notation and Preliminaries}\label{sec:appendix_A}
\begin{table}[h]
    \centering
    \begin{tabular}{||l|l||}
    \hline
           \textbf{Notation} & \textbf{Description}\\
         \hline
         $t$ & Time index \\
         $w$ & Window size \\
         $g_t$ & Inner level objective at $t$\\
         $f_t$ & Outer level objective at $t$ \\
         $F_t$ & Reparameterized  outer level objective  at $t$\\
         $h$ & Convex and potentially nonsmooth regularization term \\
         $\boldsymbol{\beta}_t$ & Inner level variable at time $t$  \\
         $\boldsymbol{\lambda}_t$ & Outer level variable at time $t$\\
         $\epsilon$ & Error random variable for $f_t$\\
         $\zeta$ & Error random variable for $g_t$\\
         $\left\|\cdot\right\|$ & The $\ell_2$ norm for vectors (spectral norm for matrices) \\
         $\mathcal{X}$ & Decision set for outer level variable  \\
        $S$ & Diameter of $\mathcal{X}$: $S=\max_{\boldsymbol{\lambda},\boldsymbol{\lambda}'\in \mathcal{X}} \left\|\boldsymbol{\lambda}-\boldsymbol{\lambda}'\right\|$ \\
        $Q$ & Upper bound on $F_t$: $\sup_{\boldsymbol{\lambda}\in \mathcal{X}}\left|F_t(\boldsymbol{\lambda})\right|\leq Q$\\ 
        $\nabla F_t(\boldsymbol{\lambda})$ & Gradient of $F_t$ w.r.t. $\boldsymbol{\lambda}$, i.e., the  hypergradient \\
        $ \widetilde{\nabla} f_t(\boldsymbol{\lambda},\boldsymbol{\beta})$ & Gradient estimate for $\nabla F_t(\boldsymbol{\lambda})$ given any $\boldsymbol{\lambda} \in \mathcal{X}$ and $\boldsymbol{\beta} \in \mathbb{R}^{d_2}$  \\
        $ \widetilde{\nabla}f_{t,w}(\boldsymbol{\lambda}_t,\boldsymbol{\beta_{t+1}})$ & Time-smoothed $ \widetilde{\nabla} f_t(\boldsymbol{\lambda}_t,\boldsymbol{\beta}_{t+1})$ across a window of size $w$  \\
        $s$ &  Batch size for $\nabla_{\boldsymbol{\beta}} g_t(\cdot,\cdot,\bar{\zeta})$, that is $\left|\bar{\zeta}\right|=s$  \\
        $m$  & Batch size for $\widetilde{\nabla} f_{t}(\cdot,\cdot,\mathcal{E} ) $, that is $\left|\mathcal{E}\right|=m$  \\
       $\alpha$   & Outer step size\\
        $\eta$   & Inner step size\\
        $K$ & Inner iteration count\\
        $ \ell_{f,0}$ & Lipschitz constant for $f_t$ \\
        $\ell_{f,1}$ & Lipschitz constant for $\nabla f_t$ \\
        $\ell_{F,1}$ & Lipschitz constant for $\nabla F_t$ \\
        $\ell_{g,1}$ &  Lipschitz constant for $\nabla_{\boldsymbol{\beta}}g_t$\\
        $\ell_{g,2}$ & Lipschitz constant for $\nabla^2_{\boldsymbol{\beta},\boldsymbol{\beta}}g_t$ \\
        $\mu_g$& Strong convexity parameter of $g_t$ \\
        $\kappa_g$   & Condition number of $g_t$: $\kappa_g=\ell_{g,1}/\mu_g$ \\
        $\phi(\boldsymbol{\lambda})$  & Continuously differentiable and strongly-convex distance generating function\\
        $D_{\phi}(\cdot,\cdot)$ & Bregman Divergence defined by  $\phi_t(\boldsymbol{\lambda})$\\
        $\rho$& Strong convexity parameter of $\phi_t(\boldsymbol{\lambda})$ \\
        $L_{1}$ & First hypergradient  error constant: $L_1:=\kappa_g(\ell_{g,1}+\mu_g)$ \\
        $L_{2}$ & Second hypergradient  error constant: $L_2:=\frac{2\ell_{f,0}\ell_{g,2}}{\mu_g}(1+\kappa_g)$ \\
        $L_{3}$ & Third hypergradient  error constant: $L_3:=\ell_{f,0}\kappa_g$ \\
        $L_{\boldsymbol{\beta}}$ & Total hypergradient error constant w.r.t. $\boldsymbol{\beta}_t$: $ L_{\boldsymbol{\beta}}:=L^2_1(1-\eta \mu_g)^{K}+L^2_2(1-\eta \mu_g)^{K-1}$ \\
        $ C_{\mu_g}$ & Inner level variable error constant: $ C_{\mu_g}>1$  \\
        $ C_{K}$ & Inner level variable error variance constant: $ C_{K}>0$\\
        $\nu$ & Decay parameter: $0<\nu<1$\\ $\Delta_{\boldsymbol{\beta}}$& Initial error term of inner level variables \\
        $\sigma^2_{g_{\boldsymbol{\beta}}}$ & Finite variance of $\nabla_{\boldsymbol{\beta}}g_t(\boldsymbol{\lambda},\boldsymbol{\beta},\zeta)$\\
        $\sigma^2_{f}$ & Finite variance of $\widetilde{\nabla}f_t(\boldsymbol{\lambda},\boldsymbol{\beta},\mathcal{E})$\\
        $\mathcal{G}_{\mathcal{X}}(\boldsymbol{u},\boldsymbol{q},\alpha)$ & Generalized projection $\forall \boldsymbol{u} \in \mathcal{X}, \boldsymbol{q}\in \mathbb{R}^{d_1}$, and $\alpha>0$\\
        $V_{p,t}$ & Outer level  evaluation variation:$\sum_{t=1}^{T} \sup_{\boldsymbol{\lambda} \in \mathcal{X}}\left|F_{t+1}\left(\boldsymbol{\lambda}\right)-F_{t}\left(\boldsymbol{\lambda}\right)\right|^p$ 
        \\
        $H_{p,t}$ & Inner level path variation: $\sum_{t=2}^T\sup_{\boldsymbol{\lambda}\in \mathcal{X}}\|\widehat{\boldsymbol{\beta}}_{t-1}(\boldsymbol{\lambda})-\widehat{\boldsymbol{\beta}}_{t}(\boldsymbol{\lambda})\|^p$ \\
        $BLR_{w}(T)$ & Bilevel local regret for window size $w\geq 1$ \\
         \hline
    \end{tabular}
    \caption{Summary of Notation}
    \label{tab:my_label}
\end{table}
\newpage
The next two lemmas introduce known results for smooth and strongly convex functions.
\begin{lemma}\label{lem:smoothness} 
Suppose a function $g(\boldsymbol{\beta}):\mathbb{R}^{d_2}\mapsto\mathbb{R}$ is $\ell$-Lipschitz continuous with respect to $\left\|\cdot\right\|$. Then the following inequality holds $\forall \boldsymbol{\beta}_1,\boldsymbol{\beta}_2\in\mathbb{R}^{d_2}$
\begin{align}
    \left|g(\boldsymbol{\beta}_2)-g(\boldsymbol{\beta}_1)\right|\leq \ell\left\|\boldsymbol{\beta}_2-\boldsymbol{\beta}_1\right\|
\end{align}
\end{lemma}

\begin{lemma}\label{lem:strong_convexity}
A function $g(\boldsymbol{\beta})$ that is $\mu_g$ strongly convex with respect to $\left\|\cdot\right\|$  satisfies the  inequality $\forall \boldsymbol{\beta}_1,\boldsymbol{\beta}_2 \in \mathbb{R}^{d_2}$ of
\begin{align}
    g(\boldsymbol{\beta}_2)-g(\boldsymbol{\beta}_1) \geq \left\langle \boldsymbol{\beta}_2-\boldsymbol{\beta}_1,\nabla g(\boldsymbol{\beta}_1)\right\rangle +\frac{\mu_g}{2}\left\|\boldsymbol{\beta}_2-\boldsymbol{\beta}_1\right\|^2
\end{align}
\end{lemma}

\begin{lemma}\label{lem:vector_ub} (Lemma 12 in \cite{tarzanagh2024online}) For any set of vectors $\left\{\boldsymbol{\beta}_i\right\}_{i=1}^m$, it holds that
\begin{align}
    \left\|\sum_{i=1}^m\boldsymbol{\beta}_i\right\|^2\leq m\sum_{i=1}^m\left\|\boldsymbol{\beta}_i\right\|^2
\end{align}

\end{lemma}

The following lemma provides progress bounds for gradient descent applied to a $\mu_g$-strongly convex and twice differentiable function $g(\boldsymbol{\beta})$.

\begin{lemma}\label{lem:alt_grad_bound}
  Let $g(\boldsymbol{\omega})$ be a  twice differentiable and $\mu_g$-strongly convex function with  $\nabla g(\boldsymbol{\omega})$ satisfying $\ell_{g,1}$-Lipschitz continuity. Further assume $g(\boldsymbol{\omega})$ has a global minimizer $\widehat{\boldsymbol{\omega}}$ over the domain $\mathbb{R}^{d_2}$. Then under the gradient descent method of   
   \begin{align}
\boldsymbol{\omega}^{k}=\boldsymbol{\omega}^{k-1}-\eta\nabla g(\boldsymbol{\omega}^{k-1}),\nonumber
    \end{align}
    the following satisfies for $\eta\leq \frac{1}{\ell_{g,1}}$
    \begin{align}
        \left\|\boldsymbol{\omega}^{k}-\widehat{\boldsymbol{\omega}}\right\|^2\leq \left(1-\eta\mu_g\right)\left\|\boldsymbol{\omega}^{k-1}-\widehat{\boldsymbol{\omega}}\right\|^2,\nonumber
    \end{align}
\end{lemma}

The following two lemmas characterize useful properties of  the generalized projection $\mathcal{G}_{\mathcal{X}}(\boldsymbol{u},\boldsymbol{q},\alpha) $.

\begin{lemma}[Lemma 1 in \cite{ghadimi2016mini}]\label{lem:gen_projection_bound_one}  Let $\boldsymbol{\lambda}^+$ be from \eqref{eq:outer_step_minimization}. Then $\forall \boldsymbol{u}\in \mathcal{X}$, $\boldsymbol{q}\in\mathbb{R}^{d_1}$, and $\alpha>0$ we have
    \begin{align}
        \left\langle \boldsymbol{q},\mathcal{G}_{\mathcal{X}}(\boldsymbol{u},\boldsymbol{q},\alpha)\right\rangle \geq \rho\left\|\mathcal{G}_{\mathcal{X}}(\boldsymbol{u},\boldsymbol{q},\alpha)\right\|^2 +\frac{1}{\alpha}\left(h(\boldsymbol{\lambda}^+)-h(\boldsymbol{u})\right)
    \end{align}
    such that $\rho>0$ is the strong convexity parameter of the distance generating function $\phi(\boldsymbol{\lambda})$. 
    \end{lemma}
    \begin{lemma}[Proposition 1 in \cite{ghadimi2016mini}]\label{lem:gen_projection_bound_two} 
    Let $\mathcal{G}_{\mathcal{X}}(\boldsymbol{u},\boldsymbol{q},\alpha)$ be defined in \eqref{defn:generalized_gradient}. Then $\forall \boldsymbol{q}_1,\boldsymbol{q}_2\in \mathbb{R}^{d_1}$, $\forall \boldsymbol{u}\in\mathcal{X}$, $\forall \alpha>0$, we have
    \begin{align}
        \left\| \mathcal{G}_{\mathcal{X}}(\boldsymbol{u},\boldsymbol{q}_1,\alpha)-\mathcal{G}_{\mathcal{X}}(\boldsymbol{u},\boldsymbol{q}_2,\alpha)\right\| \leq \frac{1}{\rho}\left\|\boldsymbol{q}_1-\boldsymbol{q}_2\right\|. 
    \end{align}
\end{lemma}

The next Lemma provides useful bounds on the hypergradient $\nabla F_t(\boldsymbol{\lambda})$, gradient estimate $\nabla f_t(\boldsymbol{\lambda},\boldsymbol{\beta})$, and optimal inner level variables $\widehat{\boldsymbol{\beta}}_t(\boldsymbol{\lambda})$ in the deterministic online bilevel optimization problem.
\begin{lemma}\label{lem:hypergrad_bound} (Lemma 3 in \cite{tarzanagh2024online})
Under assumptions \ref{assump:rel_smoothness} and \ref{assump:strong_convex_g_t}, it holds for all $t\in[1,T]$,  $\boldsymbol{\lambda}_1,\boldsymbol{\lambda}_2\in \mathcal{X} $, and $ \boldsymbol{\beta} \in \mathbb{R}^{d_2}$ that 
   \begin{align}
\left\|\widehat{\boldsymbol{\beta}}_t(\boldsymbol{\lambda}_1)-\widehat{\boldsymbol{\beta}}_t(\boldsymbol{\lambda}_2)\right\|\leq \kappa_g\left\|\boldsymbol{\lambda}_1-\boldsymbol{\lambda}_2\right\|,
    \end{align}
     where $\kappa_g:=\frac{\ell_{g,1}}{\mu_g}=O(\kappa_g)$,   the gradient estimator $\widetilde{\nabla}f_t(\boldsymbol{\lambda},\boldsymbol{\beta})$   satisfies
    \begin{align} 
\|\widetilde{\nabla}f_t(\boldsymbol{\lambda},\boldsymbol{\beta})-\nabla F_t(\boldsymbol{\lambda})\|\leq M_f\left\|\boldsymbol{\beta}-\widehat{\boldsymbol{\beta}}_t(\boldsymbol{\lambda})\right\|,
    \end{align}
    where $M_f:=\ell_{f,1}+\ell_{f,1}\kappa_g+\frac{\ell_{f,0}\ell_{g,2}}{\mu_g}\left(1+\kappa_g\right)=O(\kappa_g^2)$, and 
    \begin{align}
        \left\|\nabla F_t(\boldsymbol{\lambda}_1)-\nabla F_t(\boldsymbol{\lambda}_2)\right\|\leq \ell_{F,1}\left\|\boldsymbol{\lambda}_1-\boldsymbol{\lambda}_2\right\|,
    \end{align}
     where $\ell_{F,1}:= \ell_{f,1}(1+\kappa_g)+\frac{\ell_{f,0}\ell_{g,2}}{\mu_g}(1+\kappa_g)+M_f\kappa_g=O(\kappa_g^3)$.
     \end{lemma}

Lemma \ref{lem:itd_hypergradient_bound} provides an upper bound on the hypergradient error when utilizing an iterative differentiation approach for estimation.

\begin{lemma}\label{lem:itd_hypergradient_bound} (Lemma 6 in \cite{ji2021bilevel})
    Suppose Assumptions \ref{assump:rel_smoothness} and \ref{assump:strong_convex_g_t} are satisfied with $\eta< \frac{1}{\ell_{g,1}}$ and $K\geq1$. Then we have  $\forall t \in [1,T]$
    \begin{align}
        \left\|\frac{\partial f_t(\boldsymbol{\lambda},\boldsymbol{\omega}^K_t)}{\partial \boldsymbol{\lambda}}-\nabla F_t(\boldsymbol{\lambda})\right\| \nonumber\\\leq \left(L_1(1-\eta \mu_g)^{\frac{K}{2}}+L_2(1-\eta \mu_g)^{\frac{K-1}{2}}\right)\left\|\boldsymbol{\beta}_t-\widehat{\boldsymbol{\beta}}_t(\boldsymbol{\lambda})\right\|+L_3(1-\eta \mu_g)^K ,
    \end{align}
    where $L_1=\kappa_g(\ell_{g,1}+\mu_g)$, $L_2=\frac{2\ell_{f,0}\ell_{g,2}}{\mu_g} (1+\kappa_g)$, and $L_3=\ell_{f,0}\kappa_g$.
    \end{lemma}
  
The next  Lemmas characterize the bias  of the stochastic hypergradient estimate $\widetilde{\nabla}f_t(\boldsymbol{\lambda}_t,\boldsymbol{\beta}_{t+1},\mathcal{E}_t)$.
\begin{lemma}[Lemma 2.1 in \cite{khanduri2021near}]\label{lem:stochastic_gradient_estimate_bias} Suppose Assumptions \ref{assump:rel_smoothness},\ref{assump:strong_convex_g_t}, and \ref{assump:bounded_F_t}. For any $m\geq 1$  the gradient estimator of \eqref{eq:stochastic_gradient_sample} satisfies
    \begin{align}
\left\|\widetilde{\nabla}f_t(\boldsymbol{\lambda}_t,\boldsymbol{\beta}_{t+1})-\mathbb{E}_{\mathcal{E}_t}\left[\widetilde{\nabla}f_t(\boldsymbol{\lambda}_t,\boldsymbol{\beta}_{t+1},\mathcal{E}_t)\right]\right\| \leq\ell_{f,1}\kappa_g\left(1-\frac{\mu_g}{\ell_{g,1}}\right)^m
    \end{align}
\end{lemma}

\section{Proof of Main Results}\label{sec:appendix_B}
\subsection{Deterministic Setting}
First, we introduce some required lemmas. The following Lemma provides an upper bound on the cumulative difference between the time-smoothed outer level objective $F_{t,w}(\boldsymbol{\lambda})$ evaluated at  $\boldsymbol{\lambda}_t$ and $\boldsymbol{\lambda}_{t+1}$ in terms of the outer level objective upper bound $Q$,  window size $w$, and the comparator sequence on subsequent function evaluations $V_{1,T}$.
\begin{lemma}\label{lem:bounded_time_varying} Suppose Assumption \ref{assump:bounded_F_t}. If our OBBO algorithm in Algorithm \ref{alg:deterministic_online_bregman} is applied with window size $w\geq 1$ to generate the sequence $\{\boldsymbol{\lambda}_t\}_{t=1}^T$, then we have 
\begin{align}
\sum_{t=1}^T\left(F_{t,w}(\boldsymbol{\lambda}_t)-F_{t,w}(\boldsymbol{\lambda}_{t+1})\right) \leq \frac{2TQ}{w} +V_{1,T}.\nonumber
\end{align}
where $V_{1,T}:=\sum_{t=1}^{T} \sup_{\boldsymbol{\lambda} \in \mathcal{X}}\left[F_{t+1}\left(\boldsymbol{\lambda}\right)-F_{t}\left(\boldsymbol{\lambda}\right)\right] $
\end{lemma}\begin{proof}[Proof of Lemma \ref{lem:bounded_time_varying}]
By definition, in the deterministic setting, we have $F_t(\boldsymbol{\lambda})\triangleq f_t\left(\boldsymbol{\lambda},\widehat{\boldsymbol{\beta}}_t(\boldsymbol{\lambda})\right)$. Then it holds
\begin{align}
\sum_{t=1}^T\left(F_{t,w}(\boldsymbol{\lambda}_t)-F_{t,w}(\boldsymbol{\lambda}_{t+1})\right)=\sum_{t=1}^T\frac{1}{w}\sum_{i=0}^{w-1}\left(f_{t-i}\left(\boldsymbol{\lambda}_{t-i},\widehat{\boldsymbol{\beta}}_{t-i}(\boldsymbol{\lambda}_{t-i})\right)-f_{t-i}\left(\boldsymbol{\lambda}_{t+1-i},\widehat{\boldsymbol{\beta}}_{t-i}(\boldsymbol{\lambda}_{t+1-i})\right)\right)\nonumber
\end{align}
Which is equivalent to
\begin{align}
    \sum_{t=1}^T\frac{1}{w}\sum_{i=0}^{w-1}\left(f_{t-i}\left(\boldsymbol{\lambda}_{t-i},\widehat{\boldsymbol{\beta}}_{t-i}(\boldsymbol{\lambda}_{t-i})\right)-f_{t-i}\left(\boldsymbol{\lambda}_{t+1-i},\widehat{\boldsymbol{\beta}}_{t-i}(\boldsymbol{\lambda}_{t+1-i})\right)\right) \nonumber\\ \label{eq:first_term_rhs_function_eval}=  \sum_{t=1}^T\frac{1}{w}\sum_{i=0}^{w-1}\left(f_{t-i}\left(\boldsymbol{\lambda}_{t-i},\widehat{\boldsymbol{\beta}}_{t-i}(\boldsymbol{\lambda}_{t-i})\right)-f_{t+1-i}\left(\boldsymbol{\lambda}_{t+1-i},\widehat{\boldsymbol{\beta}}_{t+1-i}(\boldsymbol{\lambda}_{t+1-i})\right)\right) \\ \label{eq:second_term_rhs_function_eval}+  \sum_{t=1}^T\frac{1}{w}\sum_{i=0}^{w-1}\left(f_{t+1-i}\left(\boldsymbol{\lambda}_{t+1-i},\widehat{\boldsymbol{\beta}}_{t+1-i}(\boldsymbol{\lambda}_{t+1-i})\right)-f_{t-i}\left(\boldsymbol{\lambda}_{t+1-i},\widehat{\boldsymbol{\beta}}_{t-i}(\boldsymbol{\lambda}_{t+1-i})\right)\right) 
\end{align}
For \eqref{eq:first_term_rhs_function_eval}, we can write
\begin{align}\label{eq:ub_first_term_rhs_function_eval}
   \frac{1}{w}\sum_{i=0}^{w-1}\left(f_{t-i}\left(\boldsymbol{\lambda}_{t-i},\widehat{\boldsymbol{\beta}}_{t-i}(\boldsymbol{\lambda}_{t-i})\right)-f_{t+1-i}\left(\boldsymbol{\lambda}_{t+1-i},\widehat{\boldsymbol{\beta}}_{t+1-i}(\boldsymbol{\lambda}_{t+1-i})\right)\right)\nonumber\\=\frac{1}{w}\left[f_{t}\left(\boldsymbol{\lambda}_{t},\widehat{\boldsymbol{\beta}}_{t}(\boldsymbol{\lambda}_{t})\right)+\ldots+f_{t+1-w}\left(\boldsymbol{\lambda}_{t+1-w},\widehat{\boldsymbol{\beta}}_{t+1-w}(\boldsymbol{\lambda}_{t+1-w})\right)\right]\nonumber\\-\frac{1}{w}\left[f_{t+1}\left(\boldsymbol{\lambda}_{t+1},\widehat{\boldsymbol{\beta}}_{t+1}(\boldsymbol{\lambda}_{t+1})\right)+\ldots+f_{t+2-w}\left(\boldsymbol{\lambda}_{t+2-w},\widehat{\boldsymbol{\beta}}_{t+2-w}(\boldsymbol{\lambda}_{t+2-w})\right)\right]\nonumber\\=\frac{1}{w}\left[f_{t+1-w}\left(\boldsymbol{\lambda}_{t+1-w},\widehat{\boldsymbol{\beta}}_{t+1-w}(\boldsymbol{\lambda}_{t+1-w})\right)-f_{t+1}\left(\boldsymbol{\lambda}_{t+1},\widehat{\boldsymbol{\beta}}_{t+1}(\boldsymbol{\lambda}_{t+1})\right)\right]\nonumber\\=\frac{1}{w}\left(F_{t+1-w}(\boldsymbol{\lambda}_{t+1-w})-F_{t+1}(\boldsymbol{\lambda}_{t+1})\right)\leq \frac{2Q}{w},
\end{align}
where the last inequality comes from Assumption \ref{assump:bounded_F_t}. Note \eqref{eq:second_term_rhs_function_eval} can be bounded through 
\begin{align}\label{ub_second_term_rhs_function_eval}
    \sum_{t=1}^T\frac{1}{w}\sum_{i=0}^{w-1}\left(f_{t+1-i}\left(\boldsymbol{\lambda}_{t+1-i},\widehat{\boldsymbol{\beta}}_{t+1-i}(\boldsymbol{\lambda}_{t+1-i})\right)-f_{t-i}\left(\boldsymbol{\lambda}_{t+1-i},\widehat{\boldsymbol{\beta}}_{t-i}(\boldsymbol{\lambda}_{t+1-i})\right)\right) \nonumber\\ \leq \sum_{t=1}^T\frac{1}{w}\sum_{i=0}^{w-1}\sup_{\boldsymbol{\lambda} \in \mathcal{X}}\left[f_{t+1-i}\left(\boldsymbol{\lambda},\widehat{\boldsymbol{\beta}}_{t+1-i}(\boldsymbol{\lambda})\right)-f_{t-i}\left(\boldsymbol{\lambda},\widehat{\boldsymbol{\beta}}_{t-i}(\boldsymbol{\lambda})\right)\right] \leq V_{1,T}
\end{align}
Combining \eqref{eq:ub_first_term_rhs_function_eval} and \eqref{ub_second_term_rhs_function_eval} results in the upper bound of \begin{align}
\sum_{t=1}^T\left(F_{t,w}(\boldsymbol{\lambda}_t)-F_{t,w}(\boldsymbol{\lambda}_{t+1})\right) \leq \frac{2TQ}{w} +V_{1,T}.\nonumber
\end{align}
\end{proof} 
The next Lemma provides an upper bound on the error of $\left\|\boldsymbol{\beta}_{t}-\widehat{\boldsymbol{\beta}}_t(\boldsymbol{\lambda}_t)\right\|^2$ for all $t\in[1, T]$ in terms of an initial error, the cumulative differences of the outer level variable, and the cumulative differences of the optimal inner level variables.
\begin{lemma}\label{lem:deterministic_inner_tracking_error} 
Suppose Assumptions \ref{assump:rel_smoothness} and \ref{assump:strong_convex_g_t}. Choose the inner step size of $\eta$ and inner iteration count of $K$ to satisfy
\begin{align}
    \eta < \min{\left(\frac{1}{\ell_{g,1}},\frac{1}{\mu_g}\right)},\ \text{and} \quad K\geq1, \nonumber
\end{align}
and define the decay parameter $\nu$, inner level variable error constant $C_{\mu_g}$, and initial error $\Delta_{\boldsymbol{\beta}}$ respectively as 
\begin{align}\label{eq:misc_defns}
    \nu:=\left(1-\frac{\eta\mu_g}{2}\right)(1-\eta\mu_g)^{K-1},\ \text{and} \quad C_{\mu_g}:=\left(1+\frac{2}{\eta\mu_g}\right),\nonumber\\\ \text{and}\quad \Delta_{\boldsymbol{\beta}}:=\left\|\boldsymbol{\beta}_{1}-\widehat{\boldsymbol{\beta}}_1(\boldsymbol{\lambda}_1)\right\|^2. \nonumber
\end{align}
Then
our OBBO algorithm in Algorithm \ref{alg:deterministic_online_bregman} guarantees $\forall t\in [1,T]$
\begin{align}
\left\|\boldsymbol{\beta}_{t}-\widehat{\boldsymbol{\beta}}_t(\boldsymbol{\lambda}_t)\right\|^2 \leq \nu^{t-1}\Delta_{\boldsymbol{\beta}}  \nonumber\\+2C_{\mu_g}\kappa_g^2\sum_{j=0}^{t-2}\nu^{j} \left\|\boldsymbol{\lambda}_{t-1-j}-\boldsymbol{\lambda}_{t-j}\right\|^2 +2C_{\mu_g}\sum_{j=0}^{t-2}\nu^{j}\left\|\widehat{\boldsymbol{\beta}}_{t-j}(\boldsymbol{\lambda}_{t-1-j})-\widehat{\boldsymbol{\beta}}_{t-1-j}(\boldsymbol{\lambda}_{t-1-j})\right\|^2.
\end{align}
\end{lemma}
\begin{proof}[Proof of Lemma \ref{lem:deterministic_inner_tracking_error}]
By definition for $t=1$, we have $\left\|\boldsymbol{\beta}_{1}-\widehat{\boldsymbol{\beta}}_1(\boldsymbol{\lambda}_1)\right\|^2 =\Delta_{\boldsymbol{\beta}}$. Then  $\forall t\in[2,T]$
    \begin{align}\label{eq:deterministic_inner_tracking_error_expansion}
       \left\|\boldsymbol{\beta}_{t}-\widehat{\boldsymbol{\beta}}_t(\boldsymbol{\lambda}_t)\right\|^2=\left\|\boldsymbol{\beta}_t-\widehat{\boldsymbol{\beta}}_{t-1}(\boldsymbol{\lambda}_{t-1})+\widehat{\boldsymbol{\beta}}_{t-1}(\boldsymbol{\lambda}_{t-1})-\widehat{\boldsymbol{\beta}}_t(\boldsymbol{\lambda}_t)\right\|^2,
    \end{align}
    which can be expanded  based on the Young's Inequality for any $\delta>0$ as 
    \begin{align}
      \left\|\boldsymbol{\beta}_t-\widehat{\boldsymbol{\beta}}_{t-1}(\boldsymbol{\lambda}_{t-1})+\widehat{\boldsymbol{\beta}}_{t-1}(\boldsymbol{\lambda}_{t-1})-\widehat{\boldsymbol{\beta}}_t(\boldsymbol{\lambda}_t)\right\|^2\nonumber\\ \leq (1+\delta)      \left\|\boldsymbol{\beta}_{t}-\widehat{\boldsymbol{\beta}}_{t-1}(\boldsymbol{\lambda}_{t-1})\right\|^2\nonumber\\+\left(1+\frac{1}{\delta}\right)\left\|\widehat{\boldsymbol{\beta}}_{t-1}(\boldsymbol{\lambda}_{t-1})-\widehat{\boldsymbol{\beta}}_t(\boldsymbol{\lambda}_t)\right\|^2.\nonumber
    \end{align}
  Now  it holds that 
  \begin{align}
      \left\|\widehat{\boldsymbol{\beta}}_{t-1}(\boldsymbol{\lambda}_{t-1})-\widehat{\boldsymbol{\beta}}_t(\boldsymbol{\lambda}_t)\right\|^2\leq 2 \left\|\widehat{\boldsymbol{\beta}}_t(\boldsymbol{\lambda}_{t-1})-\widehat{\boldsymbol{\beta}}_{t}(\boldsymbol{\lambda}_{t})\right\|^2+2\left\|\widehat{\boldsymbol{\beta}}_{t}(\boldsymbol{\lambda}_{t-1})-\widehat{\boldsymbol{\beta}}_{t-1}(\boldsymbol{\lambda}_{t-1})\right\|^2 \nonumber
  \end{align}
which through  Lemma \ref{lem:hypergrad_bound} can be further upper bounded with the Lipschitz constant of $\kappa_g$ as 
\begin{align}
      \left\|\widehat{\boldsymbol{\beta}}_{t-1}(\boldsymbol{\lambda}_{t-1})-\widehat{\boldsymbol{\beta}}_t(\boldsymbol{\lambda}_t)\right\|^2\leq 2\kappa_g^2 \left\|\boldsymbol{\lambda}_{t-1}-\boldsymbol{\lambda}_{t}\right\|^2+2\left\|\widehat{\boldsymbol{\beta}}_{t}(\boldsymbol{\lambda}_{t-1})-\widehat{\boldsymbol{\beta}}_{t-1}(\boldsymbol{\lambda}_{t-1})\right\|^2 \nonumber
\end{align}
Combining above, we see that $\forall \delta>0$, \eqref{eq:deterministic_inner_tracking_error_expansion} is upper bounded as
\begin{align}\label{eq:deterministic_inner _path_error_ub1}
    \left\|\boldsymbol{\beta}_{t}-\widehat{\boldsymbol{\beta}}_t(\boldsymbol{\lambda}_t)\right\|^2\leq (1+\delta)\left\|\boldsymbol{\beta}_{t}-\widehat{\boldsymbol{\beta}}_{t-1}(\boldsymbol{\lambda}_{t-1})\right\|^2 \nonumber\\+2\left(1+\frac{1}{\delta}\right)\kappa_g^2 \left\|\boldsymbol{\lambda}_{t-1}-\boldsymbol{\lambda}_{t}\right\|^2 +2\left(1+\frac{1}{\delta}\right)\left\|\widehat{\boldsymbol{\beta}}_{t}(\boldsymbol{\lambda}_{t-1})-\widehat{\boldsymbol{\beta}}_{t-1}(\boldsymbol{\lambda}_{t-1})\right\|^2.
\end{align}
As $\eta <\frac{1}{\ell_{g,1}}$, we apply Lemma \ref{lem:alt_grad_bound} to see
\begin{align}
    (1+\delta)\left\|\boldsymbol{\beta}_{t}-\widehat{\boldsymbol{\beta}}_{t-1}(\boldsymbol{\lambda}_{t-1})\right\|^2 \leq (1+\delta)(1-\eta\mu_g)^K\left\|\boldsymbol{\beta}_{t-1}-\widehat{\boldsymbol{\beta}}_{t-1}(\boldsymbol{\lambda}_{t-1})\right\|^2\nonumber
\end{align}
Now setting $\delta=\frac{\eta\mu_g}{2}>0$ implies that 
\begin{align}
    (1+\delta)(1-\eta\mu_g)^K=(1+\frac{\eta\mu_g}{2})(1-\eta\mu_g)^K <\left(1-\frac{\eta\mu_g}{2}\right)(1-\eta\mu_g)^{K-1}<1 \nonumber
\end{align}
Using $\nu:=\left(1-\frac{\eta\mu_g}{2}\right)(1-\eta\mu_g)^{K-1}$ in \eqref{eq:deterministic_inner _path_error_ub1}, we get
\begin{align}
\nu\left\|\boldsymbol{\beta}_{t}-\widehat{\boldsymbol{\beta}}_{t}(\boldsymbol{\lambda}_{t})\right\|^2 \leq \nu^2\left\|\boldsymbol{\beta}_{t-1}-\widehat{\boldsymbol{\beta}}_{t-1}(\boldsymbol{\lambda}_{t-1})\right\|^2 \nonumber\\+2C_{\mu_g}\nu\kappa_g^2\left\|\boldsymbol{\lambda}_{t-1}-\boldsymbol{\lambda}_{t}\right\|^2 +2C_{\mu_g}\nu\left\|\widehat{\boldsymbol{\beta}}_{t}(\boldsymbol{\lambda}_{t-1})-\widehat{\boldsymbol{\beta}}_{t-1}(\boldsymbol{\lambda}_{t-1})\right\|^2,\nonumber
\end{align}
where $C_{\mu_g}=\left(1+\frac{2}{\eta\mu_g}\right)$. Starting at $t=T$ and unrolling backward to $t=1$, results in the upper bound of
\begin{align}
\left\|\boldsymbol{\beta}_{t}-\widehat{\boldsymbol{\beta}}_t(\boldsymbol{\lambda}_t)\right\|^2 \leq \nu^{t-1}\Delta_{\boldsymbol{\beta}}  \nonumber\\+2C_{\mu_g}\kappa_g^2\sum_{j=0}^{t-2}\nu^{j} \left\|\boldsymbol{\lambda}_{t-1-j}-\boldsymbol{\lambda}_{t-j}\right\|^2 +2C_{\mu_g}\sum_{j=0}^{t-2}\nu^{j}\left\|\widehat{\boldsymbol{\beta}}_{t-j}(\boldsymbol{\lambda}_{t-1-j})-\widehat{\boldsymbol{\beta}}_{t-1-j}(\boldsymbol{\lambda}_{t-1-j})\right\|^2.\nonumber
\end{align}
\end{proof}

The next Lemma utilizes Lemma \ref{lem:itd_hypergradient_bound} and Lemma \ref{lem:deterministic_inner_tracking_error} to derive an upper bound on the hypergradient error $\forall t\in[1, T]$ in terms of discounted variations of the (i) cumulative time-smoothed hypergradient error; (ii) bilevel local regret;  and (iii) cumulative difference between optimal inner-level variables.  A final term is included, composed of a discounted initial error and smoothness term of the inner objective.
\begin{lemma}\label{lem:deterministic_hypergradient_error}
Suppose Assumptions \ref{assump:rel_smoothness}, \ref{assump:strong_convex_g_t}, \ref{assump:bounded_hyperparm}, and \ref{assump:continuity_phi_t}. Choose the inner step size of $\eta$ and inner iteration count of $K$ to satisfy
\begin{align}
    \eta < \min{\left(\frac{1}{\ell_{g,1}},\frac{1}{\mu_g}\right)},\ \text{and} \quad K\geq1. \nonumber
\end{align}  Using the definitions of $\nu, \ C_{\mu_g}$, and $\Delta_{\boldsymbol{\beta}}$ from Lemma \ref{lem:deterministic_inner_tracking_error} as well as the further definition of
   \begin{align}
       L_{\boldsymbol{\beta}}:=L^2_1(1-\eta \mu_g)^{K}+L^2_2(1-\eta \mu_g)^{K-1},\nonumber
   \end{align}
   then the  hypergradient error from our OBBO algorithm in Algorithm \ref{alg:deterministic_online_bregman} is bounded $\forall t\in[1,T]$ as

       \begin{align}
 \left\|\frac{\partial f_t(\boldsymbol{\lambda}_t,\boldsymbol{\omega}^K_t)}{\partial \boldsymbol{\lambda}}-\nabla F_t(\boldsymbol{\lambda}_t)\right\|^2 \leq \delta_t+A\sum_{j=0}^{t-2}\nu^{j} \left\|\frac{\partial f_{t-1-j,w}(\boldsymbol{\lambda}_{t-1-j},\boldsymbol{\omega}^K_{t-1-j})}{\partial \boldsymbol{\lambda}}-\nabla F_{t-1-j,w}(\boldsymbol{\lambda}_{t-1-j})\right\|^2\nonumber\\+B\sum_{j=0}^{t-2}\nu^{j} \left\|\mathcal{G}_{\mathcal{X}}\left(\boldsymbol{\lambda}_{t-1-j},\nabla F_{t-1-j,w}(\boldsymbol{\lambda}_{t-1-j}),\alpha\right)\right\|^2 +C\sum_{j=0}^{t-2}\nu^{j}\left\|\widehat{\boldsymbol{\beta}}_{t-j}(\boldsymbol{\lambda}_{t-1-j})-\widehat{\boldsymbol{\beta}}_{t-1-j}(\boldsymbol{\lambda}_{t-1-j})\right\|^2,
    \end{align}
    where $\delta_t=3L^2_3(1-\eta \mu_g)^{2K} +3L_{\boldsymbol{\beta}}\nu^{t-1}\Delta_{\boldsymbol{\beta}}$ and $A=\frac{12\alpha^2C_{\mu_g}L_{\boldsymbol{\beta}}\kappa_g^2}{\rho^2},\ B=12\alpha^2C_{\mu_g}L_{\boldsymbol{\beta}}\kappa_g^2$, and $C=6L_{\boldsymbol{\beta}}C_{\mu_g}$.
\end{lemma}
\begin{proof}[Proof of Lemma \ref{lem:deterministic_hypergradient_error}]
  Note that  Lemma \ref{lem:itd_hypergradient_bound} implies $\forall t\in[1,T]$
    \begin{align}
    \left\|\frac{\partial f_t(\boldsymbol{\lambda}_t,\boldsymbol{\omega}^K_t)}{\partial \boldsymbol{\lambda}}-\nabla F_t(\boldsymbol{\lambda}_t)\right\|^2\leq 3L_{\boldsymbol{\beta}}\left\|\boldsymbol{\beta}_t-\widehat{\boldsymbol{\beta}}_t(\boldsymbol{\lambda}_t)\right\|^2+3L^2_3(1-\eta \mu_g)^{2K}.\nonumber
    \end{align}
    Substituting the upper bound on $\left\|\boldsymbol{\beta}_t-\widehat{\boldsymbol{\beta}}_t(\boldsymbol{\lambda}_t)\right\|^2$ from Lemma \ref{lem:deterministic_inner_tracking_error}, we have
\begin{align}
    \left\|\frac{\partial f_t(\boldsymbol{\lambda}_t,\boldsymbol{\omega}^K_t)}{\partial \boldsymbol{\lambda}}-\nabla F_t(\boldsymbol{\lambda}_t)\right\|^2 \leq 3L^2_3(1-\eta \mu_g)^{2K} \nonumber\\+3L_{\boldsymbol{\beta}}\left(\nu^{t-1}\Delta_{\boldsymbol{\beta}}+2C_{\mu_g}\kappa_g^2\sum_{j=0}^{t-2}\nu^{j} \left\|\boldsymbol{\lambda}_{t-1-j}-\boldsymbol{\lambda}_{t-j}\right\|^2 \right)\nonumber\\+6L_{\boldsymbol{\beta}}C_{\mu_g}\sum_{j=0}^{t-2}\nu^{j}\left\|\widehat{\boldsymbol{\beta}}_{t-j}(\boldsymbol{\lambda}_{t-1-j})-\widehat{\boldsymbol{\beta}}_{t-1-j}(\boldsymbol{\lambda}_{t-1-j})\right\|^2,\nonumber
\end{align}
By definition, we have $\mathcal{G}_{\mathcal{X}}\left(\boldsymbol{\lambda}_{t-1-j},\frac{\partial f_{t-1-j,w}(\boldsymbol{\lambda}_{t-1-j},\boldsymbol{\omega}^K_{t-1-j})}{\partial \boldsymbol{\lambda}},\alpha\right):=\frac{1}{\alpha}\left(\boldsymbol{\lambda}_{t-1-j}-\boldsymbol{\lambda}_{t-j}\right)$
\begin{align}
\sum_{j=0}^{t-2}\nu^{j} \left\|\boldsymbol{\lambda}_{t-1-j}-\boldsymbol{\lambda}_{t-j}\right\|^2=\alpha^2\sum_{j=0}^{t-2}\nu^{j} \left\|\mathcal{G}_{\mathcal{X}}\left(\boldsymbol{\lambda}_{t-1-j},\frac{\partial f_{t-1-j,w}(\boldsymbol{\lambda}_{t-1-j},\boldsymbol{\omega}^K_{t-1-j})}{\partial \boldsymbol{\lambda}},\alpha\right)\right\|^2\nonumber\\\leq 2\alpha^2\sum_{j=0}^{t-2}\nu^{j} \left(\left\|\mathcal{G}_{\mathcal{X}}(\boldsymbol{\lambda}_{t-1-j},\nabla F_{t-1-j,w}(\boldsymbol{\lambda}_{t-1-j}),\alpha)\right\|^2\right)\nonumber\\+2\alpha^2\sum_{j=0}^{t-2}\nu^{j}\left(\left\|\mathcal{G}_{\mathcal{X}}\left(\boldsymbol{\lambda}_{t-1-j},\frac{\partial f_{t-1-j,w}(\boldsymbol{\lambda}_{t-1-j},\boldsymbol{\omega}^K_{t-1-j})}{\partial \boldsymbol{\lambda}},\alpha\right)-\mathcal{G}_{\mathcal{X}}\left(\boldsymbol{\lambda}_{t-1-j},\nabla F_{t-1-j,w}(\boldsymbol{\lambda}_{t-1-j}),\alpha\right)\right\|^2\right)\nonumber\\\leq2\alpha^2\sum_{j=0}^{t-2}\nu^{j} \left(\left\|\mathcal{G}_{\mathcal{X}}(\boldsymbol{\lambda}_{t-1-j},\nabla F_{t-1-j,w}(\boldsymbol{\lambda}_{t-1-j}),\alpha)\right\|^2\right)\nonumber\\+2\alpha^2\sum_{j=0}^{t-2}\nu^{j} \left(\frac{1}{\rho^2} \left\|\frac{\partial f_{t-1-j,w}(\boldsymbol{\lambda}_{t-1-j},\boldsymbol{\omega}^K_{t-1-j})}{\partial \boldsymbol{\lambda}}-\nabla F_{t-1-j,w}(\boldsymbol{\lambda}_{t-1-j})\right\|^2\right)
\end{align}
such that the last inequality comes from Lemma \ref{lem:gen_projection_bound_two}. Rearranging terms,  we have decomposed the hypergradient error term at $t$ in terms of the cumulative hypergradient error from $j=1,\ldots, t-1$
\begin{align}
 \left\|\frac{\partial f_t(\boldsymbol{\lambda}_t,\boldsymbol{\omega}^K_t)}{\partial \boldsymbol{\lambda}}-\nabla F_t(\boldsymbol{\lambda}_t)\right\|^2 \leq 3L^2_3(1-\eta \mu_g)^{2K} +3L_{\boldsymbol{\beta}}\nu^{t-1}\Delta_{\boldsymbol{\beta}}\nonumber\\+12\alpha^2C_{\mu_g}L_{\boldsymbol{\beta}}\kappa_g^2\sum_{j=0}^{t-2}\nu^{j} \left\|\mathcal{G}_{\mathcal{X}}(\boldsymbol{\lambda}_{t-1-j},\nabla F_{t-1-j,w}(\boldsymbol{\lambda}_{t-1-j}),\alpha)\right\|^2 \nonumber\\+\frac{12\alpha^2C_{\mu_g}L_{\boldsymbol{\beta}}\kappa_g^2}{\rho^2}\sum_{j=0}^{t-2}\nu^{j} \left\|\frac{\partial f_{t-1-j,w}(\boldsymbol{\lambda}_{t-1-j},\boldsymbol{\omega}^K_{t-1-j})}{\partial \boldsymbol{\lambda}}-\nabla F_{t-1-j,w}(\boldsymbol{\lambda}_{t-1-j})\right\|^2\nonumber\\+6L_{\boldsymbol{\beta}}C_{\mu_g}\sum_{j=0}^{t-2}\nu^{j}\left\|\widehat{\boldsymbol{\beta}}_{t-j}(\boldsymbol{\lambda}_{t-1-j})-\widehat{\boldsymbol{\beta}}_{t-1-j}(\boldsymbol{\lambda}_{t-1-j})\right\|^2,\nonumber
\end{align}
\end{proof}
The next Lemma provides an upper bound on the cumulative time-smoothed hypergradient error using the result of Lemma \ref{lem:deterministic_hypergradient_error}
\begin{lemma}\label{lem:deterministic_cumulative_hypergradient_error}
    Suppose Assumptions \ref{assump:rel_smoothness}, \ref{assump:strong_convex_g_t}, \ref{assump:bounded_hyperparm}, and \ref{assump:continuity_phi_t}.  Choose the inner step size of $\eta < \min{\left(\frac{1}{\ell_{g,1}},\frac{1}{\mu_g}\right)}$, the outer step size $\alpha\leq \frac{\rho\sqrt{(1-\nu)}}{\kappa_g\sqrt{108C_{\mu_g}L_{\boldsymbol{\beta}}}}$, and inner iteration count $  K=\frac{\log{(T)}}{\log{\left((1-\eta\mu_g)^{-1}\right)}}+1$.
    Then the cumulative time-smoothed hypergradient error from our OBBO algorithm in Algorithm \ref{alg:deterministic_online_bregman} satisfies
    \begin{align}
     \sum_{t=1}^T \left\|\frac{\partial f_{t,w}(\boldsymbol{\lambda}_t,\boldsymbol{\omega}^K_t)}{\partial \boldsymbol{\lambda}}-\nabla F_{t,w}(\boldsymbol{\lambda}_t)\right\|^2 \leq \frac{27}{8}\left(\frac{\Delta_{\boldsymbol{\beta}} L_{\boldsymbol{\beta}}}{(1-\nu)} +L^2_3\right)\nonumber\\+\frac{\rho^2}{8}\sum_{t=1}^T\left\|\mathcal{G}_{\mathcal{X}}(\boldsymbol{\lambda}_{t-1-j},\nabla F_{t-1-j,w}(\boldsymbol{\lambda}_{t-1-j}),\alpha)\right\|^2 +\frac{27L_{\boldsymbol{\beta}}C_{\mu_g}}{2(1-\nu)}\sum_{t=2}^T\left\|\widehat{\boldsymbol{\beta}}_{t}(\boldsymbol{\lambda}_{t-1})-\widehat{\boldsymbol{\beta}}_{t-1}(\boldsymbol{\lambda}_{t-1})\right\|^2,\nonumber
    \end{align}
\end{lemma}
\begin{proof}[Proof of Lemma \ref{lem:deterministic_cumulative_hypergradient_error}]
        Note by definition of the time-smoothed outer level objective and application of Young's inequality we have
        \begin{align}
          \left\|\frac{\partial f_{t,w}(\boldsymbol{\lambda}_t,\boldsymbol{\omega}^K_t)}{\partial \boldsymbol{\lambda}}-\nabla F_{t,w}(\boldsymbol{\lambda}_t)\right\|^2=\left\|\frac{1}{w}\sum_{i=0}^{w-1}\left[\frac{\partial f_{t-i}(\boldsymbol{\lambda}_{t-i},\boldsymbol{\omega}^K_{t-i})}{\partial \boldsymbol{\lambda}}-\nabla F_{t-i}(\boldsymbol{\lambda}_{t-i})\right]\right\|^2\nonumber\\=\left[\sum_{i=0}^{w-1}\frac{1}{w}\sum_{j=0}^{w-1}\frac{1}{w}\left\langle\frac{\partial f_{t-i}(\boldsymbol{\lambda}_{t-i},\boldsymbol{\omega}^K_{t-i})}{\partial \boldsymbol{\lambda}}-\nabla F_{t-i}(\boldsymbol{\lambda}_{t-i}),\frac{\partial f_{t-j}(\boldsymbol{\lambda}_{t-j},\boldsymbol{\omega}^K_{t-j})}{\partial \boldsymbol{\lambda}}-\nabla F_{t-j}(\boldsymbol{\lambda}_{t-j})\right\rangle\right]\nonumber\\\leq[\sum_{i=0}^{w-1}\frac{1}{w}\sum_{j=0}^{w-1}\frac{1}{w}(\frac{1}{2}\left\|\frac{\partial f_{t-i}(\boldsymbol{\lambda}_{t-i},\boldsymbol{\omega}^K_{t-i})}{\partial \boldsymbol{\lambda}}-\nabla F_{t-i}(\boldsymbol{\lambda}_{t-i})\right\|^2\nonumber\\+\frac{1}{2}\left\|\frac{\partial f_{t-j}(\boldsymbol{\lambda}_{t-j},\boldsymbol{\omega}^K_{t-j})}{\partial \boldsymbol{\lambda}}-\nabla F_{t-j}(\boldsymbol{\lambda}_{t-j})\right\|^2)]\nonumber\\=\frac{1}{w}\sum_{i=0}^{w-1}\left\|\frac{\partial f_{t-i}(\boldsymbol{\lambda}_{t-i},\boldsymbol{\omega}^K_{t-i})}{\partial \boldsymbol{\lambda}}-\nabla F_{t-i}(\boldsymbol{\lambda}_{t-i})\right\|^2
        \end{align}
        Substituting the upper bound on $\left\|\frac{\partial f_t(\boldsymbol{\lambda}_t,\boldsymbol{\omega}^K_t)}{\partial \boldsymbol{\lambda}}-\nabla F_t(\boldsymbol{\lambda}_t)\right\|^2$ from Lemma \ref{lem:deterministic_hypergradient_error} and re-indexing  the bilevel local regret and the cumulative time-smoothed hypergradient error, we construct the upper bound of
        \begin{align}\label{eq:second_det_cumulative_hg_error}
     \sum_{t=1}^T \left\|\frac{\partial f_{t,w}(\boldsymbol{\lambda}_t,\boldsymbol{\omega}^K_t)}{\partial \boldsymbol{\lambda}}-\nabla F_{t,w}(\boldsymbol{\lambda}_t)\right\|^2\nonumber\\\leq \sum_{t=1}^T\frac{1}{w}\left[\sum_{i=0}^{w-1}\left(3L^2_3(1-\eta \mu_g)^{2K} +3L_{\boldsymbol{\beta}}\nu^{t-i-1}\Delta_{\boldsymbol{\beta}} \right)\right] \nonumber\\+\sum_{t=1}^T\frac{1}{w}\left[\sum_{i=0}^{w-1}\left(12\alpha^2C_{\mu_g}L_{\boldsymbol{\beta}}\kappa_g^2\sum_{j=0}^{t-i-2}\nu^{j} \left\|\mathcal{G}_{\mathcal{X}}(\boldsymbol{\lambda}_{t-i-j},\nabla F_{t-i-j,w}(\boldsymbol{\lambda}_{t-i-j}),\alpha)\right\|^2 \right)\right]\nonumber\\+\sum_{t=1}^T\frac{1}{w}\left[\sum_{i=0}^{w-1}\left(\frac{12\alpha^2C_{\mu_g}L_{\boldsymbol{\beta}}\kappa_g^2}{\rho^2}\sum_{j=0}^{t-i-2}\nu^{j} \left\|\frac{\partial f_{t-i-j,w}(\boldsymbol{\lambda}_{t-i-j},\boldsymbol{\omega}^K_{t-i-j})}{\partial \boldsymbol{\lambda}}-\nabla F_{t-i-j,w}(\boldsymbol{\lambda}_{t-i-j})\right\|^2\right)\right]\nonumber\\+\sum_{t=2}^T\frac{1}{w}\left[\sum_{i=0}^{w-1}\left(6L_{\boldsymbol{\beta}}C_{\mu_g}\sum_{j=0}^{t-i-2}\nu^{j}\left\|\widehat{\boldsymbol{\beta}}_{t-i-j}(\boldsymbol{\lambda}_{t-i-1-j})-\widehat{\boldsymbol{\beta}}_{t-i-1-j}(\boldsymbol{\lambda}_{t-i-1-j})\right\|^2\right)\right]
        \end{align}
Given $\nu<1$, it holds that $\sum_{j=0}^{t-2}\nu^j<\sum_{j=0}^{\infty}\nu^j=\frac{1}{1-\nu}$, which lets us upper bound \eqref{eq:second_det_cumulative_hg_error} as
\begin{align}
    \sum_{t=1}^T \left\|\frac{\partial f_{t,w}(\boldsymbol{\lambda}_t,\boldsymbol{\omega}^K_t)}{\partial \boldsymbol{\lambda}}-\nabla F_{t,w}(\boldsymbol{\lambda}_t)\right\|^2\nonumber\\\leq \sum_{t=1}^T\frac{1}{w}\left[\sum_{i=0}^{w-1}\left(3L^2_3(1-\eta \mu_g)^{2K} +3L_{\boldsymbol{\beta}}\nu^{t-i-1}\Delta_{\boldsymbol{\beta}}  \right)\right]\nonumber\\+\frac{12\alpha^2C_{\mu_g}L_{\boldsymbol{\beta}}\kappa_g^2}{(1-\nu)}\sum_{t=1}^T\frac{1}{w}\left[\sum_{i=0}^{w-1} \left\|\mathcal{G}_{\mathcal{X}}(\boldsymbol{\lambda}_{t-i},\nabla F_{t-i,w}(\boldsymbol{\lambda}_{t-i}),\alpha)\right\|^2 \right]\nonumber\\+\frac{12\alpha^2C_{\mu_g}L_{\boldsymbol{\beta}}\kappa_g^2}{\rho^2(1-\nu)}\sum_{t=1}^T\frac{1}{w}\left[\sum_{i=0}^{w-1}\left\|\frac{\partial f_{t-i,w}(\boldsymbol{\lambda}_{t-i},\boldsymbol{\omega}^K_{t-i})}{\partial \boldsymbol{\lambda}}-\nabla F_{t-i,w}(\boldsymbol{\lambda}_{t-i})\right\|^2\right]\nonumber\\+\frac{6L_{\boldsymbol{\beta}}C_{\mu_g}}{(1-\nu)}\sum_{t=2}^T\frac{1}{w}\left[\sum_{i=0}^{w-1}\left\|\widehat{\boldsymbol{\beta}}_{t-i}(\boldsymbol{\lambda}_{t-i-1})-\widehat{\boldsymbol{\beta}}_{t-i-1}(\boldsymbol{\lambda}_{t-i-1})\right\|^2\right].\nonumber
\end{align}
and further 
\begin{align}
   \sum_{t=1}^T \left\|\frac{\partial f_{t,w}(\boldsymbol{\lambda}_t,\boldsymbol{\omega}^K_t)}{\partial \boldsymbol{\lambda}}-\nabla F_{t,w}(\boldsymbol{\lambda}_t)\right\|^2\nonumber\\\leq \sum_{t=1}^T\left(3L^2_3(1-\eta \mu_g)^{2K} +3L_{\boldsymbol{\beta}}\nu^{t-1}\Delta_{\boldsymbol{\beta}}\right)+\frac{12\alpha^2C_{\mu_g}L_{\boldsymbol{\beta}}\kappa_g^2}{(1-\nu)}\sum_{t=1}^T \left\|\mathcal{G}_{\mathcal{X}}(\boldsymbol{\lambda}_{t},\nabla F_{t,w}(\boldsymbol{\lambda}_{t}),\alpha)\right\|^2 \nonumber\\+\frac{12\alpha^2C_{\mu_g}L_{\boldsymbol{\beta}}\kappa_g^2}{\rho^2(1-\nu)}\sum_{t=1}^T \left\|\frac{\partial f_{t,w}(\boldsymbol{\lambda}_t,\boldsymbol{\omega}^K_t)}{\partial \boldsymbol{\lambda}}-\nabla F_{t,w}(\boldsymbol{\lambda}_t)\right\|^2+\frac{6L_{\boldsymbol{\beta}}C_{\mu_g}}{(1-\nu)}\sum_{t=2}^T\left\|\widehat{\boldsymbol{\beta}}_{t}(\boldsymbol{\lambda}_{t-1})-\widehat{\boldsymbol{\beta}}_{t-1}(\boldsymbol{\lambda}_{t-1})\right\|^2\nonumber
\end{align}
which implies that
\begin{align}
    \left(1-\frac{12\alpha^2C_{\mu_g}L_{\boldsymbol{\beta}}\kappa_g^2}{\rho^2(1-\nu)}\right)\sum_{t=1}^T \left\|\frac{\partial f_{t,w}(\boldsymbol{\lambda}_t,\boldsymbol{\omega}^K_t)}{\partial \boldsymbol{\lambda}}-\nabla F_{t,w}(\boldsymbol{\lambda}_t)\right\|^2 \nonumber\\\leq \frac{3\Delta_{\boldsymbol{\beta}}L_{\boldsymbol{\beta}}}{1-\nu} +\sum_{t=1}^T\left(3L^2_3(1-\eta \mu_g)^{2K}  \right) +\frac{12\alpha^2C_{\mu_g}L_{\boldsymbol{\beta}}\kappa_g^2}{(1-\nu)}\sum_{t=1}^T\left\|\mathcal{G}_{\mathcal{X}}(\boldsymbol{\lambda}_{t},\nabla F_{t,w}(\boldsymbol{\lambda}_{t}),\alpha)\right\|^2\nonumber\\ +\frac{6L_{\boldsymbol{\beta}}C_{\mu_g}}{(1-\nu)}\sum_{t=2}^T\left\|\widehat{\boldsymbol{\beta}}_{t}(\boldsymbol{\lambda}_{t-1})-\widehat{\boldsymbol{\beta}}_{t-1}(\boldsymbol{\lambda}_{t-1})\right\|^2,\nonumber
\end{align}

Setting $K=\log{(T)}/\log{\left((1-\eta\mu_g)^{-1}\right)}+1$ and   $0<\alpha\leq \frac{\rho\sqrt{(1-\nu)}}{\kappa_g\sqrt{108C_{\mu_g}L_{\boldsymbol{\beta}}}}$ 
\begin{align}
   \left(1-\frac{12\alpha^2C_{\mu_g}L_{\boldsymbol{\beta}}\kappa_g^2}{\rho^2(1-\nu)}\right)\geq\frac{8}{9}\nonumber
\end{align}
 implies the upper bound of
\begin{align}
    \sum_{t=1}^T \left\|\frac{\partial f_{t,w}(\boldsymbol{\lambda}_t,\boldsymbol{\omega}^K_t)}{\partial \boldsymbol{\lambda}}-\nabla F_{t,w}(\boldsymbol{\lambda}_t)\right\|^2 \leq \frac{27}{8}\left(\frac{\Delta_{\boldsymbol{\beta}} L_{\boldsymbol{\beta}}}{(1-\nu)} +L^2_3\right)\nonumber\\+\frac{\rho^2}{8}\sum_{t=1}^T\left\|\mathcal{G}_{\mathcal{X}}(\boldsymbol{\lambda}_{t},\nabla F_{t,w}(\boldsymbol{\lambda}_{t}),\alpha)\right\|^2 +\frac{27L_{\boldsymbol{\beta}}C_{\mu_g}}{2(1-\nu)}\sum_{t=2}^T\left\|\widehat{\boldsymbol{\beta}}_{t}(\boldsymbol{\lambda}_{t-1})-\widehat{\boldsymbol{\beta}}_{t-1}(\boldsymbol{\lambda}_{t-1})\right\|^2,\nonumber
\end{align}
    \end{proof}
The next theorem presents the theoretical contribution for our OBBO algorithm in Algorithm \ref{alg:deterministic_online_bregman}. For suitably chosen step sizes, the sequence of iterates $\left\{\boldsymbol{\lambda}_t\right\}_{t=1}^T$ achieves sublinear bilevel local regret.
\begin{theorem}\label{thrm:non_convex_regret}
   Suppose Assumptions \ref{assump:rel_smoothness}, \ref{assump:strong_convex_g_t},
   \ref{assump:bounded_hyperparm},
   \ref{assump:bounded_F_t}, \ref{assump:continuity_phi_t}. Choose the inner step size of $\eta < \min{\left(\frac{1}{\ell_{g,1}},\frac{1}{\mu_g}\right)}$, the outer step size of $\alpha\leq \min\left\{\frac{3\rho}{4\ell_{F,1}},\frac{\rho\sqrt{(1-\nu)}}{\kappa_g\sqrt{108C_{\mu_g}L_{\boldsymbol{\beta}}}}\right\}$,   and inner iteration count $  K=\frac{\log{(T)}}{\log{\left((1-\eta\mu_g)^{-1}\right)}}+1$. Then the bilevel local regret of our OBBO algorithm in Algorithm \ref{alg:deterministic_online_bregman} satisfies the bound of 
    \begin{align}
   BLR_w(T):= \sum_{t=1}^T\left\|\mathcal{G}_{\mathcal{X}}(\boldsymbol{\lambda}_{t},\nabla F_{t,w}(\boldsymbol{\lambda}_{t}),\alpha)\right\|^2\leq O\left(\frac{T}{w}+V_{1,T}+\kappa_g^2H_{2,T}\right),
\end{align}
\end{theorem}

    \begin{proof}[Proof of Theorem \ref{thrm:non_convex_regret}]
Note, with Assumption A  we have the upper bound of
\begin{align}
F_{t,w}\left(\boldsymbol{\lambda}_{t+1}\right)- F_{t,w}\left(\boldsymbol{\lambda}_{t}\right) =\frac{1}{w}\sum_{i=0}^{w-1}F_{t-i}\left(\boldsymbol{\lambda}_{t+1-i}\right)-\frac{1}{w}\sum_{i=0}^{w-1}F_{t-i}\left(\boldsymbol{\lambda}_{t-i}\right)\nonumber\\=\frac{1}{w}\sum_{i=0}^{w-1}\left[F_{t-i}\left(\boldsymbol{\lambda}_{t+1-i}\right)-F_{t-i}\left(\boldsymbol{\lambda}_{t-i}\right)\right]\nonumber\\\leq \frac{1}{w}\sum_{i=0}^{w-1}\left[\left\langle\nabla F_{t-i}\left(\boldsymbol{\lambda}_{t-i}\right),\boldsymbol{\lambda}_{t+1}-\boldsymbol{\lambda}_t\right\rangle+\frac{\ell_{F,1}}{2}\left\|\boldsymbol{\lambda}_{t+1}-\boldsymbol{\lambda}_t\right\|^2\right]\nonumber\\=\left\langle\nabla F_{t,w}\left(\boldsymbol{\lambda}_{t}\right),\boldsymbol{\lambda}_{t+1}-\boldsymbol{\lambda}_t\right\rangle+\frac{\ell_{F,1}}{2}\left\|\boldsymbol{\lambda}_{t+1}-\boldsymbol{\lambda}_t\right\|^2.\nonumber
\end{align}
Substituting in $\mathcal{G}_{\mathcal{X}}\left(\boldsymbol{\lambda}_{t},\frac{\partial f_{t,w}(\boldsymbol{\lambda}_{t},\boldsymbol{\omega}^K_{t})}{\partial \boldsymbol{\lambda}},\alpha\right):=\frac{1}{\alpha}\left(\boldsymbol{\lambda}_{t}-\boldsymbol{\lambda}_{t+1}\right)$,
\begin{align}\label{eq:initial_det_theorem_decomp}
F_{t,w}\left(\boldsymbol{\lambda}_{t+1}\right)- F_{t,w}\left(\boldsymbol{\lambda}_{t}\right)\leq\left\langle\nabla F_{t,w}\left(\boldsymbol{\lambda}_{t}\right),\boldsymbol{\lambda}_{t+1}-\boldsymbol{\lambda}_t\right\rangle+\frac{\ell_{F,1}}{2}\left\|\boldsymbol{\lambda}_{t+1}-\boldsymbol{\lambda}_t\right\|^2\nonumber\\=-\alpha\left\langle \nabla F_{t,w}\left(\boldsymbol{\lambda}_{t}\right),\mathcal{G}_{\mathcal{X}}\left(\boldsymbol{\lambda}_{t},\frac{\partial f_{t,w}(\boldsymbol{\lambda}_{t},\boldsymbol{\omega}^K_{t})}{\partial \boldsymbol{\lambda}},\alpha\right)\right\rangle+\frac{\alpha^2\ell_{F,1}}{2}\left\|\mathcal{G}_{\mathcal{X}}\left(\boldsymbol{\lambda}_{t},\frac{\partial f_{t,w}(\boldsymbol{\lambda}_{t},\boldsymbol{\omega}^K_{t})}{\partial \boldsymbol{\lambda}},\alpha\right)\right\|^2, \nonumber\\=-\alpha\left\langle \frac{\partial f_{t,w}(\boldsymbol{\lambda}_t,\boldsymbol{\omega}^K_t)}{\partial \boldsymbol{\lambda}},\mathcal{G}_{\mathcal{X}}\left(\boldsymbol{\lambda}_{t},\frac{\partial f_{t,w}(\boldsymbol{\lambda}_{t},\boldsymbol{\omega}^K_{t})}{\partial \boldsymbol{\lambda}},\alpha\right)\right\rangle\nonumber\\+\alpha\left\langle \frac{\partial f_{t,w}(\boldsymbol{\lambda}_t,\boldsymbol{\omega}^K_t)}{\partial \boldsymbol{\lambda}}-\nabla F_{t,w}\left(\boldsymbol{\lambda}_{t}\right),\mathcal{G}_{\mathcal{X}}\left(\boldsymbol{\lambda}_{t},\frac{\partial f_{t,w}(\boldsymbol{\lambda}_{t},\boldsymbol{\omega}^K_{t})}{\partial \boldsymbol{\lambda}},\alpha\right)\right\rangle+\frac{\alpha^2\ell_{F,1}}{2}\left\|\mathcal{G}_{\mathcal{X}}\left(\boldsymbol{\lambda}_{t},\frac{\partial f_{t,w}(\boldsymbol{\lambda}_{t},\boldsymbol{\omega}^K_{t})}{\partial \boldsymbol{\lambda}},\alpha\right)\right\|^2.
\end{align}
Using Lemma \ref{lem:gen_projection_bound_one} with $\boldsymbol{q}=\frac{\partial f_{t,w}(\boldsymbol{\lambda}_{t},\boldsymbol{\omega}^K_{t})}{\partial \boldsymbol{\lambda}}$, note that 
\begin{align}\label{eq:bregman_grad_1_det_thrm}
    \alpha\left\langle \frac{\partial f_{t,w}(\boldsymbol{\lambda}_t,\boldsymbol{\omega}^K_t)}{\partial \boldsymbol{\lambda}},\mathcal{G}_{\mathcal{X}}\left(\boldsymbol{\lambda}_{t},\frac{\partial f_{t,w}(\boldsymbol{\lambda}_{t},\boldsymbol{\omega}^K_{t})}{\partial \boldsymbol{\lambda}},\alpha\right)\right\rangle \nonumber\\\geq \alpha\rho\left\|\mathcal{G}_{\mathcal{X}}\left(\boldsymbol{\lambda}_{t},\frac{\partial f_{t,w}(\boldsymbol{\lambda}_{t},\boldsymbol{\omega}^K_{t})}{\partial \boldsymbol{\lambda}},\alpha\right)\right\|^2  +h(\boldsymbol{\lambda}_{t+1})-h(\boldsymbol{\lambda}_t)
\end{align}
and further we get the following based on a variation of Young's Inequality
\begin{align}\label{eq:bregman_grad_2_det_thrm}
    \left\langle \frac{\partial f_{t,w}(\boldsymbol{\lambda}_t,\boldsymbol{\omega}^K_t)}{\partial \boldsymbol{\lambda}}-\nabla F_{t,w}\left(\boldsymbol{\lambda}_{t}\right),\mathcal{G}_{\mathcal{X}}\left(\boldsymbol{\lambda}_{t},\frac{\partial f_{t,w}(\boldsymbol{\lambda}_{t},\boldsymbol{\omega}^K_{t})}{\partial \boldsymbol{\lambda}},\alpha\right)\right\rangle\nonumber\\\leq \frac{1}{\rho}\left\| \frac{\partial f_{t,w}(\boldsymbol{\lambda}_t,\boldsymbol{\omega}^K_t)}{\partial \boldsymbol{\lambda}}-\nabla F_{t,w}\left(\boldsymbol{\lambda}_{t}\right)\right\|^2+\frac{\rho}{4}\left\|\mathcal{G}_{\mathcal{X}}\left(\boldsymbol{\lambda}_{t},\frac{\partial f_{t,w}(\boldsymbol{\lambda}_{t},\boldsymbol{\omega}^K_{t})}{\partial \boldsymbol{\lambda}},\alpha\right)\right\|^2
\end{align}
Using \eqref{eq:bregman_grad_1_det_thrm} and \eqref{eq:bregman_grad_2_det_thrm} in \eqref{eq:initial_det_theorem_decomp} we get 
\begin{align}
F_{t,w}\left(\boldsymbol{\lambda}_{t+1}\right)- F_{t,w}\left(\boldsymbol{\lambda}_{t}\right)\leq \left(\frac{\alpha^2\ell_{F,1}}{2}-\frac{3\alpha\rho}{4}\right)\left\|\mathcal{G}_{\mathcal{X}}\left(\boldsymbol{\lambda}_{t},\frac{\partial f_{t,w}(\boldsymbol{\lambda}_{t},\boldsymbol{\omega}^K_{t})}{\partial \boldsymbol{\lambda}},\alpha\right)\right\|^2\nonumber\\+\frac{\alpha}{\rho}\left\|\frac{\partial f_{t,w}(\boldsymbol{\lambda}_t,\boldsymbol{\omega}^K_t)}{\partial \boldsymbol{\lambda}}-\nabla F_{t,w}\left(\boldsymbol{\lambda}_{t}\right)\right\|^2 +h(\boldsymbol{\lambda}_t)-h(\boldsymbol{\lambda}_{t+1})
\end{align}
which as $0<\alpha \leq \frac{3\rho}{4\ell_{F,1}}$ results in the further upper bound of
\begin{align}\label{eq:l_smoothnes_ub_det_thrm}
F_{t,w}\left(\boldsymbol{\lambda}_{t+1}\right)- F_{t,w}\left(\boldsymbol{\lambda}_{t}\right)\leq -\frac{3\alpha\rho}{8}\left\|\mathcal{G}_{\mathcal{X}}\left(\boldsymbol{\lambda}_{t},\frac{\partial f_{t,w}(\boldsymbol{\lambda}_{t},\boldsymbol{\omega}^K_{t})}{\partial \boldsymbol{\lambda}},\alpha\right)\right\|^2\nonumber\\+\frac{\alpha}{\rho}\left\|\frac{\partial f_{t,w}(\boldsymbol{\lambda}_t,\boldsymbol{\omega}^K_t)}{\partial \boldsymbol{\lambda}}-\nabla F_{t,w}\left(\boldsymbol{\lambda}_{t}\right)\right\|^2+h(\boldsymbol{\lambda}_t)-h(\boldsymbol{\lambda}_{t+1})
\end{align}
Further note we can upper bound the local regret as 
\begin{align}
    \left\|\mathcal{G}_{\mathcal{X}}(\boldsymbol{\lambda}_{t},\nabla F_{t,w}(\boldsymbol{\lambda}_{t}),\alpha)\right\|^2\leq 2\left\|\mathcal{G}_{\mathcal{X}}\left(\boldsymbol{\lambda}_{t},\frac{\partial f_{t,w}(\boldsymbol{\lambda}_{t},\boldsymbol{\omega}^K_{t})}{\partial \boldsymbol{\lambda}},\alpha\right)\right\|^2\nonumber\\+2\left\|\mathcal{G}_{\mathcal{X}}\left(\boldsymbol{\lambda}_{t},\frac{\partial f_{t,w}(\boldsymbol{\lambda}_{t},\boldsymbol{\omega}^K_{t})}{\partial \boldsymbol{\lambda}},\alpha\right)-\mathcal{G}_{\mathcal{X}}(\boldsymbol{\lambda}_{t},\nabla F_{t,w}(\boldsymbol{\lambda}_{t}),\alpha)\right\|^2\nonumber\\ \leq 2\left\|\mathcal{G}_{\mathcal{X}}\left(\boldsymbol{\lambda}_{t},\frac{\partial f_{t,w}(\boldsymbol{\lambda}_{t},\boldsymbol{\omega}^K_{t})}{\partial \boldsymbol{\lambda}},\alpha\right)\right\|^2+\frac{2}{\rho^2}\left\|\frac{\partial f_{t,w}(\boldsymbol{\lambda}_t,\boldsymbol{\omega}^K_t)}{\partial \boldsymbol{\lambda}}-\nabla F_{t,w}\left(\boldsymbol{\lambda}_{t}\right)\right\|^2,\nonumber
\end{align}
where the last inequality comes from Lemma \ref{lem:gen_projection_bound_two}. This then implies that
\begin{align}\label{eq:gradient_step_ub_deterministic_thrm}
-\left\|\mathcal{G}_{\mathcal{X}}\left(\boldsymbol{\lambda}_{t},\frac{\partial f_{t,w}(\boldsymbol{\lambda}_{t},\boldsymbol{\omega}^K_{t})}{\partial \boldsymbol{\lambda}},\alpha\right)\right\|^2\leq-\frac{1}{2}\left\|\mathcal{G}_{\mathcal{X}}(\boldsymbol{\lambda}_{t},\nabla F_{t,w}(\boldsymbol{\lambda}_{t}),\alpha)\right\|^2\nonumber\\+\frac{1}{\rho^2}\left\|\frac{\partial f_{t,w}(\boldsymbol{\lambda}_t,\boldsymbol{\omega}^K_t)}{\partial \boldsymbol{\lambda}}-\nabla F_{t,w}\left(\boldsymbol{\lambda}_{t}\right)\right\|^2
\end{align}
Substituting \eqref{eq:gradient_step_ub_deterministic_thrm} into \eqref{eq:l_smoothnes_ub_det_thrm} gives us
\begin{align}
F_{t,w}\left(\boldsymbol{\lambda}_{t+1}\right)- F_{t,w}\left(\boldsymbol{\lambda}_{t}\right)\leq -\frac{3\alpha\rho}{16}\left\|\mathcal{G}_{\mathcal{X}}(\boldsymbol{\lambda}_{t},\nabla F_{t,w}(\boldsymbol{\lambda}_{t}),\alpha)\right\|^2\nonumber\\+\left(\frac{\alpha}{\rho}+\frac{3\alpha}{8\rho}\right)\left\|\frac{\partial f_{t,w}(\boldsymbol{\lambda}_t,\boldsymbol{\omega}^K_t)}{\partial \boldsymbol{\lambda}}-\nabla F_{t,w}\left(\boldsymbol{\lambda}_{t}\right)\right\|^2 +h(\boldsymbol{\lambda}_t)-h(\boldsymbol{\lambda}_{t+1}).
\end{align}
Rearranging we see 
\begin{align}
    \frac{3\alpha\rho}{16}\left\|\mathcal{G}_{\mathcal{X}}(\boldsymbol{\lambda}_{t},\nabla F_{t,w}(\boldsymbol{\lambda}_{t}),\alpha)\right\|^2\leq   F_{t,w}\left(\boldsymbol{\lambda}_{t}\right)-F_{t,w}\left(\boldsymbol{\lambda}_{t+1}\right)\nonumber\\+\frac{11\alpha}{8\rho}\left\|\frac{\partial f_{t,w}(\boldsymbol{\lambda}_t,\boldsymbol{\omega}^K_t)}{\partial \boldsymbol{\lambda}}-\nabla F_{t,w}\left(\boldsymbol{\lambda}_{t}\right)\right\|^2 +h(\boldsymbol{\lambda}_t)-h(\boldsymbol{\lambda}_{t+1}).
\end{align}
Summing from $1,\ldots,T$ and telescoping $h(\boldsymbol{\lambda}_t)$ 
\begin{align}
    \frac{3\alpha\rho}{16}\sum_{t=1}^T\left\|\mathcal{G}_{\mathcal{X}}(\boldsymbol{\lambda}_{t},\nabla F_{t,w}(\boldsymbol{\lambda}_{t}),\alpha)\right\|^2\leq  \sum_{t=1}^T\left(F_{t,w}\left(\boldsymbol{\lambda}_{t}\right)-F_{t,w}\left(\boldsymbol{\lambda}_{t+1}\right)\right)\nonumber\\+\frac{11\alpha}{8\rho}\sum_{t=1}^T\left(\left\|\frac{\partial f_{t,w}(\boldsymbol{\lambda}_t,\boldsymbol{\omega}^K_t)}{\partial \boldsymbol{\lambda}}-\nabla F_{t,w}\left(\boldsymbol{\lambda}_{t}\right)\right\|^2\right)+\Delta_h, \nonumber
\end{align}
where $\Delta_h:=h(\boldsymbol{\lambda}_1)-h(\boldsymbol{\lambda}_{T+1})$ Then  we can substitute Lemma \ref{lem:deterministic_cumulative_hypergradient_error} to get
\begin{align}
    \frac{3\alpha\rho}{16}\sum_{t=1}^T\left\|\mathcal{G}_{\mathcal{X}}(\boldsymbol{\lambda}_{t},\nabla F_{t,w}(\boldsymbol{\lambda}_{t}),\alpha)\right\|^2\leq  \sum_{t=1}^T\left(F_{t,w}\left(\boldsymbol{\lambda}_{t}\right)-F_{t,w}\left(\boldsymbol{\lambda}_{t+1}\right)\right)\nonumber\\+\frac{11\alpha}{8\rho}\left( \frac{27}{8}\left(\frac{\Delta_{\boldsymbol{\beta}}L_{\boldsymbol{\beta}}}{(1-\nu)} +L^2_3\right)+\frac{\rho^2}{8}\sum_{t=1}^T\left\|\mathcal{G}_{\mathcal{X}}(\boldsymbol{\lambda}_{t},\nabla F_{t,w}(\boldsymbol{\lambda}_{t}),\alpha)\right\|^2 \right)\nonumber\\+\frac{11\alpha}{8\rho}\left(\frac{27L_{\boldsymbol{\beta}}C_{\mu_g}}{2(1-\nu)}\sum_{t=2}^T\left\|\widehat{\boldsymbol{\beta}}_{t}(\boldsymbol{\lambda}_{t-1})-\widehat{\boldsymbol{\beta}}_{t-1}(\boldsymbol{\lambda}_{t-1})\right\|^2\right)+\Delta_h  .\nonumber
\end{align}
Rearranging we have
\begin{align}
    \frac{12\alpha\rho}{64}\sum_{t=1}^T\left\|\mathcal{G}_{\mathcal{X}}(\boldsymbol{\lambda}_{t},\nabla F_{t,w}(\boldsymbol{\lambda}_{t}),\alpha)\right\|^2\leq  \sum_{t=1}^T\left(F_{t,w}\left(\boldsymbol{\lambda}_{t}\right)-F_{t,w}\left(\boldsymbol{\lambda}_{t+1}\right)\right)\nonumber\\+\frac{11\alpha\rho}{64}\sum_{t=1}^T\left\|\mathcal{G}_{\mathcal{X}}(\boldsymbol{\lambda}_{t},\nabla F_{t,w}(\boldsymbol{\lambda}_{t}),\alpha)\right\|^2+\frac{11\alpha}{8\rho}\left(\frac{27}{8}\left(\frac{\Delta_{\boldsymbol{\beta}} L_{\boldsymbol{\beta}}}{(1-\nu)} +L^2_3\right)\right)\nonumber\\+\frac{11\alpha}{8\rho}\left(\frac{27L_{\boldsymbol{\beta}}C_{\mu_g}}{2(1-\nu)}\sum_{t=2}^T\left\|\widehat{\boldsymbol{\beta}}_{t}(\boldsymbol{\lambda}_{t-1})-\widehat{\boldsymbol{\beta}}_{t-1}(\boldsymbol{\lambda}_{t-1})\right\|^2\right)+\Delta_h, \nonumber
\end{align}
or more succinctly 
\begin{align}\label{eq:succinct_det_thrm_ub}
    \sum_{t=1}^T\left\|\mathcal{G}_{\mathcal{X}}(\boldsymbol{\lambda}_{t},\nabla F_{t,w}(\boldsymbol{\lambda}_{t}),\alpha)\right\|^2\leq  \frac{64}{\alpha\rho}\sum_{t=1}^T\left(F_{t,w}\left(\boldsymbol{\lambda}_{t}\right)-F_{t,w}\left(\boldsymbol{\lambda}_{t+1}\right)\right)\nonumber\\+\frac{88}{\rho^2}\left(\frac{27}{8}\left(\frac{\Delta_{\boldsymbol{\beta}}L_{\boldsymbol{\beta}}}{(1-\nu)} +L^2_3\right) \right)+\frac{88}{\rho^2}\frac{27L_{\boldsymbol{\beta}}C_{\mu_g}}{2(1-\nu)}\sum_{t=2}^T\left\|\widehat{\boldsymbol{\beta}}_{t}(\boldsymbol{\lambda}_{t-1})-\widehat{\boldsymbol{\beta}}_{t-1}(\boldsymbol{\lambda}_{t-1})\right\|^2+\frac{64\Delta_h}{\alpha\rho}.
\end{align}
Applying Lemma \ref{lem:bounded_time_varying} we see 
\begin{align}\label{eq:subsequent_function_ub}
 \sum_{t=1}^T\left(F_{t,w}\left(\boldsymbol{\lambda}_{t}\right)-F_{t,w}\left(\boldsymbol{\lambda}_{t+1}\right)\right)\leq  \frac{2TQ}{w} +V_{1,T},
\end{align}
which by using \eqref{eq:subsequent_function_ub} in \eqref{eq:succinct_det_thrm_ub} we get for $L_{\boldsymbol{\beta}}=O(\kappa^2_g)$
\begin{align}
        \sum_{t=1}^T\left\|\mathcal{G}_{\mathcal{X}}(\boldsymbol{\lambda}_{t},\nabla F_{t,w}(\boldsymbol{\lambda}_{t}),\alpha)\right\|^2\leq  \frac{64}{\alpha\rho}\left(\frac{2TQ}{w} +V_{1,T}\right)+\frac{297}{\rho^2}\left(\frac{\Delta_{\boldsymbol{\beta}} L_{\boldsymbol{\beta}}}{(1-\nu)} +L^2_3\right)\nonumber\\+\frac{64\Delta_h}{\alpha\rho} +\frac{1188L_{\boldsymbol{\beta}}C_{\mu_g}}{\rho^2(1-\nu)}H_{2,T},
\end{align}

which  dividing by $T$ and recalling we imposed regularity constraints of $H_{2,T}=o(T)$, as well as $V_{1,T}=o(T)$,  implies the bilevel local regret of our OBBO algorithm is  sublinear on the order of
\begin{align}
   BLR_w(T):= \sum_{t=1}^T\left\|\mathcal{G}_{\mathcal{X}}(\boldsymbol{\lambda}_{t},\nabla F_{t,w}(\boldsymbol{\lambda}_{t}),\alpha)\right\|^2\leq O\left(\frac{T}{w}+V_{1,T}+\kappa_g^2H_{2,T}\right).
\end{align}
\end{proof}
\subsection{Stochastic Setting}
The next Lemma upper bounds the expected cumulative difference between the time-smoothed outer level objective $F_{t,w}(\boldsymbol{\lambda})$ evaluated at  $\boldsymbol{\lambda}_t$ and $\boldsymbol{\lambda}_{t+1}$ in terms of the outer level objective upper bound $m$,  window size $w$, and a comparator sequence on subsequent function evaluations $V_{1, T}$.
\begin{lemma}\label{lem:stochastic_bounded_time_varying} Suppose Assumption \ref{assump:bounded_F_t}. If our SOBBO algorithm in Algorithm \ref{alg:stochastic_online_bregman} is applied with window size $w\geq 1$ to generate the sequence $\{\boldsymbol{\lambda}_t\}_{t=1}^T$, then we have the upper bound in expectation of
\begin{align}
\sum_{t=1}^T\left(F_{t,w}(\boldsymbol{\lambda}_t)-F_{t,w}(\boldsymbol{\lambda}_{t+1})\right) \leq \frac{2TQ}{w} +V_{1,T}.\nonumber
\end{align}
where $V_{1,T}:=\sum_{t=1}^{T} \sup_{\boldsymbol{\lambda} \in \mathcal{X}}\left[F_{t+1}\left(\boldsymbol{\lambda}\right)-F_{t}\left(\boldsymbol{\lambda}\right)\right] $.
\end{lemma}
\begin{proof}[Proof of Lemma \ref{lem:stochastic_bounded_time_varying}]
By definition in the stochastic setting, we have $F_t(\boldsymbol{\lambda})\triangleq\mathbb{E}_{\epsilon}\left[f_t(\boldsymbol{\lambda},\widehat{\boldsymbol{\beta}}_t(\boldsymbol{\lambda}),\epsilon)\right]$. Then it holds, with the linearity of expectation that
\begin{align}
\sum_{t=1}^T\left(F_{t,w}(\boldsymbol{\lambda}_t)-F_{t,w}(\boldsymbol{\lambda}_{t+1})\right)=\sum_{t=1}^T\frac{1}{w}\sum_{i=0}^{w-1}\left(F_{t-i}(\boldsymbol{\lambda}_{t-i})-F_{t-i}(\boldsymbol{\lambda}_{t+1-i})\right)\nonumber\\=\sum_{t=1}^T\frac{1}{w}\sum_{i=0}^{w-1}\left(\mathbb{E}_{\epsilon}\left[f_{t-i}\left(\boldsymbol{\lambda}_{t-i},\widehat{\boldsymbol{\beta}}_{t-i}(\boldsymbol{\lambda}_{t-i}),\epsilon\right)\right]-\mathbb{E}_{\epsilon}\left[f_{t-i}\left(\boldsymbol{\lambda}_{t+1-i},\widehat{\boldsymbol{\beta}}_{t-i}(\boldsymbol{\lambda}_{t+1-i}),\epsilon\right)\right]\right)\nonumber\\=\sum_{t=1}^T\frac{1}{w}\sum_{i=0}^{w-1}\mathbb{E}_{\epsilon}\left[f_{t-i}\left(\boldsymbol{\lambda}_{t-i},\widehat{\boldsymbol{\beta}}_{t-i}(\boldsymbol{\lambda}_{t-i}),\epsilon\right)-f_{t-i}\left(\boldsymbol{\lambda}_{t+1-i},\widehat{\boldsymbol{\beta}}_{t-i}(\boldsymbol{\lambda}_{t+1-i}),\epsilon\right)\right]\nonumber
\end{align}
Which with the linearity of expectation is equivalent to
\begin{align}
    \sum_{t=1}^T\frac{1}{w}\sum_{i=0}^{w-1}\mathbb{E}_{\epsilon}\left[f_{t-i}\left(\boldsymbol{\lambda}_{t-i},\widehat{\boldsymbol{\beta}}_{t-i}(\boldsymbol{\lambda}_{t-i}),\epsilon\right)-f_{t-i}\left(\boldsymbol{\lambda}_{t+1-i},\widehat{\boldsymbol{\beta}}_{t-i}(\boldsymbol{\lambda}_{t+1-i}),\epsilon\right) \right]\nonumber\\ \label{eq:stochastic_first_term_rhs_function_eval}=  \sum_{t=1}^T\frac{1}{w}\sum_{i=0}^{w-1}\mathbb{E}_{\epsilon}\left[f_{t-i}\left(\boldsymbol{\lambda}_{t-i},\widehat{\boldsymbol{\beta}}_{t-i}(\boldsymbol{\lambda}_{t-i}),\epsilon\right)-f_{t+1-i}\left(\boldsymbol{\lambda}_{t+1-i},\widehat{\boldsymbol{\beta}}_{t+1-i}(\boldsymbol{\lambda}_{t+1-i}),\epsilon\right)\right] \\ \label{eq:stochastic_second_term_rhs_function_eval}+  \sum_{t=1}^T\frac{1}{w}\sum_{i=0}^{w-1}\mathbb{E}_{\epsilon}\left[f_{t+1-i}\left(\boldsymbol{\lambda}_{t+1-i},\widehat{\boldsymbol{\beta}}_{t+1-i}(\boldsymbol{\lambda}_{t+1-i}),\epsilon\right)-f_{t-i}\left(\boldsymbol{\lambda}_{t+1-i},\widehat{\boldsymbol{\beta}}_{t-i}(\boldsymbol{\lambda}_{t+1-i}),\epsilon\right)\right]
\end{align}
For \eqref{eq:stochastic_first_term_rhs_function_eval}, with linearity of expectation, we have
\begin{align}\label{eq:stochastic_ub_first_term_rhs_function_eval}
   \frac{1}{w}\sum_{i=0}^{w-1}\mathbb{E}_{\epsilon}\left[f_{t-i}\left(\boldsymbol{\lambda}_{t-i},\widehat{\boldsymbol{\beta}}_{t-i}(\boldsymbol{\lambda}_{t-i}),\epsilon\right)-f_{t+1-i}\left(\boldsymbol{\lambda}_{t+1-i},\widehat{\boldsymbol{\beta}}_{t+1-i}(\boldsymbol{\lambda}_{t+1-i}),\epsilon\right)\right]\nonumber\\=\frac{1}{w}\mathbb{E}_{\epsilon}\left[f_{t}\left(\boldsymbol{\lambda}_{t},\widehat{\boldsymbol{\beta}}_{t}(\boldsymbol{\lambda}_{t}),\epsilon\right)+\ldots+f_{t+1-w}\left(\boldsymbol{\lambda}_{t+1-w},\widehat{\boldsymbol{\beta}}_{t+1-w}(\boldsymbol{\lambda}_{t+1-w}),\epsilon\right)\right]\nonumber\\-\frac{1}{w}\mathbb{E}_{\epsilon}\left[f_{t+1}\left(\boldsymbol{\lambda}_{t+1},\widehat{\boldsymbol{\beta}}_{t+1}(\boldsymbol{\lambda}_{t+1})\right)+\ldots+f_{t+2-w}\left(\boldsymbol{\lambda}_{t+2-w},\widehat{\boldsymbol{\beta}}_{t+2-w}(\boldsymbol{\lambda}_{t+2-w}),\epsilon\right)\right]\nonumber\\=\frac{1}{w}\mathbb{E}_{\epsilon}\left[f_{t+1-w}\left(\boldsymbol{\lambda}_{t+1-w},\widehat{\boldsymbol{\beta}}_{t+1-w}(\boldsymbol{\lambda}_{t+1-w}),\epsilon\right)-f_{t+1}\left(\boldsymbol{\lambda}_{t+1},\widehat{\boldsymbol{\beta}}_{t+1}(\boldsymbol{\lambda}_{t+1}),\epsilon\right)\right]\nonumber\\=\frac{1}{w}\left(F_{t+1-w}(\boldsymbol{\lambda}_{t+1-w})-F_{t+1}(\boldsymbol{\lambda}_{t+1})\right)\leq \frac{2Q}{w},
\end{align}
where the last inequality comes from Assumption \ref{assump:bounded_F_t}. Note \eqref{eq:stochastic_second_term_rhs_function_eval} can be bounded through 
\begin{align}\label{eq:stochastic_ub_second_term_rhs_function_eval}
    \sum_{t=1}^T\frac{1}{w}\sum_{i=0}^{w-1}\mathbb{E}_{\epsilon}\left[f_{t+1-i}\left(\boldsymbol{\lambda}_{t+1-i},\widehat{\boldsymbol{\beta}}_{t+1-i}(\boldsymbol{\lambda}_{t+1-i}),\epsilon\right)-f_{t-i}\left(\boldsymbol{\lambda}_{t+1-i},\widehat{\boldsymbol{\beta}}_{t-i}(\boldsymbol{\lambda}_{t+1-i}),\epsilon\right)\right] \nonumber\\ \leq \sum_{t=1}^T\frac{1}{w}\sum_{i=0}^{w-1}\sup_{\boldsymbol{\lambda}}\mathbb{E}_{\epsilon}\left[f_{t+1-i}\left(\boldsymbol{\lambda},\widehat{\boldsymbol{\beta}}_{t+1-i}(\boldsymbol{\lambda}),\epsilon\right)-f_{t-i}\left(\boldsymbol{\lambda},\widehat{\boldsymbol{\beta}}_{t-i}(\boldsymbol{\lambda}),\epsilon\right)\right] \nonumber\\=\sum_{t=1}^{T} \sup_{\boldsymbol{\lambda} \in \mathcal{X}}\left[F_{t+1}\left(\boldsymbol{\lambda}\right)-F_{t}\left(\boldsymbol{\lambda}\right)\right] :=V_{1,T}
\end{align}
Combining \eqref{eq:stochastic_second_term_rhs_function_eval} and \eqref{eq:stochastic_ub_second_term_rhs_function_eval} results in the upper bound of \begin{align}
\sum_{t=1}^T\left(F_{t,w}(\boldsymbol{\lambda}_t)-F_{t,w}(\boldsymbol{\lambda}_{t+1})\right) \leq \frac{2TQ}{w} +V_{1,T}.\nonumber
\end{align}
\end{proof} 
The next Lemma provides an upper bound on the expected error of $\mathbb{E}\left[\left\|\boldsymbol{\beta}_{t}-\widehat{\boldsymbol{\beta}}_t(\boldsymbol{\lambda}_t)\right\|^2\right]$ for all $t\in[1, T]$ in terms of an expected initial error, the expected cumulative differences of the outer level variable, the expected cumulative differences of the optimal inner level variables, and a variance term arising from the stochasticity of $g_t(\boldsymbol{\lambda},\boldsymbol{\beta},\zeta)$.
\begin{lemma}\label{lem:stochastic_inner_tracking_error}
Suppose Assumptions \ref{assump:rel_smoothness}, \ref{assump:strong_convex_g_t}, and  \ref{assump:unbiased_finite_var}. Choose the inner step size of $\eta$  and the inner iteration count $K$ as
\begin{align}
    0<\eta\leq\frac{2}{\ell_{g,1}+\mu_g}, \ \text{and} \quad K\geq 1\nonumber,
\end{align}
and define the decay parameter $\nu$, the inner level variable error constant $C_{\mu_g}$, the initial error $\Delta_{\boldsymbol{\beta}}$, and the inner level variable error variance $C_K$ respectively as
\begin{align}
    \nu:=\left(1-\frac{\eta\ell_{g,1}\mu_g}{\ell_{g,1}+\mu_g}\right)\left(1-\frac{2\eta\ell_{g,1}\mu_g}{\ell_{g,1}+\mu_g}\right)^{K-1}, \quad C_{\mu_g}:=\left(1+\frac{\ell_{g,1}+\mu_g}{\eta\ell_{g,1}\mu_g}\right),\nonumber\\ \quad \Delta_{\boldsymbol{\beta}}:=\left\|\boldsymbol{\beta}_{2}-\widehat{\boldsymbol{\beta}}_1(\boldsymbol{\lambda}_1)\right\|^2, \quad  \text{and} \quad C_K:=\sum_{k=1}^{K}\left(1-\frac{2\eta\ell_{g,1}\mu_g}{\ell_{g,1}+\mu_g}\right)^k.
\end{align}
Then we have $\forall t \in[1,T]$,
\begin{align}
\mathbb{E}_{\bar{\zeta}_{t,K+1}}\left[\left\|\boldsymbol{\beta}_{t+1}-\widehat{\boldsymbol{\beta}}_t(\boldsymbol{\lambda}_t)\right\|^2\right] \leq \nu^{t-1}\Delta_{\boldsymbol{\beta}}+2C_{\mu_g}\kappa_g^2\sum_{j=0}^{t-2}\nu^{j+1} \left[\left\|\boldsymbol{\lambda}_{t-1-j}-\boldsymbol{\lambda}_{t-j}\right\|^2\right]\nonumber\\ +2C_{\mu_g}\sum_{j=0}^{t-2}\nu^{j+1}\left[\left\|\widehat{\boldsymbol{\beta}}_{t-j}(\boldsymbol{\lambda}_{t-1-j})-\widehat{\boldsymbol{\beta}}_{t-1-j}(\boldsymbol{\lambda}_{t-1-j})\right\|^2\right]+\frac{C_K\eta^2\sigma^2_{g_{\boldsymbol{\beta}}}}{s} \sum_{j=0}^{t-2}\nu^{j}.\nonumber
\end{align}
\end{lemma}
\begin{proof}[Proof of Lemma \ref{lem:stochastic_inner_tracking_error}]
        Note $\forall k\in[1,K]$ the following expansion holds
        \begin{align}
           \left\|\boldsymbol{\omega}^{k}_{t}-\widehat{\boldsymbol{\beta}}_t(\boldsymbol{\lambda}_t)\right\|^2\nonumber\\=\left\|\boldsymbol{\omega}^k_{t}-\boldsymbol{\omega}^{k-1}_{t}\right\|^2+ 2\left\langle\boldsymbol{\omega}^k_{t}-\boldsymbol{\omega}^{k-1}_{t}, \boldsymbol{\omega}^{k-1}_{t}-\widehat{\boldsymbol{\beta}}_t(\boldsymbol{\lambda}_t)\right\rangle +\left\|\boldsymbol{\omega}^{k-1}_{t}-\widehat{\boldsymbol{\beta}}_t(\boldsymbol{\lambda}_t)\right\|^2\nonumber\\=\eta^2\left\|\nabla_{\boldsymbol{\omega}} g_t(\boldsymbol{\lambda}_t,\boldsymbol{\omega}_t^{k-1},\bar{\zeta}_{t,k})\right\|^2-2\eta\left\langle\nabla_{\boldsymbol{\omega}} g_t(\boldsymbol{\lambda}_t,\boldsymbol{\omega}_t^{k-1},\bar{\zeta}_{t,k}),\boldsymbol{\omega}^{k-1}_{t}-\widehat{\boldsymbol{\beta}}_t(\boldsymbol{\lambda}_t)\right\rangle\nonumber\\+\left\|\boldsymbol{\omega}^{k-1}_{t}-\widehat{\boldsymbol{\beta}}_t(\boldsymbol{\lambda}_t)\right\|^2.\nonumber
        \end{align}
        Using the definition of variance of
        \begin{align}
            VAR_{\bar{\zeta}_{t,k}}\left[\left\|\nabla_{\boldsymbol{\omega}} g_t(\boldsymbol{\lambda}_t,\boldsymbol{\omega}_t^{k-1},\bar{\zeta}_{t,k})\right\|\right]\nonumber\\=\mathbb{E}_{\bar{\zeta}_{t,k}}\left[\left\|\nabla_{\boldsymbol{\omega}} g_t(\boldsymbol{\lambda}_t,\boldsymbol{\omega}_t^{k-1},\bar{\zeta}_{t,k})\right\|^2\right]-\mathbb{E}_{\bar{\zeta}_{t,k}}\left[\left\|\nabla_{\boldsymbol{\omega}} g_t(\boldsymbol{\lambda}_t,\boldsymbol{\omega}_t^{k-1},\bar{\zeta}_{t,k})\right\|\right]^2\nonumber,
        \end{align}
and conditioning on $\boldsymbol{\omega}^{k-1}_t$, we take expectation to provide the upper bound of
        \begin{align}\label{eq:initial_stochastic_inner_ub}
    \mathbb{E}_{\bar{\zeta}_{t,k}}\left[\left\|\boldsymbol{\omega}^{k}_{t}-\widehat{\boldsymbol{\beta}}_t(\boldsymbol{\lambda}_t)\right\|^2\right]\leq\eta^2\left(\frac{\sigma^2_{g_{\boldsymbol{\beta}}}}{s}+\left\|\nabla_{\boldsymbol{\omega}} g_t(\boldsymbol{\lambda}_t,\boldsymbol{\omega}_t^{k-1})\right\|^2\right)\nonumber\\-2\eta\left\langle\nabla_{\boldsymbol{\omega}} g_t(\boldsymbol{\lambda}_t,\boldsymbol{\omega}_t^{k-1}),\boldsymbol{\omega}^{k-1}_{t}-\widehat{\boldsymbol{\beta}}_t(\boldsymbol{\lambda}_t)\right\rangle+\left\|\boldsymbol{\omega}^{k-1}_{t}-\widehat{\boldsymbol{\beta}}_t(\boldsymbol{\lambda}_t)\right\|^2.
    \end{align}
    The above upper bound is deterministic, and as such we can utilize the $\mu_g$-strong convexity of $g_t$ to bound
    \begin{align}
    -2\eta\left\langle\nabla_{\boldsymbol{\omega}} g_t(\boldsymbol{\lambda}_t,\boldsymbol{\omega}_t^{k-1}),\boldsymbol{\omega}^{k-1}_{t}-\widehat{\boldsymbol{\beta}}_t(\boldsymbol{\lambda}_t)\right\rangle \nonumber\\\leq -2\eta\left(\frac{\ell_{g,1}\mu_g}{\ell_{g,1}+\mu_g}\left\|\boldsymbol{\omega}^{k-1}_{t}-\widehat{\boldsymbol{\beta}}_t(\boldsymbol{\lambda}_t)\right\|^2+\frac{1}{\ell_{g,1}+\mu_g}\left\|\nabla_{\boldsymbol{\omega}} g_t(\boldsymbol{\lambda}_t,\boldsymbol{\omega}_t^{k-1})\right\|^2\right),\nonumber
    \end{align}
    which we can substitute in \eqref{eq:initial_stochastic_inner_ub} to get
    \begin{align}
    \mathbb{E}_{\bar{\zeta}_{t,k}}\left[\left\|\boldsymbol{\omega}^{k}_{t}-\widehat{\boldsymbol{\beta}}_t(\boldsymbol{\lambda}_t)\right\|^2\right]\leq    
    \frac{\eta^2\sigma^2_{g_{\boldsymbol{\beta}}}}{s}-\eta\left(\frac{2}{\ell_{g,1}+\mu_g}-\eta\right)\left\|\nabla_{\boldsymbol{\omega}} g_t(\boldsymbol{\lambda}_t,\boldsymbol{\omega}_t^{k-1})\right\|^2\nonumber\\+\left(1-\frac{2\eta\ell_{g,1}\mu_g}{\ell_{g,1}+\mu_g}\right)\left\|\boldsymbol{\omega}^{k-1}_{t}-\widehat{\boldsymbol{\beta}}_t(\boldsymbol{\lambda}_t)\right\|^2\nonumber.
        \end{align} 
        As $\eta\leq\frac{2}{\ell_{g,1}+\mu_g}$ this provides the upper bound to \eqref{eq:initial_stochastic_inner_ub} of 
        \begin{align}
        \mathbb{E}_{\bar{\zeta}_{t,k}}\left[\left\|\boldsymbol{\omega}^{k}_{t}-\widehat{\boldsymbol{\beta}}_t(\boldsymbol{\lambda}_t)\right\|^2\right]\leq \left(1-\frac{2\eta\ell_{g,1}\mu_g}{\ell_{g,1}+\mu_g}\right)\left\|\boldsymbol{\omega}^{k-1}_{t}-\widehat{\boldsymbol{\beta}}_t(\boldsymbol{\lambda}_t)\right\|^2+\frac{\eta^2\sigma^2_{g_{\boldsymbol{\beta}}}}{s}\nonumber.
        \end{align}
        This can be unrolled, through iterative conditioning, from $k=K,\ldots,1$
        \begin{align}
        \mathbb{E}_{\bar{\zeta}_{t,K+1}}\left[\left\|\boldsymbol{\omega}^{K}_{t}-\widehat{\boldsymbol{\beta}}_t(\boldsymbol{\lambda}_t)\right\|^2\right]\leq \left(1-\frac{2\eta\ell_{g,1}\mu_g}{\ell_{g,1}+\mu_g}\right)^K\mathbb{E}_{\bar{\zeta}_{t,1}}\left\|\boldsymbol{\omega}^{0}_{t}-\widehat{\boldsymbol{\beta}}_t(\boldsymbol{\lambda}_t)\right\|^2+\frac{C_{K}\eta^2\sigma^2_{g_{\boldsymbol{\beta}}}}{s}\nonumber,
        \end{align}
        for $C_K:=\sum_{k=1}^{K}\left(1-\frac{2\eta\ell_{g,1}\mu_g}{\ell_{g,1}+\mu_g}\right)^k$.  By definition of $\boldsymbol{\beta}_{t+1}=\boldsymbol{\omega}^K_{t}$ and $\boldsymbol{\omega}^0_t=\boldsymbol{\beta}_t$ gives us
        \begin{align}
        \mathbb{E}_{\bar{\zeta}_{t,K+1}}\left[\left\|\boldsymbol{\beta}_{t+1}-\widehat{\boldsymbol{\beta}}_t(\boldsymbol{\lambda}_t)\right\|^2\right]\leq \left(1-\frac{2\eta\ell_{g,1}\mu_g}{\ell_{g,1}+\mu_g}\right)^K\mathbb{E}_{\bar{\zeta}_{t-1,K+1}}\left\|\boldsymbol{\beta}_{t}-\widehat{\boldsymbol{\beta}}_t(\boldsymbol{\lambda}_t)\right\|^2+\frac{C_K\eta^2\sigma^2_{g_{\boldsymbol{\beta}}}}{s}\nonumber.
        \end{align}
    Note we can decompose
    \begin{align}
       \mathbb{E}_{\bar{\zeta}_{t-1,K+1}}\left\|\boldsymbol{\beta}_{t}-\widehat{\boldsymbol{\beta}}_t(\boldsymbol{\lambda}_t)\right\|^2=\mathbb{E}_{\bar{\zeta}_{t-1,K+1}}\left\|\boldsymbol{\beta}_t-\widehat{\boldsymbol{\beta}}_{t-1}(\boldsymbol{\lambda}_{t-1})+\widehat{\boldsymbol{\beta}}_{t-1}(\boldsymbol{\lambda}_{t-1})-\widehat{\boldsymbol{\beta}}_t(\boldsymbol{\lambda}_t)\right\|^2,\nonumber
    \end{align}
    which can be expanded  based on Young's Inequality and the linearity of expectation for any $\delta>0$ as 
    \begin{align}\label{eq:youngs_decomp_beta_stochastic}
      \mathbb{E}_{\bar{\zeta}_{t-1,K+1}}\left\|\boldsymbol{\beta}_t-\widehat{\boldsymbol{\beta}}_{t-1}(\boldsymbol{\lambda}_{t-1})+\widehat{\boldsymbol{\beta}}_{t-1}(\boldsymbol{\lambda}_{t-1})-\widehat{\boldsymbol{\beta}}_t(\boldsymbol{\lambda}_t)\right\|^2\nonumber\\ \leq (1+\delta)      \mathbb{E}_{\bar{\zeta}_{t-1,K+1}}\left\|\boldsymbol{\beta}_{t}-\widehat{\boldsymbol{\beta}}_{t-1}(\boldsymbol{\lambda}_{t-1})\right\|^2\nonumber\\+\left(1+\frac{1}{\delta}\right)\mathbb{E}_{\bar{\zeta}_{t-1,K+1}}\left\|\widehat{\boldsymbol{\beta}}_{t-1}(\boldsymbol{\lambda}_{t-1})-\widehat{\boldsymbol{\beta}}_t(\boldsymbol{\lambda}_t)\right\|^2.
    \end{align}
  Now  it holds through linearity of expectation that 
\begin{align}\label{eq:stochastic_inner_beta_decomp}
      \mathbb{E}_{\bar{\zeta}_{t-1,K+1}}\left\|\widehat{\boldsymbol{\beta}}_{t-1}(\boldsymbol{\lambda}_{t-1})-\widehat{\boldsymbol{\beta}}_t(\boldsymbol{\lambda}_t)\right\|^2\leq 2 \mathbb{E}_{\bar{\zeta}_{t-1,K+1}}\left\|\widehat{\boldsymbol{\beta}}_t(\boldsymbol{\lambda}_{t-1})-\widehat{\boldsymbol{\beta}}_{t}(\boldsymbol{\lambda}_{t})\right\|^2\nonumber\\+2\mathbb{E}_{\bar{\zeta}_{t-1,K+1}}\left\|\widehat{\boldsymbol{\beta}}_{t}(\boldsymbol{\lambda}_{t-1})-\widehat{\boldsymbol{\beta}}_{t-1}(\boldsymbol{\lambda}_{t-1})\right\|^2 
  \end{align}
which through Lemma \ref{lem:hypergrad_bound} can be further upper bounded with the Lipschitz constant of $\kappa_g$ as 
\begin{align}
      \mathbb{E}_{\bar{\zeta}_{t-1,K+1}}\left\|\widehat{\boldsymbol{\beta}}_{t-1}(\boldsymbol{\lambda}_{t-1})-\widehat{\boldsymbol{\beta}}_t(\boldsymbol{\lambda}_t)\right\|^2\nonumber\\\leq 2\kappa_g^2 \mathbb{E}_{\bar{\zeta}_{t-1,K+1}}\left\|\boldsymbol{\lambda}_{t-1}-\boldsymbol{\lambda}_{t}\right\|^2+2\mathbb{E}_{\bar{\zeta}_{t-1,K+1}}\left\|\widehat{\boldsymbol{\beta}}_{t}(\boldsymbol{\lambda}_{t-1})-\widehat{\boldsymbol{\beta}}_{t-1}(\boldsymbol{\lambda}_{t-1})\right\|^2 \nonumber\\=2\kappa_g^2 \left\|\boldsymbol{\lambda}_{t-1}-\boldsymbol{\lambda}_{t}\right\|^2+2\left\|\widehat{\boldsymbol{\beta}}_{t}(\boldsymbol{\lambda}_{t-1})-\widehat{\boldsymbol{\beta}}_{t-1}(\boldsymbol{\lambda}_{t-1})\right\|^2 \nonumber\\
\end{align}
where the last line comes from the non-randomness of $\left\|\boldsymbol{\lambda}_{t-1}-\boldsymbol{\lambda}_{t}\right\|^2$ and $\left\|\widehat{\boldsymbol{\beta}}_{t}(\boldsymbol{\lambda}_{t-1})-\widehat{\boldsymbol{\beta}}_{t-1}(\boldsymbol{\lambda}_{t-1})\right\|^2$ with respect to $\bar{\zeta}_{t,k}$. Combining \eqref{eq:stochastic_inner_beta_decomp} and \eqref{eq:youngs_decomp_beta_stochastic}, we have $\forall \delta>0$
\begin{align}
    \mathbb{E}_{\bar{\zeta}_{t,K+1}}\left[\left\|\boldsymbol{\beta}_{t+1}-\widehat{\boldsymbol{\beta}}_t(\boldsymbol{\lambda}_t)\right\|^2\right]\leq \left(1-\frac{2\eta\ell_{g,1}\mu_g}{\ell_{g,1}+\mu_g}\right)^K(1+\delta)\mathbb{E}_{\bar{\zeta}_{t-1,K+1}}\left[\left\|\boldsymbol{\beta}_{t}-\widehat{\boldsymbol{\beta}}_{t-1}(\boldsymbol{\lambda}_{t-1})\right\|^2\right] \nonumber\\+2\left(1-\frac{2\eta\ell_{g,1}\mu_g}{\ell_{g,1}+\mu_g}\right)^K\left(1+\frac{1}{\delta}\right)\kappa_g^2 \left\|\boldsymbol{\lambda}_{t-1}-\boldsymbol{\lambda}_{t}\right\|^2\nonumber\\ +2\left(1-\frac{2\eta\ell_{g,1}\mu_g}{\ell_{g,1}+\mu_g}\right)^K\left(1+\frac{1}{\delta}\right)\left\|\widehat{\boldsymbol{\beta}}_{t}(\boldsymbol{\lambda}_{t-1})-\widehat{\boldsymbol{\beta}}_{t-1}(\boldsymbol{\lambda}_{t-1})\right\|^2+\frac{C_K\eta^2\sigma^2_{g_{\boldsymbol{\beta}}}}{s}.\nonumber
\end{align}
Now setting $\delta=\frac{\eta\ell_{g,1}\mu_g}{\ell_{g,1}+\mu_g}>0$ implies the upper bound of 
\begin{align}
    (1+\delta)\left(1-\frac{2\eta\ell_{g,1}\mu_g}{\ell_{g,1}+\mu_g}\right)^K<\left(1-\frac{\eta\ell_{g,1}\mu_g}{\ell_{g,1}+\mu_g}\right)\left(1-\frac{2\eta\ell_{g,1}\mu_g}{\ell_{g,1}+\mu_g}\right)^{K-1}<1\nonumber,
\end{align}
which defining $\nu:=\left(1-\frac{\eta\mu_g\ell_{g,1}}{\ell_{g,1}+\mu_g}\right)\left(1-\frac{2\eta\ell_{g,1}\mu_g}{\ell_{g,1}+\mu_g}\right)^{K-1}$ and $\delta>0$ implies
\begin{align}
    \left(1-\frac{2\eta\ell_{g,1}\mu_g}{\ell_{g,1}+\mu_g}\right)^K<\nu\nonumber,
\end{align}
Using the definition of $\nu$, we get
\begin{align}
\nu\mathbb{E}_{\bar{\zeta}_{t,K+1}}\left[\left\|\boldsymbol{\beta}_{t+1}-\widehat{\boldsymbol{\beta}}_{t}(\boldsymbol{\lambda}_{t})\right\|^2\right] \leq \nu^2\mathbb{E}_{\bar{\zeta}_{t-1,K+1}}\left[\left\|\boldsymbol{\beta}_{t}-\widehat{\boldsymbol{\beta}}_{t-1}(\boldsymbol{\lambda}_{t-1})\right\|^2\right]\nonumber\\+2C_{\mu_g}\nu^2\kappa_g^2\left\|\boldsymbol{\lambda}_{t-1}-\boldsymbol{\lambda}_{t}\right\|^2+2C_{\mu_g}\nu^2\left\|\widehat{\boldsymbol{\beta}}_{t}(\boldsymbol{\lambda}_{t-1})-\widehat{\boldsymbol{\beta}}_{t-1}(\boldsymbol{\lambda}_{t-1})\right\|^2+\frac{\nu C_K\eta^2\sigma^2_{g_{\boldsymbol{\beta}}}}{s},\nonumber
\end{align}
where $C_{\mu_g}=\left(1+\frac{\ell_{g,1}+\mu_g}{\eta\ell_{g,1}\mu_g}\right)$. 
Starting at $t=T$, and unrolling to $t=1$, we can write
\begin{align}
 \mathbb{E}_{\bar{\zeta}_{t,K+1}}\left[\left\|\boldsymbol{\beta}_{t+1}-\widehat{\boldsymbol{\beta}}_t(\boldsymbol{\lambda}_t)\right\|^2\right] \leq \nu^{t-1}\Delta_{\boldsymbol{\beta}}+2C_{\mu_g}\kappa_g^2\sum_{j=0}^{t-2}\nu^{j+1} \left[\left\|\boldsymbol{\lambda}_{t-1-j}-\boldsymbol{\lambda}_{t-j}\right\|^2\right]\nonumber\\ +2C_{\mu_g}\sum_{j=0}^{t-2}\nu^{j+1}\left[\left\|\widehat{\boldsymbol{\beta}}_{t-j}(\boldsymbol{\lambda}_{t-1-j})-\widehat{\boldsymbol{\beta}}_{t-1-j}(\boldsymbol{\lambda}_{t-1-j})\right\|^2\right]+\frac{C_K\eta^2\sigma^2_{g_{\boldsymbol{\beta}}}}{s} \sum_{j=0}^{t-2}\nu^{j}.\nonumber
\end{align}
\end{proof}

The next Lemma utilizes Lemma \ref{lem:hypergrad_bound} and Lemma \ref{lem:stochastic_inner_tracking_error} to derive an upper bound on the expected hypergradient error $\forall t\in[1, T]$ with respect to $\bar{\zeta}_{t,k}$ in terms of   discounted variations of the (i) cumulative time-smoothed hypergradient error; (ii) bilevel local regret;  and (iii) cumulative difference between optimal inner-level variables.  There is a term composed of a discounted initial error and smoothness term of the inner objective, as well as an additional term arising from the variance of the stochastic gradients of $g_t(\boldsymbol{\lambda},\boldsymbol{\beta},\zeta)$. 
\begin{lemma}\label{lem:stochastic_hypergradient_error}
   Suppose Assumptions \ref{assump:rel_smoothness}, \ref{assump:strong_convex_g_t}, \ref{assump:unbiased_finite_var}, \ref{assump:bounded_hyperparm}, and \ref{assump:continuity_phi_t}. Choose the inner step size of $\eta$  and inner iteration count $K$ as
\begin{align}
    0<\eta\leq\frac{2}{\ell_{g,1}+\mu_g}, \ \text{and} \quad K\geq 1.\nonumber
\end{align} With the definitions of $\nu, \ C_{\mu_g},\ \Delta_{\boldsymbol{\beta}},$ and $C_K$ from Lemma \ref{lem:stochastic_inner_tracking_error},  the expected hypergradient error can be bounded as
      \begin{align}
          \mathbb{E}_{\bar{\zeta}_{t,K+1}}\left[\left\|\widetilde{\nabla}f_{t}(\boldsymbol{\lambda}_t,\boldsymbol{\beta}_{t+1})-\nabla F_{t}\left(\boldsymbol{\lambda}_t\right)\right\|^2\right] \leq \delta_t+A\sum_{j=0}^{t-2}\nu^{j+1}\left\|\mathcal{G}_{\mathcal{X}}(\boldsymbol{\lambda}_{t-1-j},\nabla F_{t-1-j,w}(\boldsymbol{\lambda}_{t-1-j}),\alpha)\right\|^2\nonumber\\+B\sum_{j=0}^{t-2}\nu^{j+1}\left\|\widetilde{\nabla}f_{t-1-j,w}(\boldsymbol{\lambda}_{t-1-j},\boldsymbol{\beta}_{t-j},\mathcal{Z}_{t-1-j,w})-\nabla F_{t-1-j,w}(\boldsymbol{\lambda}_{t-1-j})\right\|^2\nonumber\\+C\sum_{j=0}^{t-2}\nu^{j+1}\left[\left\|\widehat{\boldsymbol{\beta}}_{t-j}(\boldsymbol{\lambda}_{t-1-j})-\widehat{\boldsymbol{\beta}}_{t-1-j}(\boldsymbol{\lambda}_{t-1-j})\right\|^2\right]+\frac{D\sigma^2_{g_{\boldsymbol{\beta}}}}{s}.\nonumber 
        \end{align}
        where $\delta_t=\kappa^2_g\nu^{t-1}\Delta_{\boldsymbol{\beta}}$ and $A=4C_{\mu_g}\kappa_g^4\alpha^2,B=\frac{4C_{\mu_g}\kappa_g^4\alpha^2}{\rho^2}$, $C=2C_{\mu_g}\kappa^2_g$, and $D=C_K\kappa^2_g\eta^2 \sum_{j=0}^{t-2}\nu^{j}$.
  
\end{lemma}
\begin{proof}[Proof of Lemma \ref{lem:stochastic_hypergradient_error}]
        First,  from Lemma \ref{lem:hypergrad_bound} we have that $\forall \boldsymbol{\lambda} \in \mathcal{X}$ and $\boldsymbol{\beta}\in\mathbb{R}^{d_2}$
        \begin{align}\label{eq:hypergradient_bias_equation}
               \left\|\widetilde{\nabla}f_{t}(\boldsymbol{\lambda}_t,\boldsymbol{\beta}_{t+1})-\nabla F_{t}\left(\boldsymbol{\lambda}_t\right)\right\|^2 \leq \kappa^2_g\left\|\boldsymbol{\beta}_{t+1}-\widehat{\boldsymbol{\beta}}_t(\boldsymbol{\lambda}_t)\right\|^2.
        \end{align}
        Taking expectation of \eqref{eq:hypergradient_bias_equation} with respect to $\bar{\zeta}_{t,K+1}$ and substituting the upper bound of Lemma \ref{lem:stochastic_inner_tracking_error}, note
        \begin{align}
               \mathbb{E}_{\bar{\zeta}_{t,K+1}}\left[\left\|\widetilde{\nabla}f_{t}(\boldsymbol{\lambda}_t,\boldsymbol{\beta}_{t+1})-\nabla F_{t}\left(\boldsymbol{\lambda}_t\right)\right\|^2\right]\nonumber\\\label{eq:term_to_focus_stochastic_hg}\leq \kappa^2_g\left(\nu^{t-1}\Delta_{\boldsymbol{\beta}}+2C_{\mu_g}\kappa_g^2\sum_{j=0}^{t-2}\nu^{j+1} \left\|\boldsymbol{\lambda}_{t-1-j}-\boldsymbol{\lambda}_{t-j}\right\|^2\right)\\ +\kappa^2_g\left(2C_{\mu_g}\sum_{j=0}^{t-2}\nu^{j+1}\left\|\widehat{\boldsymbol{\beta}}_{t-j}(\boldsymbol{\lambda}_{t-1-j})-\widehat{\boldsymbol{\beta}}_{t-1-j}(\boldsymbol{\lambda}_{t-1-j})\right\|^2+\frac{C_K\eta^2\sigma^2_{g_{\boldsymbol{\beta}}}}{s} \sum_{j=0}^{t-2}\nu^{j}\right).
        \end{align}
        Focusing on the second term of \eqref{eq:term_to_focus_stochastic_hg} we see by definition 
        \begin{align}\label{eq:gradient_step_equality}
            \sum_{j=0}^{t-2}\nu^{j+1} \left\|\boldsymbol{\lambda}_{t-1-j}-\boldsymbol{\lambda}_{t-j}\right\|^2\nonumber\\=\sum_{j=0}^{t-2}\nu^{j+1} \alpha^2\left\|\mathcal{G}_{\mathcal{X}}(\boldsymbol{\lambda}_{t-1-j},\widetilde{\nabla} f_{t-1-j,\boldsymbol{w}} (\boldsymbol{\lambda}_{t-1-j},\boldsymbol{\beta}_{t-j},\mathcal{Z}_{t-1-j,w}),\alpha)\right\|^2.
        \end{align}
        Using Lemma \ref{lem:gen_projection_bound_two} we have  $\forall j\in [0,t-2]$
        \begin{align}
            \left\|\mathcal{G}_{\mathcal{X}}(\boldsymbol{\lambda}_{t-1-j},\widetilde{\nabla} f_{t-1-j,\boldsymbol{w}} (\boldsymbol{\lambda}_{t-1-j},\boldsymbol{\beta}_{t-j},\mathcal{Z}_{t-1-j,w}),\alpha)\right\|^2\leq 2\left\|\mathcal{G}_{\mathcal{X}}(\boldsymbol{\lambda}_{t-1-j},\nabla F_{t-1-j,w}(\boldsymbol{\lambda}_{t-1-j}),\alpha)\right\|^2\nonumber\\+2\left\|\mathcal{G}_{\mathcal{X}}(\boldsymbol{\lambda}_{t-1-j},\nabla F_{t-1-j,w}(\boldsymbol{\lambda}_{t-1-j}),\alpha)-\mathcal{G}_{\mathcal{X}}(\boldsymbol{\lambda}_{t-1-j},\widetilde{\nabla} f_{t-1-j,\boldsymbol{w}} (\boldsymbol{\lambda}_{t-1-j},\boldsymbol{\beta}_{t-j},\mathcal{Z}_{t-1-j,w}),\alpha)\right\|^2\nonumber\\\leq2\left\|\mathcal{G}_{\mathcal{X}}(\boldsymbol{\lambda}_{t-1-j},\nabla F_{t-1-j,w}(\boldsymbol{\lambda}_{t-1-j}),\alpha)\right\|^2\nonumber\\+\frac{2}{\rho^2}\left\|\widetilde{\nabla}f_{t-1-j,w}(\boldsymbol{\lambda}_{t-1-j},\boldsymbol{\beta}_{t-j},\mathcal{Z}_{t-1-j,w})-\nabla F_{t-1-j,w}(\boldsymbol{\lambda}_{t-1-j})\right\|^2\nonumber.
        \end{align}
        We can write an upper bound to \eqref{eq:gradient_step_equality} as
        \begin{align}\label{eq:bregman_stochastic_hg_decomposition}
        \sum_{j=0}^{t-2}\nu^{j+1} \left\|\boldsymbol{\lambda}_{t-1-j}-\boldsymbol{\lambda}_{t-j}\right\|^2\leq 2\alpha^2\sum_{j=0}^{t-2}\nu^{j+1}\left\|\mathcal{G}_{\mathcal{X}}(\boldsymbol{\lambda}_{t-1-j},\nabla F_{t-1-j,w}(\boldsymbol{\lambda}_{t-1-j}),\alpha)\right\|^2\nonumber\\+\frac{2\alpha^2}{\rho^2}\sum_{j=0}^{t-2}\nu^{j+1}\left(\left\|\widetilde{\nabla}f_{t-1-j,w}(\boldsymbol{\lambda}_{t-1-j},\boldsymbol{\beta}_{t-j},\mathcal{Z}_{t-1-j,w})-\nabla F_{t-1-j,w}(\boldsymbol{\lambda}_{t-1-j})\right\|^2\right).
        \end{align}
        Using \eqref{eq:bregman_stochastic_hg_decomposition}, we get
      \begin{align}
          \mathbb{E}_{\bar{\zeta}_{t,K+1}}\left[\left\|\widetilde{\nabla}f_{t}(\boldsymbol{\lambda}_t,\boldsymbol{\beta}_{t+1})-\nabla F_{t}\left(\boldsymbol{\lambda}_t\right)\right\|^2\right] \leq \delta_t+A\sum_{j=0}^{t-2}\nu^{j+1}\left\|\mathcal{G}_{\mathcal{X}}(\boldsymbol{\lambda}_{t-1-j},\nabla F_{t-1-j,w}(\boldsymbol{\lambda}_{t-1-j}),\alpha)\right\|^2\nonumber\\+B\sum_{j=0}^{t-2}\nu^{j+1}\left\|\widetilde{\nabla}f_{t-1-j,w}(\boldsymbol{\lambda}_{t-1-j},\boldsymbol{\beta}_{t-j},\mathcal{Z}_{t-1-j,w})-\nabla F_{t-1-j,w}(\boldsymbol{\lambda}_{t-1-j})\right\|^2\nonumber\\+C\sum_{j=0}^{t-2}\nu^{j+1}\left[\left\|\widehat{\boldsymbol{\beta}}_{t-j}(\boldsymbol{\lambda}_{t-1-j})-\widehat{\boldsymbol{\beta}}_{t-1-j}(\boldsymbol{\lambda}_{t-1-j})\right\|^2\right]+\frac{D\sigma^2_{g_{\boldsymbol{\beta}}}}{s}.\nonumber 
        \end{align}
        where $\delta_t=\kappa^2_g\nu^{t-1}\Delta_{\boldsymbol{\beta}}$ and $A=4C_{\mu_g}\kappa_g^4\alpha^2,B=\frac{4C_{\mu_g}\kappa_g^4\alpha^2}{\rho^2}$, $C=2C_{\mu_g}\kappa^2_g$, and $D=C_K\kappa^2_g\eta^2 \sum_{j=0}^{t-2}\nu^{j}$.
  
    \end{proof}
 Lemma \ref{lem:stochastic_cumulative_hypergradient_error} provides an upper bound on the expected cumulative time-smoothed hypergradient error in terms of an initial error, expected bilevel local regret, expected cumulative differences of optimal inner level variables, as well as variance terms from the stochastic approximated gradients.
\begin{lemma}\label{lem:stochastic_cumulative_hypergradient_error}
Suppose Assumptions \ref{assump:rel_smoothness}, \ref{assump:strong_convex_g_t}, \ref{assump:unbiased_finite_var}, \ref{assump:bounded_hyperparm}, and \ref{assump:continuity_phi_t}. Choose the inner step size of $\eta$, the inner iteration count $K$, and the outer step size $\alpha$ respectively as
\begin{align}
    0<\eta\leq\frac{2}{\ell_{g,1}+\mu_g}, \ \quad K\geq 1, \  \text{and} \quad \alpha<\frac{\rho\sqrt{(1-\nu)}}{\kappa^2_g\sqrt{72C_{\mu_g}}}.\nonumber
\end{align} Then  $\forall t\in[1,T]$  the expected cumulative time-smoothed hypergradient error with respect to independent samples $Z_{t,w}$ from  SOBBO satisfies
    \begin{align}
\mathbb{E}_{Z_{t,w}}\left[\sum_{t=1}^T\left\|\widetilde{\nabla}f_{t,w}(\boldsymbol{\lambda}_t,\boldsymbol{\beta}_{t+1},\mathcal{Z}_{t,w})-\nabla F_{t,w}\left(\boldsymbol{\lambda}_t\right)\right\|^2\right]  \leq\frac{9T\sigma^2_f}{2w}+\frac{9T\ell^2_{f,1}\kappa^2_g}{2}\left(1-\frac{\mu_g}{\ell_{g,1}}\right)^{2m} \nonumber\\+\frac{9\kappa^2_g\Delta_{\boldsymbol{\beta}}}{4(1-\nu)}+\frac{\rho^2}{8}\sum_{t=1}^T\mathbb{E}\left\|\mathcal{G}_{\mathcal{X}}(\boldsymbol{\lambda}_{t},\nabla F_{t,w}(\boldsymbol{\lambda}_{t}),\alpha)\right\|^2\nonumber\\+\frac{9C_{\mu_g}\kappa^2_g}{2(1-\nu)}\sum_{t=2}^T\mathbb{E}\left[\left\|\widehat{\boldsymbol{\beta}}_{t}(\boldsymbol{\lambda}_{t-1})-\widehat{\boldsymbol{\beta}}_{t-1}(\boldsymbol{\lambda}_{t-1})\right\|^2\right]+\frac{9TC_K\kappa^2_g\eta^2\sigma^2_{g_{\boldsymbol{\beta}}}}{4s(1-\nu)}.\nonumber
\end{align}
\end{lemma}
\begin{proof}[Proof of Lemma \ref{lem:stochastic_cumulative_hypergradient_error}]
    With the linearity of expectation and by definition of \eqref{eq:stoch_time_smooth_hypergradient} and \eqref{eq:det_sobow_local_regret} we have
\begin{align}\label{eq:lin_expectation_cumulative_hg_error}
\mathbb{E}_{Z_{t,w}}\left[\sum_{t=1}^T\left\|\widetilde{\nabla}f_{t,w}(\boldsymbol{\lambda}_t,\boldsymbol{\beta}_{t+1},\mathcal{Z}_{t,w})-\nabla F_{t,w}\left(\boldsymbol{\lambda}_t\right)\right\|^2\right] \nonumber\\=\sum_{t=1}^T\mathbb{E}_{Z_{t,w}}\left[\left\|\widetilde{\nabla}f_{t,w}(\boldsymbol{\lambda}_t,\boldsymbol{\beta}_{t+1},\mathcal{Z}_{t,w})-\nabla F_{t,w}\left(\boldsymbol{\lambda}_t\right)\right\|^2\right]  \nonumber\\=\frac{1}{w^2}\sum_{t=1}^T\mathbb{E}_{\mathcal{Z}_{t,w}}\left[\left\|\sum_{i=0}^{w-1}\left[\widetilde{\nabla}f_{t-i}(\boldsymbol{\lambda}_{t-i},\boldsymbol{\beta}_{t+1-i},\mathcal{E}_{t-i})-\nabla F_{t-i}\left(\boldsymbol{\lambda}_{t-i}\right)\right]\right\|^2\right].
    \end{align}
Note that we can upper bound \eqref{eq:lin_expectation_cumulative_hg_error} as
\begin{align}
    \frac{1}{w^2}\sum_{t=1}^T\mathbb{E}_{\mathcal{Z}_{t,w}}\left[\left\|\sum_{i=0}^{w-1}\left[\widetilde{\nabla}f_{t-i}(\boldsymbol{\lambda}_{t-i},\boldsymbol{\beta}_{t+1-i},\mathcal{E}_{t-i})-\nabla F_{t-i}\left(\boldsymbol{\lambda}_{t-i}\right)\right]\right\|^2\right] \nonumber\\ \label{eq:variance_sgd_hg_estimate} \leq    \frac{2}{w^2}\sum_{t=1}^T\mathbb{E}_{\mathcal{Z}_{t,w}}\left[\left\|\sum_{i=0}^{w-1}\left[\widetilde{\nabla}f_{t-i}(\boldsymbol{\lambda}_{t-i},\boldsymbol{\beta}_{t+1-i},\mathcal{E}_{t-i})-\mathbb{E}_{\mathcal{E}_{t-i}}\left[\widetilde{\nabla}f_{t-i}(\boldsymbol{\lambda}_{t-i},\boldsymbol{\beta}_{t+1-i},\mathcal{E}_{t-i})\right]\right]\right\|^2\right]\\\label{eq:hypergradient_bias_cumulative_hg_error}+    \frac{2}{w^2}\sum_{t=1}^T\mathbb{E}_{\mathcal{Z}_{t,w}}\left[\left\|\sum_{i=0}^{w-1}\left[\mathbb{E}_{\mathcal{E}_{t-i}}\left[\widetilde{\nabla}f_{t-i}(\boldsymbol{\lambda}_{t-i},\boldsymbol{\beta}_{t+1-i},\mathcal{E}_{t-i})\right]-\nabla F_{t-i}\left(\boldsymbol{\lambda}_{t-i}\right)\right]\right\|^2\right].
\end{align}
The linearity of expectation,  definition of variance, and  independence of $Z_{t,w}:=\left\{\mathcal{E}_{t-i} \right\}_{i=0}^{w-1} \ \forall t\in[1,T]$ implies for $y_i=\widetilde{\nabla}f_{t-i}(\boldsymbol{\lambda}_{t-i},\boldsymbol{\beta}_{t+1-i},\mathcal{E}_{t-i})$ with finite variance $\sigma^2_f$, we have
\begin{align}\label{eq:variance_defn_stochastic_hg_error}
\mathbb{E}_{\mathcal{Z}_{t,w}}\left[\left\|\sum_{i=0}^{w-1}\widetilde{\nabla}f_{t-i}(\boldsymbol{\lambda}_{t-i},\boldsymbol{\beta}_{t+1-i},\mathcal{E}_{t-i})-\mathbb{E}_{\mathcal{E}_{t-i}}\left[\sum_{i=0}^{w-1}\widetilde{\nabla}f_{t-i}(\boldsymbol{\lambda}_{t-i},\boldsymbol{\beta}_{t+1-i},\mathcal{E}_{t-i})\right]\right\|^2\right]\nonumber\\\leq \sum_{i=0}^{w-1}\sigma^2_f=w\sigma^2_f.
\end{align}
Expanding \eqref{eq:hypergradient_bias_cumulative_hg_error} we have
\begin{align}
\frac{2}{w^2}\sum_{t=1}^T\mathbb{E}_{\mathcal{Z}_{t,w}}\left[\left\|\sum_{i=0}^{w-1}\left[\mathbb{E}_{\mathcal{E}_{t-i}}\left[\widetilde{\nabla}f_{t-i}(\boldsymbol{\lambda}_{t-i},\boldsymbol{\beta}_{t+1-i},\mathcal{E}_{t-i})\right]-\nabla F_{t-i}\left(\boldsymbol{\lambda}_{t-i}\right)\right]\right\|^2\right]\nonumber\\\leq \frac{4}{w^2}\sum_{t=1}^T\mathbb{E}_{\mathcal{Z}_{t,w}}\left[\left\|\sum_{i=0}^{w-1}\left[\mathbb{E}_{\mathcal{E}_{t-i}}\left[\widetilde{\nabla}f_{t-i}(\boldsymbol{\lambda}_{t-i},\boldsymbol{\beta}_{t+1-i},\mathcal{E}_{t-i})\right]-\widetilde{\nabla}f_{t-i}(\boldsymbol{\lambda}_{t-i},\boldsymbol{\beta}_{t+1-i})\right]\right\|^2\right] \label{eq:bias_expansion_appendix}\\+\frac{4}{w^2}\sum_{t=1}^T\left[\left\|\sum_{i=0}^{w-1}\left[\widetilde{\nabla}f_{t-i}(\boldsymbol{\lambda}_{t-i},\boldsymbol{\beta}_{t+1-i})-\nabla F_{t-i}\left(\boldsymbol{\lambda}_{t-i}\right)\right]\right\|^2\right]\label{eq:expected_surrogat_grad_error}
\end{align}

Utilizing Lemmas \ref{lem:vector_ub} and \ref{lem:stochastic_gradient_estimate_bias}for \eqref{eq:bias_expansion_appendix} gives us the expected stochastic gradient bias 
\begin{align}
   \frac{4}{w^2}\sum_{t=1}^T\mathbb{E}_{\mathcal{Z}_{t,w}}\left[\left\|\sum_{i=0}^{w-1}\left[\mathbb{E}_{\mathcal{E}_{t-i}}\left[\widetilde{\nabla}f_{t-i}(\boldsymbol{\lambda}_{t-i},\boldsymbol{\beta}_{t+1-i},\mathcal{E}_{t-i})\right]-\widetilde{\nabla}f_{t-i}(\boldsymbol{\lambda}_{t-i},\boldsymbol{\beta}_{t+1-i})\right]\right\|^2\right] \nonumber\\\leq\frac{4}{w^2}\sum_{t=1}^T\mathbb{E}_{\mathcal{Z}_{t,w}}\left[w\sum_{i=0}^{w-1}\left\|\mathbb{E}_{\mathcal{E}_{t-i}}\left[\widetilde{\nabla}f_{t-i}(\boldsymbol{\lambda}_{t-i},\boldsymbol{\beta}_{t+1-i},\mathcal{E}_{t-i})\right]-\widetilde{\nabla}f_{t-i}(\boldsymbol{\lambda}_{t-i},\boldsymbol{\beta}_{t+1-i})\right\|^2\right] \nonumber\\\leq \frac{4}{w^2}\sum_{t=1}^T\left(w^2\ell^2_{f,1}\kappa^2_g\left(1-\frac{\mu_g}{\ell_{g,1}}\right)^{2m}\right)=4T\ell^2_{f,1}\kappa^2_g\left(1-\frac{\mu_g}{\ell_{g,1}}\right)^{2m}
\end{align}
Applying Lemma \ref{lem:vector_ub} with linearity of expectation to \eqref{eq:expected_surrogat_grad_error} results in
\begin{align}\label{eq:inner_tracking_decomp}
    \frac{4}{w^2}\sum_{t=1}^T\left[\left\|\sum_{i=0}^{w-1}\left[\widetilde{\nabla}f_{t-i}(\boldsymbol{\lambda}_{t-i},\boldsymbol{\beta}_{t+1-i})-\nabla F_{t-i}\left(\boldsymbol{\lambda}_{t-i}\right)\right]\right\|^2\right]\nonumber\\ \leq    \frac{4}{w^2}\sum_{t=1}^Tw\sum_{i=0}^{w-1}\left[\left\|\widetilde{\nabla}f_{t-i}(\boldsymbol{\lambda}_{t-i},\boldsymbol{\beta}_{t+1-i})-\nabla F_{t-i}\left(\boldsymbol{\lambda}_{t-i}\right)\right\|^2\right] \nonumber\\= \frac{4}{w}\sum_{t=1}^T\sum_{i=0}^{w-1}\left[\left\|\widetilde{\nabla}f_{t-i}(\boldsymbol{\lambda}_{t-i},\boldsymbol{\beta}_{t+1-i})-\nabla F_{t-i}\left(\boldsymbol{\lambda}_{t-i}\right)\right\|^2\right]  
\end{align}

Combining 
\eqref{eq:variance_sgd_hg_estimate}, \eqref{eq:hypergradient_bias_cumulative_hg_error}, \eqref{eq:variance_defn_stochastic_hg_error}, and \eqref{eq:inner_tracking_decomp},  we have the upper bound of 
\begin{align}\label{eq:second_cumulative_hg_decomposition}
\mathbb{E}_{\mathcal{Z}_{t,w}}\left[\sum_{t=1}^T\left\|\widetilde{\nabla}f_{t,w}(\boldsymbol{\lambda}_t,\boldsymbol{\beta}_{t+1},\mathcal{Z}_{t,w})-\nabla F_{t,w}\left(\boldsymbol{\lambda}_t\right)\right\|^2\right] \leq   \frac{4T\sigma^2_f}{w}\nonumber\\+  4T\ell^2_{f,1}\kappa^2_g\left(1-\frac{\mu_g}{\ell_{g,1}}\right)^{2m}  + \frac{4}{w}\sum_{t=1}^T\sum_{i=0}^{w-1}\left[\left\|\widetilde{\nabla}f_{t-i}(\boldsymbol{\lambda}_{t-i},\boldsymbol{\beta}_{t+1-i})-\nabla F_{t-i}\left(\boldsymbol{\lambda}_{t-i}\right)\right\|^2\right] ,
\end{align}

Taking expectation with respect to $\bar{\zeta}_{t,K+1}$, we utilize the upper bound of $\mathbb{E}_{\bar{\zeta}_{t,K+1}}\left[\left\|\widetilde{\nabla}f_{t}(\boldsymbol{\lambda}_t,\boldsymbol{\beta}_{t+1})-\nabla F_{t}\left(\boldsymbol{\lambda}_t\right)\right\|^2\right]$ from Lemma \ref{lem:stochastic_hypergradient_error}. By iterative conditioning and re-indexing the expected cumulative hypergradient error as well as dropping expectation for non-random quantities, we derive an upper bound on \eqref{eq:second_cumulative_hg_decomposition}   as
    \begin{align}\label{eq:third_cumulative_hg_decompositon}
 \mathbb{E}_{\bar{\zeta}_{t,K+1}}\left[\mathbb{E}_{\mathcal{Z}_{t,w}}\left[\sum_{t=1}^T\left\|\widetilde{\nabla}f_{t,w}(\boldsymbol{\lambda}_t,\boldsymbol{\beta}_{t+1},\mathcal{Z}_{t,w})-\nabla F_{t,w}\left(\boldsymbol{\lambda}_t\right)\right\|^2\right]\right] \nonumber\\ \leq   \frac{4T\sigma^2_f}{w}+ 4T\ell^2_{f,1}\kappa^2_g\left(1-\frac{\mu_g}{\ell_{g,1}}\right)^{2m} +   \frac{2}{w}\sum_{t=1}^T\sum_{i=0}^{w-1}\left(\kappa^2_g\nu^{t-i-1}\Delta_{\boldsymbol{\beta}}\right)\nonumber\\+\frac{2}{w}\sum_{t=1}^T\sum_{i=0}^{w-1}\left(4C_{\mu_g}\kappa_g^4\alpha^2\sum_{j=0}^{t-i-2}\nu^{j+1}\left\|\mathcal{G}_{\mathcal{X}}(\boldsymbol{\lambda}_{t-i-j},\nabla F_{t-i-j,w}(\boldsymbol{\lambda}_{t-i-j}),\alpha)\right\|^2\right)\nonumber\\+    \frac{2}{w}\sum_{t=1}^T\sum_{i=0}^{w-1}\left(\frac{4C_{\mu_g}\kappa_g^4\alpha^2}{\rho^2}\sum_{j=0}^{t-i-2}\nu^{j+1}A_{t-i,j}\right)\nonumber\\+    \frac{2}{w}\sum_{t=2}^T\sum_{i=0}^{w-1}\left(2C_{\mu_g}\kappa^2_g\sum_{j=0}^{t-i-2}\nu^{j+1}B_{t-i,j}\right)+ \frac{2}{w}\sum_{t=1}^T\sum_{i=0}^{w-1}\left(\frac{C_K\kappa^2_g\eta^2\sigma^2_{g_{\boldsymbol{\beta}}}}{s} \sum_{j=0}^{t-i-2}\nu^{j}\right),
    \end{align}
    where \begin{align}
A_{t,j}:=\mathbb{E}_{\bar{\zeta}_{t-j,K+1}}\left[\mathbb{E}_{\mathcal{Z}_{t-j,w}}\left[\left\|\widetilde{\nabla}f_{t-j,w}(\boldsymbol{\lambda}_{t-j},\boldsymbol{\beta}_{t+1-j},\mathcal{Z}_{t-j,w})-\nabla F_{t-j,w}(\boldsymbol{\lambda}_{t-j})\right\|^2\right]\right]\nonumber\\B_{t,j}:=\left\|\widehat{\boldsymbol{\beta}}_{t-j}(\boldsymbol{\lambda}_{t-1-j})-\widehat{\boldsymbol{\beta}}_{t-1-j}(\boldsymbol{\lambda}_{t-1-j})\right\|^2.\nonumber
    \end{align}
Given $\nu<1$, it holds that $\sum_{j=0}^{t-2}\nu^j<\sum_{j=0}^{\infty}\nu^j=\frac{1}{1-\nu}$, which lets us upper bound \eqref{eq:third_cumulative_hg_decompositon} as
    \begin{align}\label{eq:fourth_cumulative_hg_decomposition}
\mathbb{E}_{\bar{\zeta}_{t,K+1}}\left[\mathbb{E}_{\mathcal{Z}_{t,w}}\left[\sum_{t=1}^T\left\|\widetilde{\nabla}f_{t,w}(\boldsymbol{\lambda}_t,\boldsymbol{\beta}_{t+1},\mathcal{Z}_{t,w})-\nabla F_{t,w}\left(\boldsymbol{\lambda}_t\right)\right\|^2\right]\right]  \leq   \frac{4T\sigma^2_f}{w}+4T\ell^2_{f,1}\kappa^2_g\left(1-\frac{\mu_g}{\ell_{g,1}}\right)^{2m} \nonumber\\+    \frac{2}{w}\sum_{t=1}^T\sum_{i=0}^{w-1}\left(\kappa^2_g\nu^{t-i-1}\Delta_{\boldsymbol{\beta}}+\frac{4C_{\mu_g}\kappa_g^4\alpha^2}{1-\nu}\left\|\mathcal{G}_{\mathcal{X}}(\boldsymbol{\lambda}_{t-i},\nabla F_{t-i,w}(\boldsymbol{\lambda}_{t-i}),\alpha)\right\|^2\right)\nonumber\\+    \frac{2}{w}\sum_{t=1}^T\sum_{i=0}^{w-1}\left(\frac{4C_{\mu_g}\kappa_g^4\alpha^2}{(1-\nu)\rho^2}\mathbb{E}_{\bar{\zeta}_{t-i,K+1}}\left[\mathbb{E}_{\mathcal{Z}_{t-i,w}}\left[\left\|\widetilde{\nabla}f_{t-i,w}(\boldsymbol{\lambda}_{t-i},\boldsymbol{\beta}_{t+1-i},\mathcal{Z}_{t-i,w})-\nabla F_{t-i,w}(\boldsymbol{\lambda}_{t-i})\right\|^2\right]\right]\right)\nonumber\\+    \frac{2}{w}\sum_{t=2}^T\sum_{i=0}^{w-1}\left(\frac{2C_{\mu_g}\kappa^2_g}{1-\nu}\left\|\widehat{\boldsymbol{\beta}}_{t-i}(\boldsymbol{\lambda}_{t-1-i})-\widehat{\boldsymbol{\beta}}_{t-1-i}(\boldsymbol{\lambda}_{t-1-i})\right\|^2\right)+\frac{2}{w}\sum_{t=1}^T\sum_{i=0}^{w-1}\left(\frac{C_K\kappa^2_g\eta^2\sigma^2_{g_{\boldsymbol{\beta}}}}{(1-\nu)s} \right).
    \end{align}
Next  we derive the upper bound of \eqref{eq:fourth_cumulative_hg_decomposition} as
\begin{align}
\mathbb{E}_{\bar{\zeta}_{t,K+1}}\left[\mathbb{E}_{\mathcal{Z}_{t,w}}\left[\sum_{t=1}^T\left\|\widetilde{\nabla}f_{t,w}(\boldsymbol{\lambda}_t,\boldsymbol{\beta}_{t+1},\mathcal{Z}_{t,w})-\nabla F_{t,w}\left(\boldsymbol{\lambda}_t\right)\right\|^2\right]\right]    \leq   \frac{4T\sigma^2_f}{w} +4T\ell^2_{f,1}\kappa^2_g\left(1-\frac{\mu_g}{\ell_{g,1}}\right)^{2m} \nonumber\\+\frac{2\kappa^2_g\Delta_{\boldsymbol{\beta}}}{(1-\nu)}+\frac{8C_{\mu_g}\kappa_g^4\alpha^2}{(1-\nu)}\sum_{t=1}^T\left\|\mathcal{G}_{\mathcal{X}}(\boldsymbol{\lambda}_{t},\nabla F_{t,w}(\boldsymbol{\lambda}_{t}),\alpha)\right\|^2\nonumber\\+\frac{8C_{\mu_g}\kappa_g^4\alpha^2}{\rho^2(1-\nu)}\mathbb{E}_{\bar{\zeta}_{t,K+1}}\left[\mathbb{E}_{\mathcal{Z}_{t,w}}\left[\sum_{t=1}^T\left\|\widetilde{\nabla}f_{t,w}\left(\boldsymbol{\lambda}_{t},\boldsymbol{\beta}_{t+1},\mathcal{Z}_{t,w})-\nabla F_{t,w}(\boldsymbol{\lambda}_{t}\right)\right\|^2\right]\right]\nonumber\\ +\frac{4C_{\mu_g}\kappa^2_g}{(1-\nu)}\sum_{t=2}^T\left\|\widehat{\boldsymbol{\beta}}_{t}(\boldsymbol{\lambda}_{t-1})-\widehat{\boldsymbol{\beta}}_{t-1}(\boldsymbol{\lambda}_{t-1})\right\|^2+\frac{2TC_K\kappa^2_g\eta^2\sigma^2_{g_{\boldsymbol{\beta}}}}{s(1-\nu)}.\nonumber
\end{align}
which implies through linearity of expectation that 
\begin{align}
    \left(1-\frac{8C_{\mu_g}\kappa_g^4\alpha^2}{\rho^2(1-\nu)}\right)\mathbb{E}_{\bar{\zeta}_{t,K+1}}\left[\mathbb{E}_{\mathcal{Z}_{t,w}}\left[\sum_{t=1}^T\left\|\widetilde{\nabla}f_{t,w}(\boldsymbol{\lambda}_t,\boldsymbol{\beta}_{t+1},\mathcal{Z}_{t,w})-\nabla F_{t,w}\left(\boldsymbol{\lambda}_t\right)\right\|^2\right]\right]  \nonumber\\\leq \frac{4T\sigma^2_f}{w}+4T\ell^2_{f,1}\kappa^2_g\left(1-\frac{\mu_g}{\ell_{g,1}}\right)^{2m} +\frac{2\kappa^2_g\Delta_{\boldsymbol{\beta}}}{(1-\nu)}+\frac{8C_{\mu_g}\kappa_g^4\alpha^2}{(1-\nu)}\sum_{t=1}^T\left\|\mathcal{G}_{\mathcal{X}}(\boldsymbol{\lambda}_{t},\nabla F_{t,w}(\boldsymbol{\lambda}_{t}),\alpha)\right\|^2\nonumber\\+\frac{4C_{\mu_g}\kappa^2_g}{(1-\nu)}\sum_{t=2}^T\left\|\widehat{\boldsymbol{\beta}}_{t}(\boldsymbol{\lambda}_{t-1})-\widehat{\boldsymbol{\beta}}_{t-1}(\boldsymbol{\lambda}_{t-1})\right\|^2+\frac{2TC_K\kappa^2_g\eta^2\sigma^2_{g_{\boldsymbol{\beta}}}}{s(1-\nu)}.\nonumber
\end{align}

As   $0<\alpha\leq \frac{\rho\sqrt{(1-\nu)}}{\kappa^2_g\sqrt{72C_{\mu_g}}}$ 
\begin{align}
   \left(1-\frac{8C_{\mu_g}\kappa_g^4}{\rho^2(1-\nu)}\right)\geq\frac{8}{9},\nonumber
\end{align}
we have the upper bound of
\begin{align}
\mathbb{E}_{\bar{\zeta}_{t,K+1}}\left[\mathbb{E}_{\mathcal{Z}_{t,w}}\left[\sum_{t=1}^T\left\|\widetilde{\nabla}f_{t,w}(\boldsymbol{\lambda}_t,\boldsymbol{\beta}_{t+1},\mathcal{Z}_{t,w})-\nabla F_{t,w}\left(\boldsymbol{\lambda}_t\right)\right\|^2\right]\right]  \nonumber\\\leq\frac{9T\sigma^2_f}{2w}+\frac{9T\ell^2_{f,1}\kappa^2_g}{2}\left(1-\frac{\mu_g}{\ell_{g,1}}\right)^{2m} +\frac{9\kappa^2_g\Delta_{\boldsymbol{\beta}}}{4(1-\nu)}+\frac{\rho^2}{8}\sum_{t=1}^T\left\|\mathcal{G}_{\mathcal{X}}(\boldsymbol{\lambda}_{t},\nabla F_{t,w}(\boldsymbol{\lambda}_{t}),\alpha)\right\|^2\nonumber\\+\frac{9C_{\mu_g}\kappa^2_g}{2(1-\nu)}\sum_{t=2}^T\left[\left\|\widehat{\boldsymbol{\beta}}_{t}(\boldsymbol{\lambda}_{t-1})-\widehat{\boldsymbol{\beta}}_{t-1}(\boldsymbol{\lambda}_{t-1})\right\|^2\right]+\frac{9TC_K\kappa^2_g\eta^2\sigma^2_{g_{\boldsymbol{\beta}}}}{4s(1-\nu)}.\nonumber
\end{align}
    \end{proof}
The following theorem presents the proof of SOBBO achieving a sublinear rate of bilevel local regret.
\begin{theorem}\label{thrm:sobbo_regret_minimization}
    Suppose Assumptions \ref{assump:rel_smoothness}, \ref{assump:strong_convex_g_t}, \ref{assump:unbiased_finite_var}, \ref{assump:bounded_hyperparm}, 
  \ref{assump:bounded_F_t}, and \ref{assump:continuity_phi_t}. Choose the inner step size of $\eta$, the inner iteration count of $K$, the outer step size of $\alpha$, and the batch sizes of $s$ and $ m$ to respectively satisfy
  \begin{align}
     \eta \leq\frac{2}{\ell_{g,1}+\mu_g}, \ \quad  K\geq 1, \ \quad  \alpha\leq \min\left\{\frac{3\rho}{4\ell_{F,1}},\frac{\rho\sqrt{(1-\nu)}}{\kappa^2_g\sqrt{72C_{\mu_g}}}\right\}, \nonumber\\ s=w, \quad \text{and} \  \quad m=\log{(w)}/\log{\left(1-\frac{\mu_g}{\ell_{g,1}}\right)}+1.
  \end{align}Then the expected bilevel local regret of our SOBBO Algorithm \ref{alg:stochastic_online_bregman} satisfies
\begin{align}
   BLR_w(T):= \sum_{t=1}^T\left\|\mathcal{G}_{\mathcal{X}}(\boldsymbol{\lambda}_{t},\nabla F_{t,w}(\boldsymbol{\lambda}_{t}),\alpha)\right\|^2\nonumber\\\leq O\left(\frac{T}{w}\left(1+\sigma^2_f+\kappa^2_g\sigma^2_{g_{\boldsymbol{\beta}}}\right)+V_{1,T}+\kappa^2_gH_{2,T}\right)
\end{align}
which is a sublinear rate of bilevel local regret  when the regularity constraints of  $V_{1,T}=o(T)$ and $H_{2,T}=o(T)$ are imposed.
\end{theorem}
\begin{proof}[Proof of Theorem \ref{thrm:sobbo_regret_minimization}]
Note, with Assumption \ref{assump:rel_smoothness}  we have the upper bound of
\begin{align}
F_{t,w}\left(\boldsymbol{\lambda}_{t+1}\right)- F_{t,w}\left(\boldsymbol{\lambda}_{t}\right) =\frac{1}{w}\sum_{i=0}^{w-1}F_{t-i}\left(\boldsymbol{\lambda}_{t+1-i}\right)-\frac{1}{w}\sum_{i=0}^{w-1}F_{t-i}\left(\boldsymbol{\lambda}_{t-i}\right)\nonumber\\=\frac{1}{w}\sum_{i=0}^{w-1}\left[F_{t-i}\left(\boldsymbol{\lambda}_{t+1-i}\right)-F_{t-i}\left(\boldsymbol{\lambda}_{t-i}\right)\right]\nonumber\\\leq \frac{1}{w}\sum_{i=0}^{w-1}\left[\left\langle\nabla F_{t-i}\left(\boldsymbol{\lambda}_{t-i}\right),\boldsymbol{\lambda}_{t+1}-\boldsymbol{\lambda}_t\right\rangle+\frac{\ell_{F,1}}{2}\left\|\boldsymbol{\lambda}_{t+1}-\boldsymbol{\lambda}_t\right\|^2\right]\nonumber\\=\left\langle\nabla F_{t,w}\left(\boldsymbol{\lambda}_{t}\right),\boldsymbol{\lambda}_{t+1}-\boldsymbol{\lambda}_t\right\rangle+\frac{\ell_{F,1}}{2}\left\|\boldsymbol{\lambda}_{t+1}-\boldsymbol{\lambda}_t\right\|^2.\nonumber
\end{align}
Substituting in $\mathcal{G}_{\mathcal{X}}\left(\boldsymbol{\lambda}_{t},\widetilde{\nabla}f_{t,w}(\boldsymbol{\lambda}_t,\boldsymbol{\beta}_{t+1},\mathcal{Z}_{t,w}),\alpha\right):=\frac{1}{\alpha}\left(\boldsymbol{\lambda}_{t}-\boldsymbol{\lambda}_{t+1}\right)$,
\begin{align}\label{eq:stochastic_theorem_chkpt_one}
F_{t,w}\left(\boldsymbol{\lambda}_{t+1}\right)- F_{t,w}\left(\boldsymbol{\lambda}_{t}\right)\leq\left\langle\nabla F_{t,w}\left(\boldsymbol{\lambda}_{t}\right),\boldsymbol{\lambda}_{t+1}-\boldsymbol{\lambda}_t\right\rangle+\frac{\ell_{F,1}}{2}\left\|\boldsymbol{\lambda}_{t+1}-\boldsymbol{\lambda}_t\right\|^2\nonumber\\=-\alpha\left\langle \nabla F_{t,w}\left(\boldsymbol{\lambda}_{t}\right),\mathcal{G}_{\mathcal{X}}\left(\boldsymbol{\lambda}_{t},\widetilde{\nabla}f_{t,w}(\boldsymbol{\lambda}_t,\boldsymbol{\beta}_{t+1},\mathcal{Z}_{t,w}),\alpha\right)\right\rangle\nonumber\\+\frac{\alpha^2\ell_{F,1}}{2}\left\|\mathcal{G}_{\mathcal{X}}\left(\boldsymbol{\lambda}_{t},\widetilde{\nabla}f_{t,w}(\boldsymbol{\lambda}_t,\boldsymbol{\beta}_{t+1},\mathcal{Z}_{t,w}),\alpha\right)\right\|^2, \nonumber\\=-\alpha\left\langle \widetilde{\nabla}f_{t,w}(\boldsymbol{\lambda}_t,\boldsymbol{\beta}_{t+1},\mathcal{Z}_{t,w}),\mathcal{G}_{\mathcal{X}}\left(\boldsymbol{\lambda}_{t},\widetilde{\nabla}f_{t,w}(\boldsymbol{\lambda}_t,\boldsymbol{\beta}_{t+1},\mathcal{Z}_{t,w}),\alpha\right)\right\rangle\nonumber\\+\alpha\left\langle \widetilde{\nabla}f_{t,w}(\boldsymbol{\lambda}_t,\boldsymbol{\beta}_{t+1},\mathcal{Z}_{t,w})-\nabla F_{t,w}\left(\boldsymbol{\lambda}_{t}\right),\mathcal{G}_{\mathcal{X}}\left(\boldsymbol{\lambda}_{t},\widetilde{\nabla}f_{t,w}(\boldsymbol{\lambda}_t,\boldsymbol{\beta}_{t+1},\mathcal{Z}_{t,w}),\alpha\right)\right\rangle\nonumber\\+\frac{\alpha^2\ell_{F,1}}{2}\left\|\mathcal{G}_{\mathcal{X}}\left(\boldsymbol{\lambda}_{t},\widetilde{\nabla}f_{t,w}(\boldsymbol{\lambda}_t,\boldsymbol{\beta}_{t+1},\mathcal{Z}_{t,w}),\alpha\right)\right\|^2. 
\end{align}
With Lemma \ref{lem:gen_projection_bound_one}, note for $\boldsymbol{q}=\widetilde{\nabla}f_{t,w}(\boldsymbol{\lambda}_t,\boldsymbol{\beta}_{t+1},\mathcal{Z}_{t,w})$ 
\begin{align}\label{eq:stochastic_bregman_inequality_one}
    \alpha\left\langle \widetilde{\nabla}f_{t,w}(\boldsymbol{\lambda}_t,\boldsymbol{\beta}_{t+1},\mathcal{Z}_{t,w}),\mathcal{G}_{\mathcal{X}}\left(\boldsymbol{\lambda}_{t},\widetilde{\nabla}f_{t,w}(\boldsymbol{\lambda}_t,\boldsymbol{\beta}_{t+1},\mathcal{Z}_{t,w}),\alpha\right)\right\rangle \nonumber\\\geq \alpha\rho\left\|\mathcal{G}_{\mathcal{X}}\left(\boldsymbol{\lambda}_{t},\widetilde{\nabla}f_{t,w}(\boldsymbol{\lambda}_t,\boldsymbol{\beta}_{t+1},\mathcal{Z}_{t,w}),\alpha\right)\right\|^2 +h(\boldsymbol{\lambda}_{t+1})-h(\boldsymbol{\lambda}_t)
\end{align}
and further we get the following based on a variation of Young's Inequality
\begin{align}\label{eq:stohastic_bregman_inequality_two}
    \left\langle \widetilde{\nabla}f_{t,w}(\boldsymbol{\lambda}_t,\boldsymbol{\beta}_{t+1},\mathcal{Z}_{t,w})-\nabla F_{t,w}\left(\boldsymbol{\lambda}_{t}\right),\mathcal{G}_{\mathcal{X}}\left(\boldsymbol{\lambda}_{t},\widetilde{\nabla}f_{t,w}(\boldsymbol{\lambda}_t,\boldsymbol{\beta}_{t+1},\mathcal{Z}_{t,w}),\alpha\right)\right\rangle\nonumber\\\leq \frac{1}{\rho}\left\| \widetilde{\nabla}f_{t,w}(\boldsymbol{\lambda}_t,\boldsymbol{\beta}_{t+1},\mathcal{Z}_{t,w})-\nabla F_{t,w}\left(\boldsymbol{\lambda}_{t}\right)\right\|^2\nonumber\\+\frac{\rho}{4}\left\|\mathcal{G}_{\mathcal{X}}\left(\boldsymbol{\lambda}_{t},\widetilde{\nabla}f_{t,w}(\boldsymbol{\lambda}_t,\boldsymbol{\beta}_{t+1},\mathcal{Z}_{t,w}),\alpha\right)\right\|^2.
\end{align}
Combining \eqref{eq:stochastic_bregman_inequality_one} and \eqref{eq:stohastic_bregman_inequality_two} in \eqref{eq:stochastic_theorem_chkpt_one} we get
\begin{align}
F_{t,w}\left(\boldsymbol{\lambda}_{t+1}\right)- F_{t,w}\left(\boldsymbol{\lambda}_{t}\right)\leq \left(\frac{\alpha^2\ell_{F,1}}{2}-\frac{3\alpha\rho}{4}\right)\left\|\mathcal{G}_{\mathcal{X}}\left(\boldsymbol{\lambda}_{t},\widetilde{\nabla}f_{t,w}(\boldsymbol{\lambda}_t,\boldsymbol{\beta}_{t+1},\mathcal{Z}_{t,w}),\alpha\right)\right\|^2\nonumber\\+\frac{\alpha}{\rho}\left\|\widetilde{\nabla}f_{t,w}(\boldsymbol{\lambda}_t,\boldsymbol{\beta}_{t+1},\mathcal{Z}_{t,w})-\nabla F_{t,w}\left(\boldsymbol{\lambda}_{t}\right)\right\|^2+h(\boldsymbol{\lambda}_t)-h(\boldsymbol{\lambda}_{t+1})\nonumber,
\end{align}
which as $0<\alpha \leq \frac{3\rho}{4\ell_{F,1}}$ results in the further upper bound of
\begin{align}\label{eq:stochastic_theorem_ckpt_two}
F_{t,w}\left(\boldsymbol{\lambda}_{t+1}\right)- F_{t,w}\left(\boldsymbol{\lambda}_{t}\right)\leq \frac{3\alpha\rho}{8}\left\|\mathcal{G}_{\mathcal{X}}\left(\boldsymbol{\lambda}_{t},\widetilde{\nabla}f_{t,w}(\boldsymbol{\lambda}_t,\boldsymbol{\beta}_{t+1},\mathcal{Z}_{t,w}),\alpha\right)\right\|^2\nonumber\\+\frac{\alpha}{\rho}\left\|\widetilde{\nabla}f_{t,w}(\boldsymbol{\lambda}_t,\boldsymbol{\beta}_{t+1},\mathcal{Z}_{t,w})-\nabla F_{t,w}\left(\boldsymbol{\lambda}_{t}\right)\right\|^2+h(\boldsymbol{\lambda}_t)-h(\boldsymbol{\lambda}_{t+1}).
\end{align}
Further, we have
\begin{align}
    \left\|\mathcal{G}_{\mathcal{X}}(\boldsymbol{\lambda}_{t},\nabla F_{t,w}(\boldsymbol{\lambda}_{t}),\alpha)\right\|^2\leq 2\left\|\mathcal{G}_{\mathcal{X}}\left(\boldsymbol{\lambda}_{t},\widetilde{\nabla}f_{t,w}(\boldsymbol{\lambda}_t,\boldsymbol{\beta}_{t+1},\mathcal{Z}_{t,w}),\alpha\right)\right\|^2\nonumber\\+2\left\|\mathcal{G}_{\mathcal{X}}\left(\boldsymbol{\lambda}_{t},\widetilde{\nabla}f_{t,w}(\boldsymbol{\lambda}_t,\boldsymbol{\beta}_{t+1},\mathcal{Z}_{t,w}),\alpha\right)-\mathcal{G}_{\mathcal{X}}(\boldsymbol{\lambda}_{t},\nabla F_{t,w}(\boldsymbol{\lambda}_{t}),\alpha)\right\|^2\nonumber\\ \leq 2\left\|\mathcal{G}_{\mathcal{X}}\left(\boldsymbol{\lambda}_{t},\widetilde{\nabla}f_{t,w}(\boldsymbol{\lambda}_t,\boldsymbol{\beta}_{t+1},\mathcal{Z}_{t,w}),\alpha\right)\right\|^2+\frac{2}{\rho^2}\left\|\widetilde{\nabla}f_{t,w}(\boldsymbol{\lambda}_t,\boldsymbol{\beta}_{t+1},\mathcal{Z}_{t,w})-\nabla F_{t,w}\left(\boldsymbol{\lambda}_{t}\right)\right\|^2\nonumber,
\end{align}
where the last inequality comes from through Lemma \ref{lem:gen_projection_bound_two}. Then we have
\begin{align}\label{eq:stochastic_bregamn_triangle}
-\left\|\mathcal{G}_{\mathcal{X}}\left(\boldsymbol{\lambda}_{t},\widetilde{\nabla}f_{t,w}(\boldsymbol{\lambda}_t,\boldsymbol{\beta}_{t+1},\mathcal{Z}_{t,w}),\alpha\right)\right\|^2\leq-\frac{1}{2}\left\|\mathcal{G}_{\mathcal{X}}(\boldsymbol{\lambda}_{t},\nabla F_{t,w}(\boldsymbol{\lambda}_{t}),\alpha)\right\|^2\nonumber\\+\frac{1}{\rho^2}\left\|\widetilde{\nabla}f_{t,w}(\boldsymbol{\lambda}_t,\boldsymbol{\beta}_{t+1},\mathcal{Z}_{t,w})-\nabla F_{t,w}\left(\boldsymbol{\lambda}_{t}\right)\right\|^2.
\end{align}
Substituting \eqref{eq:stochastic_bregamn_triangle} in \eqref{eq:stochastic_theorem_ckpt_two}
\begin{align}
F_{t,w}\left(\boldsymbol{\lambda}_{t+1}\right)- F_{t,w}\left(\boldsymbol{\lambda}_{t}\right)\leq -\frac{3\alpha\rho}{16}\left\|\mathcal{G}_{\mathcal{X}}(\boldsymbol{\lambda}_{t},\nabla F_{t,w}(\boldsymbol{\lambda}_{t}),\alpha)\right\|^2\nonumber\\+\left(\frac{\alpha}{\rho}+\frac{3\alpha}{8\rho}\right)\left\|\widetilde{\nabla}f_{t,w}(\boldsymbol{\lambda}_t,\boldsymbol{\beta}_{t+1},\mathcal{Z}_{t,w})-\nabla F_{t,w}\left(\boldsymbol{\lambda}_{t}\right)\right\|^2+h(\boldsymbol{\lambda}_t)-h(\boldsymbol{\lambda}_{t+1})\nonumber
\end{align}
Telescoping $t=1,\ldots,T$ and  taking expectation with respect to $\bar{\zeta}_{t,k}$ and $\mathcal{Z}_{t,w}$ gives us
\begin{align}\label{eq:stochastic_theorem_ckpt_three}
\frac{3\alpha\rho}{16}\sum_{t=1}^T\left\|\mathcal{G}_{\mathcal{X}}(\boldsymbol{\lambda}_{t},\nabla F_{t,w}(\boldsymbol{\lambda}_{t}),\alpha)\right\|^2\leq  \sum_{t=1}^T\left(F_{t,w}\left(\boldsymbol{\lambda}_{t}\right)-F_{t,w}\left(\boldsymbol{\lambda}_{t+1}\right)\right)\nonumber\\+\frac{11\alpha}{8\rho}\mathbb{E}_{\bar{\zeta}_{t,K+1}}\left[\mathbb{E}_{\mathcal{Z}_{t,w}}\left[\sum_{t=1}^T\left\|\widetilde{\nabla}f_{t,w}(\boldsymbol{\lambda}_t,\boldsymbol{\beta}_{t+1},\mathcal{Z}_{t,w})-\nabla F_{t,w}\left(\boldsymbol{\lambda}_t\right)\right\|^2\right]\right]+\Delta_h,
\end{align}
where $\Delta_h:=h(\boldsymbol{\lambda}_1)-h(\boldsymbol{\lambda}_{T+1})$ . Substituting  the result of Lemma \ref{lem:stochastic_cumulative_hypergradient_error} in \eqref{eq:stochastic_theorem_ckpt_three}
\begin{align}
\frac{3\alpha\rho}{16}\sum_{t=1}^T\left\|\mathcal{G}_{\mathcal{X}}(\boldsymbol{\lambda}_{t},\nabla F_{t,w}(\boldsymbol{\lambda}_{t}),\alpha)\right\|^2\leq  \sum_{t=1}^T\left(F_{t,w}\left(\boldsymbol{\lambda}_{t}\right)-F_{t,w}\left(\boldsymbol{\lambda}_{t+1}\right)\right)+\Delta_h\nonumber\\+\frac{11\alpha}{8\rho}\left(\frac{9T\sigma^2_f}{2w}+\frac{9T\ell^2_{f,1}\kappa^2_g}{2}\left(1-\frac{\mu_g}{\ell_{g,1}}\right)^{2m} +\frac{9\kappa^2_g\Delta_{\boldsymbol{\beta}}}{4(1-\nu)}+\frac{9TC_K\kappa^2_g\eta^2\sigma^2_{g_{\boldsymbol{\beta}}}}{4S(1-\nu)}\right)\nonumber\\+\frac{11\alpha}{8\rho}\left(\frac{\rho^2}{8}\sum_{t=1}^T\left\|\mathcal{G}_{\mathcal{X}}(\boldsymbol{\lambda}_{t},\nabla F_{t,w}(\boldsymbol{\lambda}_{t}),\alpha)\right\|^2+\frac{9C_{\mu_g}\kappa^2_g}{2(1-\nu)}\sum_{t=2}^T\left[\left\|\widehat{\boldsymbol{\beta}}_{t}(\boldsymbol{\lambda}_{t-1})-\widehat{\boldsymbol{\beta}}_{t-1}(\boldsymbol{\lambda}_{t-1})\right\|^2\right]\right)
\end{align}
we have to rearrange
\begin{align}
\frac{\alpha\rho}{64}\sum_{t=1}^T\left\|\mathcal{G}_{\mathcal{X}}(\boldsymbol{\lambda}_{t},\nabla F_{t,w}(\boldsymbol{\lambda}_{t}),\alpha)\right\|^2\leq  \sum_{t=1}^T\left(F_{t,w}\left(\boldsymbol{\lambda}_{t}\right)-F_{t,w}\left(\boldsymbol{\lambda}_{t+1}\right)\right)+\Delta_h\nonumber\\+\frac{99\alpha}{32\rho}\left(\frac{2T\sigma^2_f}{w}+2T\ell^2_{f,1}\kappa^2_g\left(1-\frac{\mu_g}{\ell_{g,1}}\right)^{2m} +\frac{\kappa^2_g\Delta_{\boldsymbol{\beta}}}{(1-\nu)}+\frac{TC_K\kappa^2_g\eta^2\sigma^2_{g_{\boldsymbol{\beta}}}}{s(1-\nu)}\right)\nonumber\\+\frac{99\alpha C_{\mu_g}\kappa^2_g}{16\rho(1-\nu)}\sum_{t=2}^T\left\|\widehat{\boldsymbol{\beta}}_{t}(\boldsymbol{\lambda}_{t-1})-\widehat{\boldsymbol{\beta}}_{t-1}(\boldsymbol{\lambda}_{t-1})\right\|^2
\end{align}
or more succinctly with the choice of $s=w$
and $m=\log{(w)}/\log{\left(1-\frac{\mu_g}{\ell_{g,1}}\right)}+1$, we have
\begin{align}
\sum_{t=1}^T\left\|\mathcal{G}_{\mathcal{X}}(\boldsymbol{\lambda}_{t},\nabla F_{t,w}(\boldsymbol{\lambda}_{t}),\alpha)\right\|^2\leq  \frac{64}{\alpha\rho}\left(\sum_{t=1}^T\left(F_{t,w}\left(\boldsymbol{\lambda}_{t}\right)-F_{t,w}\left(\boldsymbol{\lambda}_{t+1}\right)\right)+\Delta_h\right)\nonumber\\+\frac{198}{\rho^2}\frac{T}{w}\left(2\sigma^2_f+2\ell^2_{f,1}\kappa^2_g+\frac{C_K\kappa^2_g\eta^2\sigma^2_{g_{\boldsymbol{\beta}}}}{(1-\nu)}\right)+\frac{198}{\rho^2}\frac{\kappa^2_g\Delta_{\boldsymbol{\beta}}}{(1-\nu)}\nonumber\\+\frac{396 C_{\mu_g}\kappa^2_g}{\rho^2(1-\nu)}\sum_{t=2}^T\left\|\widehat{\boldsymbol{\beta}}_{t}(\boldsymbol{\lambda}_{t-1})-\widehat{\boldsymbol{\beta}}_{t-1}(\boldsymbol{\lambda}_{t-1})\right\|^2
\end{align}
Using the result of Lemma \ref{lem:stochastic_bounded_time_varying}, we have 
\begin{align}
\sum_{t=1}^T\left(F_{t,w}(\boldsymbol{\lambda}_t)-F_{t,w}(\boldsymbol{\lambda}_{t+1})\right) \leq \frac{2TQ}{w} +V_{1,T}
\end{align}
or all together
\begin{align}
\sum_{t=1}^T\left\|\mathcal{G}_{\mathcal{X}}(\boldsymbol{\lambda}_{t},\nabla F_{t,w}(\boldsymbol{\lambda}_{t}),\alpha)\right\|^2\leq  \frac{64}{\alpha\rho}\left(\frac{2TQ}{w} +V_{1,T}+\Delta_h\right)\nonumber\\+\frac{198}{\rho^2}\frac{T}{w}\left(2\sigma^2_f+2\ell^2_{f,1}\kappa^2_g+\frac{C_K\kappa^2_g\eta^2\sigma^2_{g_{\boldsymbol{\beta}}}}{(1-\nu)}\right)+\frac{198}{\rho^2}\frac{\kappa^2_g\Delta_{\boldsymbol{\beta}}}{(1-\nu)}\nonumber\\+\frac{396 C_{\mu_g}\kappa^2_g}{\rho^2(1-\nu)}\sum_{t=2}^T\left\|\widehat{\boldsymbol{\beta}}_{t}(\boldsymbol{\lambda}_{t-1})-\widehat{\boldsymbol{\beta}}_{t-1}(\boldsymbol{\lambda}_{t-1})\right\|^2
\end{align}
which  dividing by $T$ and recalling we imposed regularity constraints of $H_{2,T}=o(T)$, as well as $V_{1,T}=o(T)$,  implies the bilevel local regret of our SOBBO algorithm is  sublinear on the order of
\begin{align}
    BLR_w(T):= \sum_{t=1}^T\left\|\mathcal{G}_{\mathcal{X}}(\boldsymbol{\lambda}_{t},\nabla F_{t,w}(\boldsymbol{\lambda}_{t}),\alpha)\right\|^2\leq \nonumber\\O\left(\frac{T}{w}\left(1+\kappa^2_g+\sigma^2_f+\kappa^2_g\sigma^2_{g_{\boldsymbol{\beta}}}\right)+V_{1,T}+\kappa^2_gH_{2,T}\right)
\end{align}
\end{proof}

\section{Bilevel Local Regret Comparison}\label{sec:appendix_C}
In this section, we provide a detailed comparison of the bilevel local regret achieved by OBBO relative to online bilevel benchmarks of OAGD (\cite{tarzanagh2024online}) and SOBOW (\cite{sobow}).

\subsection{OAGD} 
We provide a restatement of the bilevel local regret presented in Theorem 9 of \cite{tarzanagh2024online} using our notation.
\begin{theorem}\label{thrm:oagd}(Theorem 9 in \cite{tarzanagh2024online}) Suppose Assumptions \ref{assump:rel_smoothness}, \ref{assump:strong_convex_g_t}, \ref{assump:bounded_hyperparm}, and \ref{assump:bounded_F_t}. Then  the OAGD algorithm of \cite{tarzanagh2024online} for an inner step size of $\eta=\frac{2}{\ell_{g,1}+\mu_g}$, inner iteration count of $K=1$, and outer step size of $\alpha\leq \min\left\{\frac{1}{8\ell_{F,1}},\frac{1}{2\sqrt{2}L_{\boldsymbol{\beta}}M_f(\kappa^2_g-1)^{1/2}}\right\}$ satisfies 
\begin{align}
   \sum_{t=1}^T\left\|\nabla F_{t,w}(\boldsymbol{\lambda}_t)\right\|^2\leq\frac{16}{\alpha}\left(\frac{2TQ}{w}+2Q+\ell_{f,0}H_{1,T}\right)\nonumber\\+10M^2_f(\kappa_g-1)^2\left(\frac{\left\|\boldsymbol{\beta}_1-\widehat{\boldsymbol{\beta}}_1(\boldsymbol{\lambda}_1\right\|^2}{2(\kappa_g+1)}+2H_{2,T}\right)=  O\left(\frac{T}{w}+H_{1,T}+\kappa^4_gH_{2,T}\right)
\end{align} 
\end{theorem}
Following the literature on dynamic regret, \cite{tarzanagh2024online} considers the setting where the inner level path variation of $H_{1,T}$ and $H_{2,T}$ are sublinear (i.e., $H_{1,T}=o(T)$, $H_{2,T}=o(T)$) which enforces that the amount of nonstationarity  cannot grow faster than or equal to time itself. In such a setting, the rate of bilevel local regret achieved by OAGD is sublinear when properly selecting the window size such that $w=o(T)$. Note as the  constant $M_f$ in \cite{tarzanagh2024online} has a quadratic dependency on $\kappa_g$, that is $M_f=O(\kappa^2_g)$, the total  dependency of  the condition number $\kappa_g$ in the sublinear rate of bilevel local regret of OAGD  is fourth order.

\subsection{SOBOW}
We restate the bilevel local regret of Theorem 5.7 from \cite{sobow} with our notation.
\begin{theorem}\label{thrm:sobow}(Theorem 5.7 in \cite{sobow})
Suppose Assumptions \ref{assump:rel_smoothness}, \ref{assump:strong_convex_g_t}, \ref{assump:bounded_hyperparm},  \ref{assump:bounded_F_t} and furthermore Assumption 5.4 of \cite{sobow}. Let  the inner step size of $\eta\leq\frac{1}{\ell_{g,1}}$, decay parameter of $\nu\in\left(1-\frac{\alpha\mu_g}{2}\right)$, and outer step size of $\alpha\leq \min\left\{\frac{1}{4\ell_{F,1}},\frac{\mu^2_g\ell_{F,1}W(1-\nu)(\nu-1+\alpha\mu_g/2}{24\ell_{g,1}^4G_2\nu}\right\}$. Then we have
\begin{align}
     \sum_{t=1}^T\left\|\nabla F_{t,w}(\boldsymbol{\lambda}_t)\right\|^2 \leq C_1\left(\frac{2TQ}{w}+V_{1,T}\right) \nonumber\\+C_2\left(\frac{1}{2}+\beta \ell_{F,1}\right)\frac{G_2G_4}{G_3}H_{2,T}+C_3 =O\left(\frac{T}{w}+V_{1,T}+\kappa^3_gH_{2,T}\right)
\end{align}
where  $G_1=O(\kappa^2_g), G_2=O(\kappa^2_g), G_3=O(1), G_4=O(\kappa_g), \beta \ell_{F,1}=O(1)$ and   constants $C_1,C_2,C_3\in \mathbb{R}$ are from Theorem 5.7 in \cite{sobow}. 
\end{theorem}

The  work of  \cite{sobow}, similarly inspired by the dynamic regret literature, considers  the setting of sublinear $H_{2,T}$ and  $V_{1,T}$-- where the former term is second order inner level path variation and the latter term measures the variation in evaluations of the outer level objective function.  In such a setting, the rate of bilevel local regret achieved by SOBOW is sublinear when the window size is properly selected such that $w=o(T)$. Recalling the dependencies of $\kappa_g$ in the terms of $G_1,G_2,G_3,G_4, \beta \ell_{F,1}$ in \cite{sobow}, we remark the total  dependency of  the condition number $\kappa_g$ in the sublinear rate of bilevel local regret of SOBOW  is third order.

\subsection{OBBO}
We restate the sublinear rate of bilevel local regret achieved by OBBO in Theorem \ref{thrm:non_convex_regret_paper} below.

\begin{theorem}\label{thrm:obbo}
    Suppose Assumptions \ref{assump:rel_smoothness}, \ref{assump:strong_convex_g_t}, \ref{assump:bounded_hyperparm}, \ref{assump:bounded_F_t}, and \ref{assump:continuity_phi_t}. Let the inner step size of $\eta < \min{\left(\frac{1}{\ell_{g,1}},\frac{1}{\mu_g}\right)}$, outer step size of $  \alpha\leq \min{\left\{\frac{3\rho}{4\ell_{F,1}},\frac{\rho\sqrt{(1-\nu)}}{\kappa_g\sqrt{108C_{\mu_g}L_{\boldsymbol{\beta}}}}\right\}}$, and inner iteration count $K=\frac{\log{(T)}}{\log{\left((1-\eta\mu_g)^{-1}\right)}}+1$. For simplicity, assume $\phi_t(\boldsymbol{\lambda})=\phi(\boldsymbol{\lambda})=\frac{1}{2}\left\|\boldsymbol{\lambda}\right\|^2$, and $h(\boldsymbol{\lambda})=0$. Then the bilevel local regret of our OBBO algorithm  satisfies 
    \begin{align}
\sum_{t=1}^T\left\|\nabla F_{t,w}(\boldsymbol{\lambda}_t)\right\|^2\leq  \frac{64}{\alpha\rho}\left(\frac{2TQ}{w} +V_{1,T}\right)\nonumber\\+\frac{297}{\rho^2}\left(\frac{\Delta_{\boldsymbol{\beta}} L_{\boldsymbol{\beta}}}{(1-\nu)} +L^2_3\right)+\frac{1188L_{\boldsymbol{\beta}}C_{\mu_g}}{\rho^2(1-\nu)}H_{2,T}= O\left(\frac{T}{w}+V_{1,T}+\kappa_g^2H_{2,T}\right)
\end{align}
\end{theorem}

As in the work of \cite{sobow}, we consider the setting of sublinear  $H_{2,T}$ and $V_{1,T}$.  In such a setting, the rate of bilevel local regret achieved by OBBO is sublinear when properly selecting the window size such that $w=o(T)$. Further, we remark that as $L_{\boldsymbol{\beta}}=O(\kappa^2_g)$, the total  dependency of  the condition number $\kappa_g$ in the sublinear rate of bilevel local regret of OBBO  is second order. Compared to the  regret achieved by OAGD and SOBOW in Theorem \ref{thrm:oagd} and \ref{thrm:sobow}  this is a second-order and first-order improvement respectively.

\subsection{SOBBO}
We restate the sublinear rate of bilevel local regret achieved by SOBBO in Theorem \ref{thrm:sobbo_regret_minimization_paper} below.

\begin{theorem}\label{thrm:obbo}
      Suppose Assumptions \ref{assump:rel_smoothness}, \ref{assump:strong_convex_g_t}, \ref{assump:unbiased_finite_var}, \ref{assump:bounded_hyperparm}, \ref{assump:bounded_F_t}, and \ref{assump:continuity_phi_t}. Let the inner step size $\eta\leq\frac{2}{\ell_{g,1}+\mu_g}$, the inner iteration count $K\geq 1$,  the outer step size $\alpha\leq \min\left\{\frac{3\rho}{4\ell_{F,1}},\frac{\rho\sqrt{(1-\nu)}}{\kappa^2_g\sqrt{72C_{\mu_g}}}\right\}$, and inner and outer level batch sizes of $s=w$ and $m=\log{(w)}/\log{\left(1-\frac{\mu_g}{\ell_{g,1}}\right)}+1$ respectively.  For simplicity, assume $\phi_t(\boldsymbol{\lambda})=\phi(\boldsymbol{\lambda})=\frac{1}{2}\left\|\boldsymbol{\lambda}\right\|^2$, and $h(\boldsymbol{\lambda})=0$. Then the  bilevel local regret of  SOBBO satisfies
    \begin{align}
\sum_{t=1}^T\left\|\nabla F_{t,w}(\boldsymbol{\lambda}_t)\right\|^2\leq  \frac{64}{\alpha\rho}\left(\frac{2TQ}{w} +V_{1,T}+\Delta_h\right)\nonumber\\+\frac{198}{\rho^2}\frac{T}{w}\left(2\sigma^2_f+2\ell^2_{f,1}\kappa^2_g+\frac{C_K\kappa^2_g\eta^2\sigma^2_{g_{\boldsymbol{\beta}}}}{(1-\nu)}\right)+\frac{198}{\rho^2}\frac{\kappa^2_g\Delta_{\boldsymbol{\beta}}}{(1-\nu)}+\frac{396 C_{\mu_g}\kappa^2_g}{\rho^2(1-\nu)}\sum_{t=2}^TH_{2,T}\nonumber\\= O\left(\frac{T}{w}\left(1+\kappa^2_g+\sigma^2_f+\kappa^2_g\sigma^2_{g_{\boldsymbol{\beta}}}\right)+V_{1,T}+\kappa^2_gH_{2,T}\right)
\end{align}
\end{theorem}
Following the setup for the deterministic case, we consider the setting of sublinear  $H_{2,T}$ and $V_{1,T}$.  In such a setting, the rate of bilevel local regret achieved by SOBBO is sublinear when properly selecting the window and \emph{ batch} sizes such that $w=o(T)$ and $s=o(T)$. Due to SOBBO only having access to noisy gradient samples, a sublinear rate of gradient samples is required. Assuming the above conditions are satisfied,  SOBBO achieves a sublinear rate of bilevel local regret with the sublinear rate further generalizing the deterministic result to finite outer and inner variances $\sigma^2_f,\sigma^2_{g_{\boldsymbol{\beta}}}$.
 
 \begin{table}[htp]
    \centering
    \begin{tabular}{|c|c|}
    \hline
       \textbf{Algorithm}  & $BLR_{w}(T)$ \\ \hline 
          OAGD &  $O(T/w+H_{1,T}+\boldsymbol{\kappa^4_g}H_{2,T})$   \\ \hline
          SOBOW  &  $O(T/w+V_{1,T}+\boldsymbol{\kappa^3_g}H_{2,T})$   \\ \hline
          OBBO &  $O(T/w+V_{1,T}+\boldsymbol{\kappa^2_g}H_{2,T})$   \\ \hline
        SOBBO &  $O\left(T/w\left(1+\mathbf{\kappa^2_g}+\sigma^2_f+\mathbf{\kappa^2_g}\sigma^2_{g_{\boldsymbol{\beta}}}\right)+V_{1,T}+\mathbf{\kappa^2_g}H_{2,T}\right)$   \\ \hline
    \end{tabular}
    \caption{Bilevel local regret, $BLR_{w}(T)$, of OBBO vs.  online bilevel benchmarks SOBOW from \cite{sobow} and OAGD from \cite{tarzanagh2024online}. Note the first and second order-wise improvement OBBO offers in terms of the condition number $\kappa_g$ dependency. Bilevel local regret for SOBBO is also included, generalizing the deterministic case. The  $BLR_{w}(T)$ is reported in online rounds $T$ , window parameter $w$, comparator sequences $V_{1,T},H_{1,T}, H_{2,T}$ , condition number $\kappa_g$, and finite  outer and inner variances $\sigma^2_f,\sigma^2_{g_{\boldsymbol{\beta}}}$, respectively.}
    \label{tab:label_table}
\end{table}

\section{Additional Experimental Details and Results}\label{sec:appendix_D}
\subsection{Online Hyperparameter Optimization}

\textbf{Market Impact} This dataset consists of equity price time series. In particular, this dataset contains time series corresponding to 440 significant market impact events annotated by experts from the components of the S\&P 500 index between January 2021 and December 2022. For each annotated event, there is a corresponding sequence of 600 training-validation subsets with equal length of 700 observations constructed on a rolling basis, see a sample of this time series split in Figure \ref{fig:sample_exp_design}. Additionally, for each annotated event, there is a test set of 120  observations held out for evaluation such that the last  corresponding training subset goes up to the  annotation. We consider the task of time series forecasting on the post-annotation test set given the available training-validation subsets, and quantify performance by the mean-squared error.

\begin{minipage}[c]{\linewidth}
    \centering
\includegraphics[width=.33\textwidth,height=45mm]{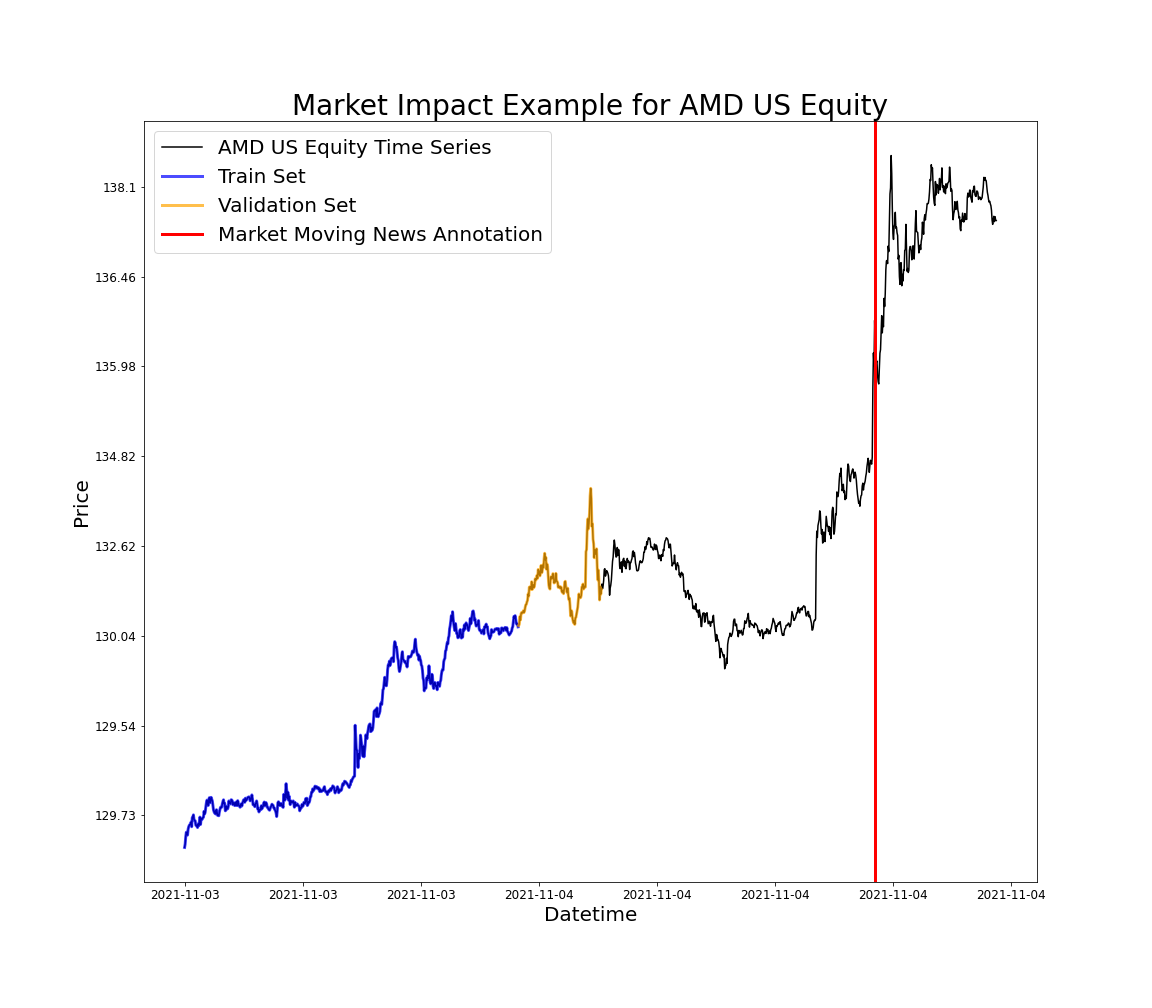}\includegraphics[width=.33\textwidth,height=45mm]{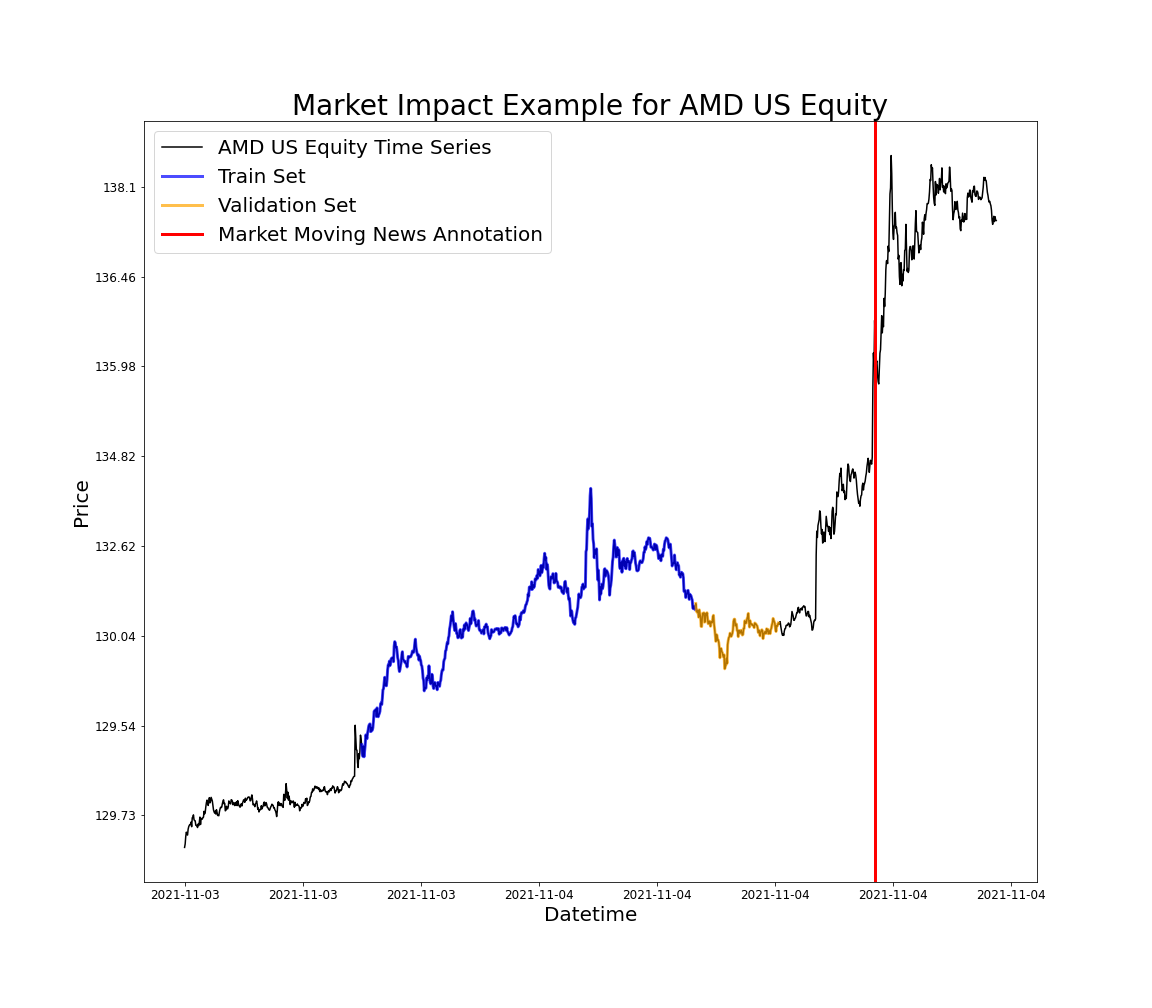}\includegraphics[width=.33\textwidth,height=45mm]{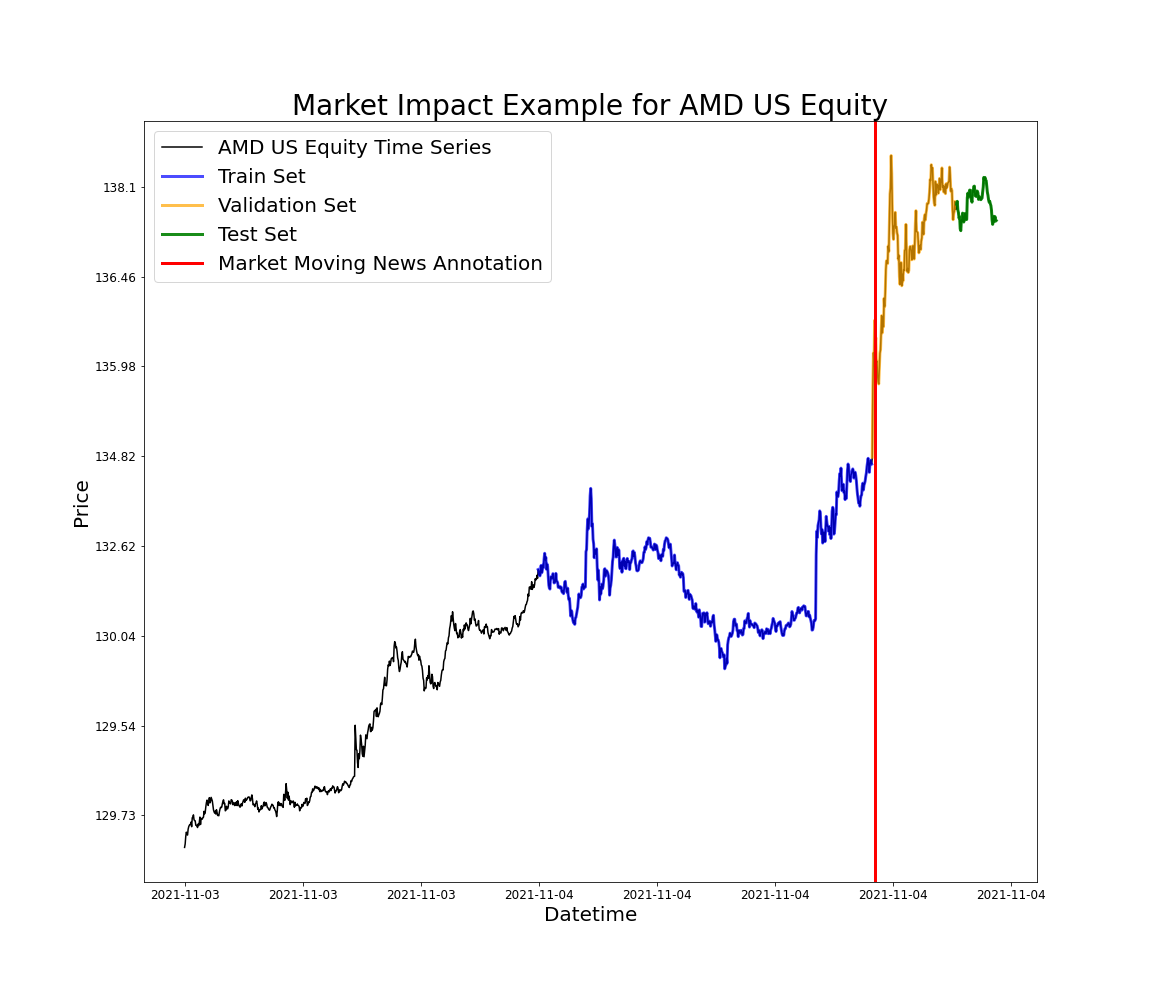}
    \captionof{figure}{Sample training-validation subsets for AMD U.S. Equity with annotated market event on 11-08-2021. }\label{fig:sample_exp_design}
  \end{minipage}
\\

We consider a smoothing spline model of linear order where the inner level variables are  B-spline coefficients  and the outer level variable is a positive regularization hyperparameter, respectively fitted on the train and validation datasets. The simplicity of such a model allows us to use closed-form hypergradients, instead of an inner gradient descent loop. All algorithms and window configurations have  the outer learning rate set at $\alpha=0.001$. For OBBO we use the  reference function of  $\phi_t(\boldsymbol{\lambda})=\frac{1}{2}\boldsymbol{\lambda}^T\mathbf{H}_t\boldsymbol{\lambda}$ such that $\mathbf{H}_t$ is an adaptive matrix of averaged gradient squares with coefficient 0.9, previously applied in prior works (\cite{huang2022enhanced},\cite{huang2021super}). Default coefficients from PyTorch \cite{paszke2019pytorch} are used for ADAM ($\beta_1=0.9,\beta_2=0.999$) as well as SGDM ($\beta_1=0.9$) with gradient clipping applied to all algorithms using the threshold on the gradient norm of $\left\|\nabla f_{t,w}\right\|^2\leq1000$.

In the left panel of Figure \ref{fig:full_local_regret}, note the significant improvement in the  median cumulative local regret achieved by OBBO relative to  online bilevel  (OAGD, SOBOW) and offline general purpose (ADAM, SGDM) benchmark algorithms across 440 U.S. markets. This empirical improvement in cumulative local regret further justifies our theoretical results provided in Theorem \ref{thrm:non_convex_regret_paper} and Theorem \eqref{thrm:sobbo_regret_minimization_paper}. Further in the middle and right panels of Figure \ref{fig:full_local_regret} we visualize  stability metrics across algorithms of (i) the gradient norm at $t=T$, and (ii) the forecasting mean-squared error on a test set. Specifically we see OBBO often achieves  a smaller gradient norm at $t=T$ and a smaller forecasting loss relative to online and offline benchmarks.

\begin{minipage}[c]{\linewidth}
    \centering
\includegraphics[width=.33\textwidth]{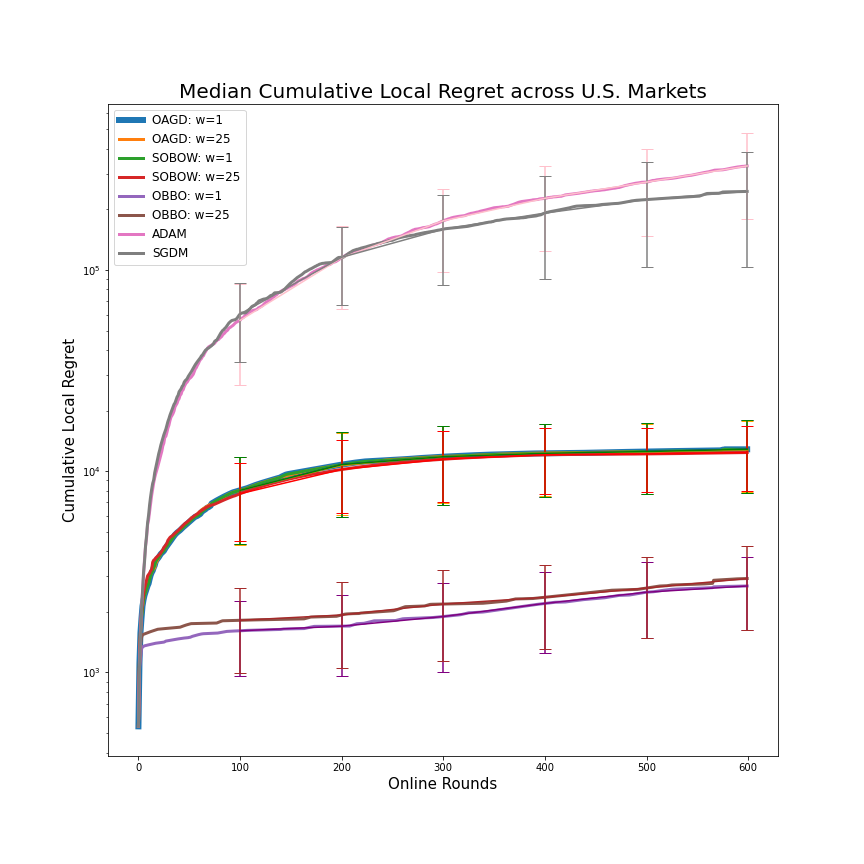}\includegraphics[width=.33\textwidth]{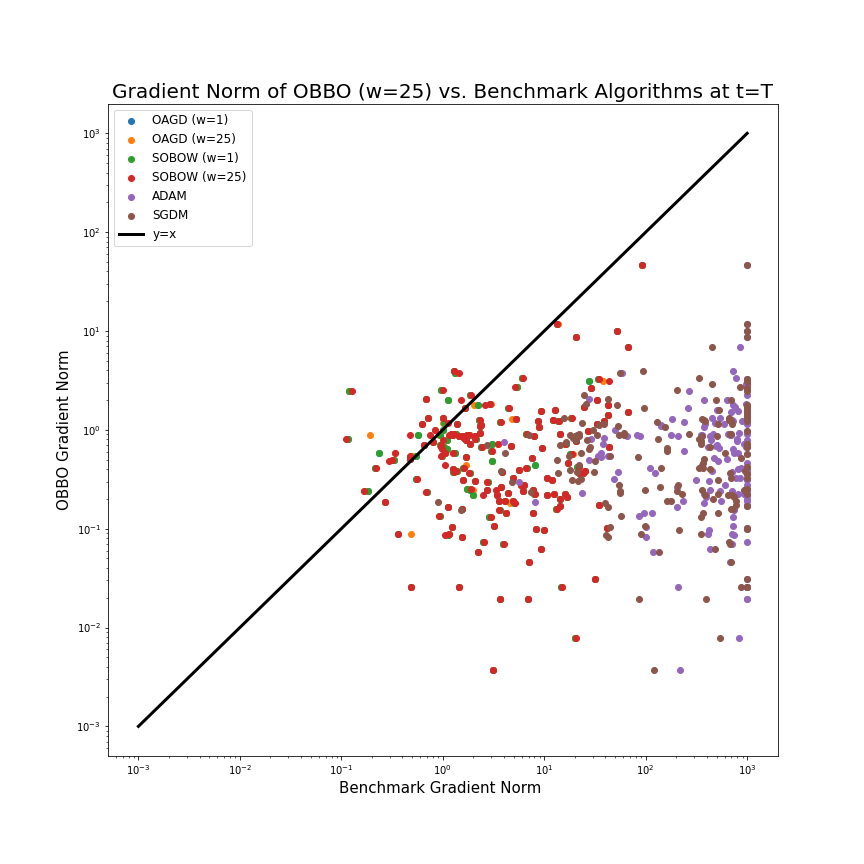}\includegraphics[width=.33\textwidth]{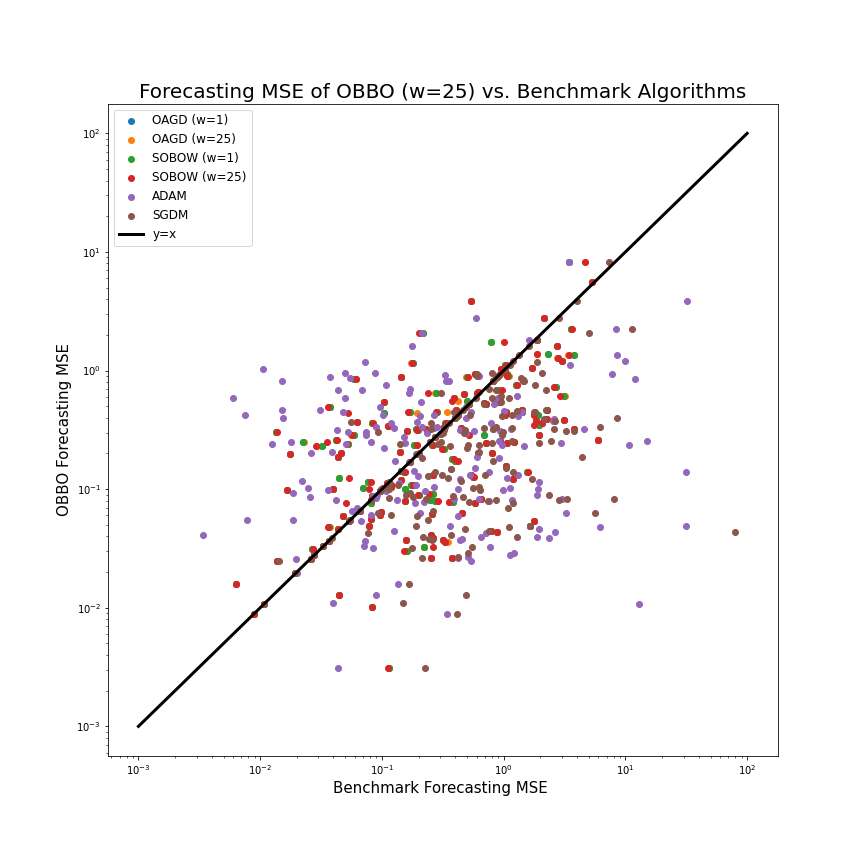}
    \captionof{figure}{\textbf{Left Panel}: Median cumulative local regret of OBBO vs. online and offline benchmark algorithms across 440 U.S. markets with window size parameter $w=1,25$ and median deviation bars  plotted every 100 rounds. \textbf{Middle Panel}: Gradient norm of OBBO (w=25) vs. online and offline benchmark algorithms  at  $t=T$ with $y=x$ line plotted to visualize the  improvement OBBO offers in achieving a smaller gradient norm. \textbf{Right Panel}: Forecasting mean-squared error of OBBO (w=25) vs. online and offline benchmark algorithms with $y=x$ line plotted to visualize the improvement OBBO offers in forecasting  loss on a test set.}\label{fig:full_local_regret}
  \end{minipage}\hfill\\
  
In Figure \ref{fig:results_exp}, we include forecasts generated from OBBO vs. benchmarks algorithms across a sample of training-validation subsets for the AMD U.S. equity time series. Note how OBBO is quicker to adapt to the annotated market impact event and achieves a better fit (i.e., smaller forecasting loss) relative to the benchmarks on the post-annotation test set. Both of the aforementioned improvements are exhibited across the Market Impact dataset and are not particularly sensitive to experiment design (e.g., number of subsets) or hyperparameters (e.g., window size). Descriptive statistics of forecasting loss aggregated across samples from the Market Impact dataset are in Table \ref{tab:table_3}.

\begin{minipage}[c]{\linewidth}
    \centering
\includegraphics[width=.33\textwidth,height=45mm]{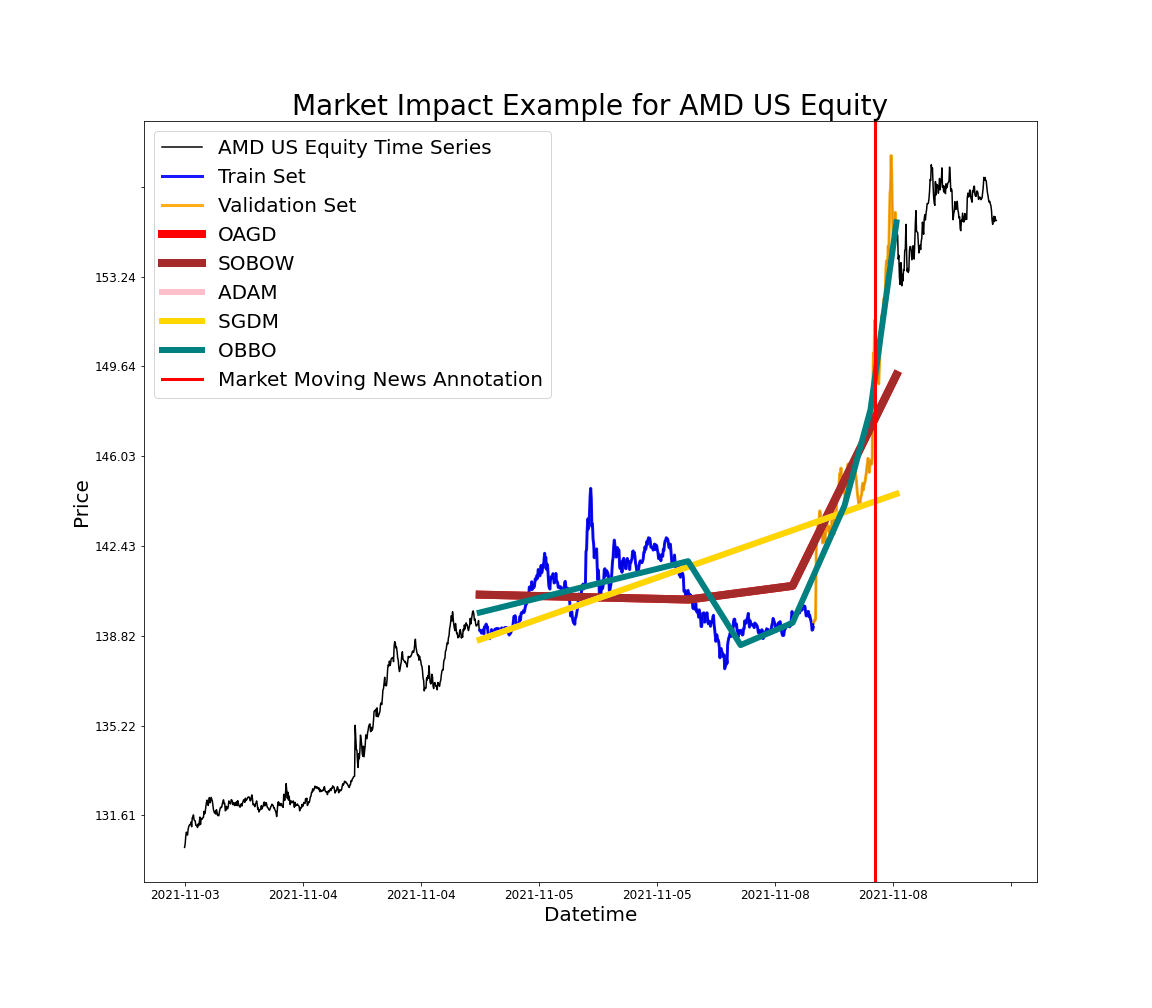}\includegraphics[width=.33\textwidth,height=45mm]{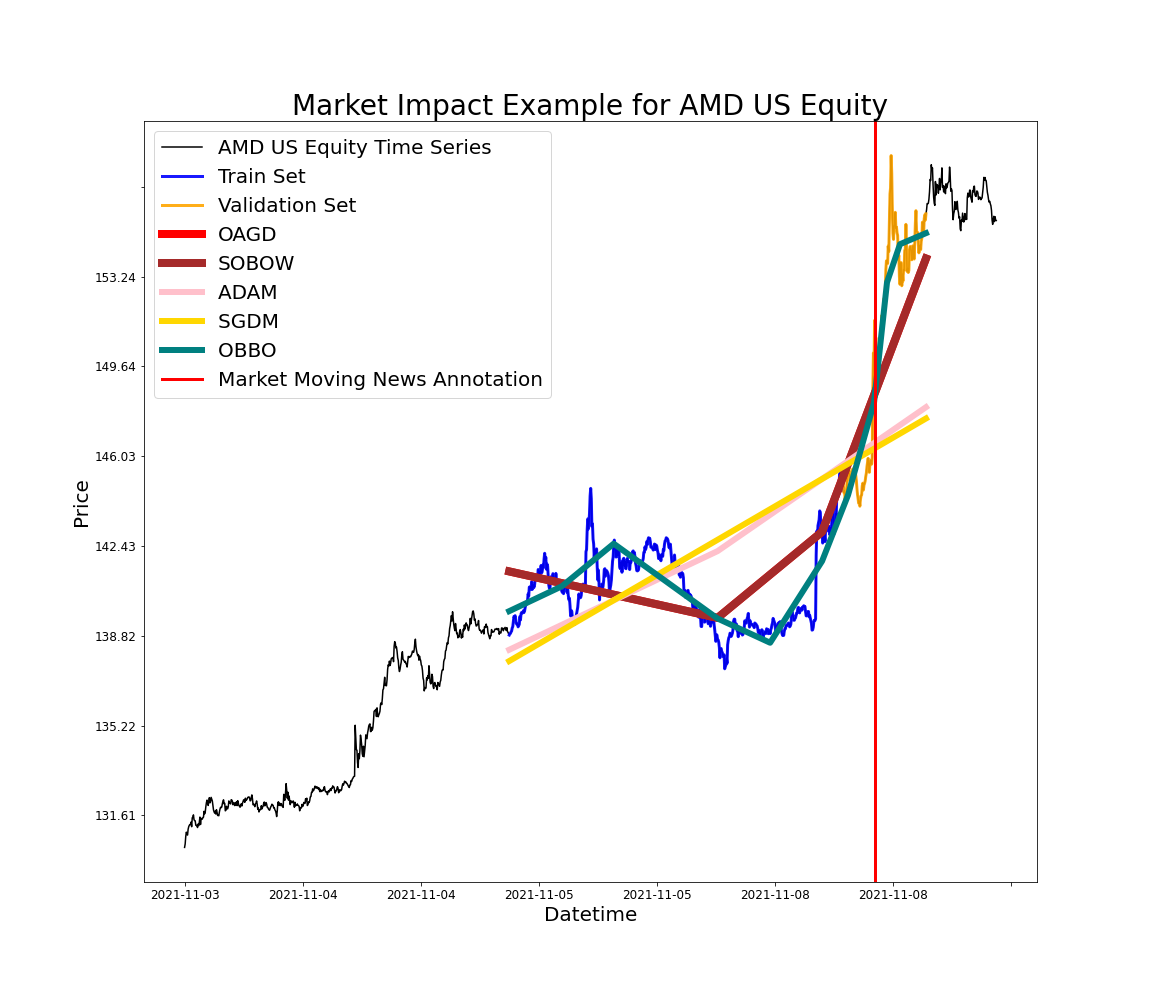}\includegraphics[width=.33\textwidth,height=45mm]{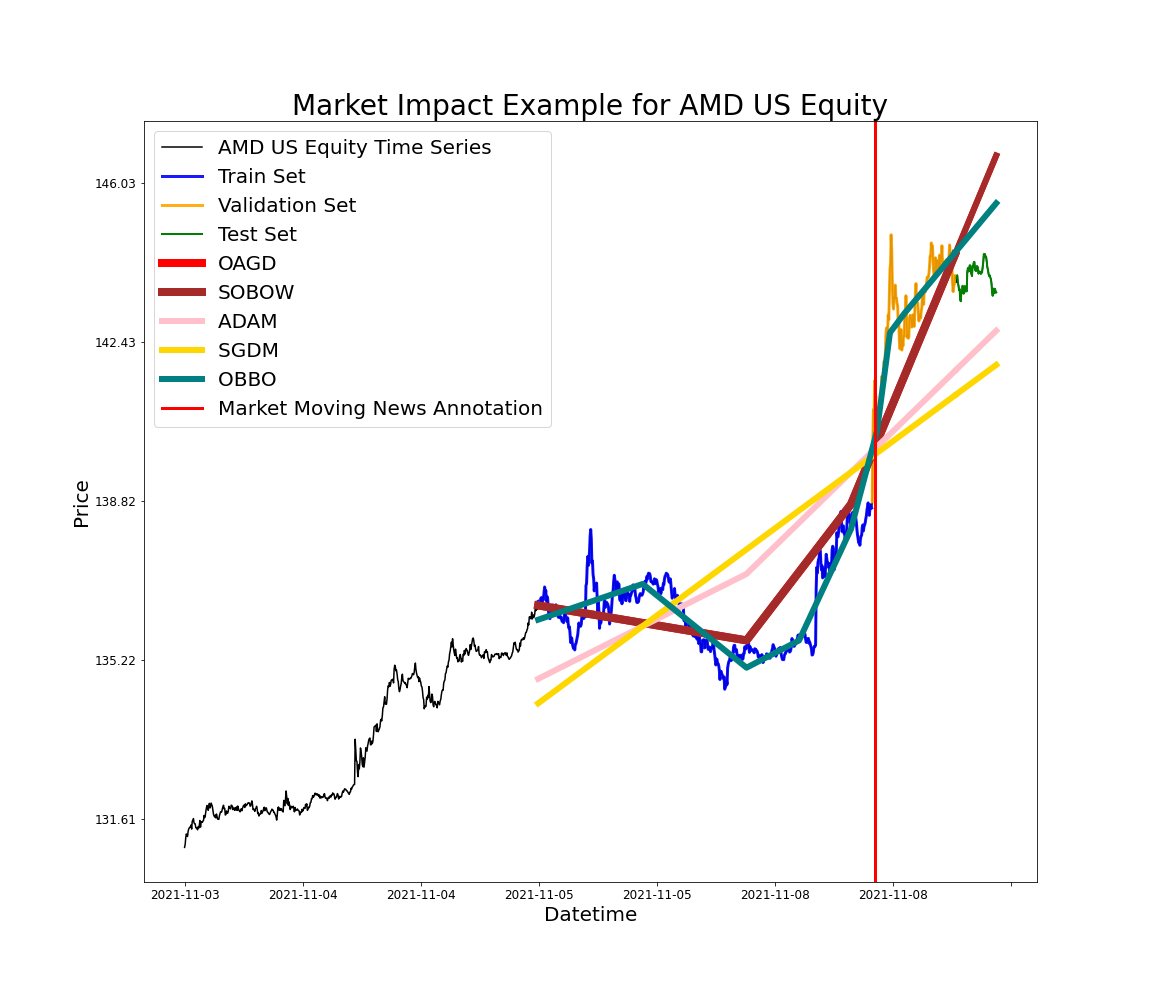}
    \captionof{figure}{Example forecasts generated with OBBO vs. online (w=25) and offline benchmark algorithms. Note how OBBO achieves a better fit (i.e., smaller loss) relative to benchmarks on the post-annotation test set.}\label{fig:results_exp}
  \end{minipage}

\begin{table}
     \begin{tabular}{lcccc}
      \toprule
      & \multicolumn{4}{c}{Forecasting Loss across U.S. Markets}\\
      \cmidrule(lr){2-5}
      \textbf{Algorithm}    &   Mean Loss & Standard Error  &    Median Loss & Median Absolute Deviation   \\
      \midrule
      \textbf{OBBO}
      & 0.661   & 0.055 & 0.205 & 0.150 \\
      \textbf{OAGD}
      & 0.707  & 0.053 & 0.265 & 0.209 \\
      \textbf{SOBOW}
      & 0.689 & 0.053  & 0.273 & 0.215 \\
        \textbf{Adam}
      &1.265   & 0.176 & 0.267 & 0.230 \\
      \textbf{SGDM}
      & 0.872 & 0.078  & 0.401 & 0.286 \\
      \bottomrule
    \end{tabular}

  \caption{Statistics of forecasting mean-squared error across U.S. markets for window parameter $w=25$. }\label{tab:table_3}
\end{table}

\subsection{Online Meta-Learning}

\textbf{FC100}
The FC100 dataset is constructed from the CIFAR100 dataset for few-shot learning tasks. Originally introduced within \cite{oreshkin2018tadam}, this dataset has been previously utilized in online meta-learning experiments such as in the OBO work of \cite{tarzanagh2024online}. The dataset contains 100 classes, split into 60:20:20 classes for meta-training, meta-validation and meta-testing respectively. Samples are transformed into tasks via the  procedure of \cite{tarzanagh2024online} resulting in 20,000, 600, and 600 training-validation-test tasks.

As in the experiment setup of \cite{tarzanagh2024online}, we consider a 4-layer convolutional neural network with each layer containing 64 filters. The CNN utilized has 4 convolutional blocks such that there is $3\times3$ convolution, batch normalization, ReLU activation, and $2\times2$ max pooling.  Inner and outer learning rates are set as $\eta=0.1$ and $\alpha=1e-4$. Following \cite{tarzanagh2024online},  we use the hypergradient estimate of $\widetilde{\nabla}f_t(\boldsymbol{\lambda},\boldsymbol{\beta})$  computed via a fixed point approach as in \cite{grazzi2020iteration}. For OBBO we use the function of  $\phi_t(\boldsymbol{\lambda})=\frac{1}{2}\boldsymbol{\lambda}^T\mathbf{H}_t\boldsymbol{\lambda}$ such that $\mathbf{H}_t$ is an adaptive matrix of averaged gradient squares with coefficient 0.9.

In the left panel of Figure \ref{fig:meta_learn_regret}, OBBO achieves a significant improvement in cumulative bilevel local regret relative to benchmark algorithms of OAGD and SOBOW across samples from the FC100 dataset. In the right panel of Figure \ref{fig:meta_learn_regret}, the histogram displays how OBBO achieves smaller evaluated gradient norms across iterations vs. the benchmark algorithms of OAGD and SOBOW. In the left  panel of Figure \ref{fig:meta_learn_train_test}, we report higher training accuracy achieved with OBBO. In the right panel of Figure \ref{fig:meta_learn_train_test}, OBBO outperforms  test accuracy relative to SOBOW while achieving OAGD performance with a 10x ($w=10$) computationally cheaper update. All results are averaged across 5 random seeds.

\begin{figure}[h]
  \begin{minipage}[c]{\linewidth}
    \centering
\includegraphics[width=.5\textwidth,,height=60mm]{meta-learning/regret.png}\includegraphics[width=.5\textwidth,height=60mm]{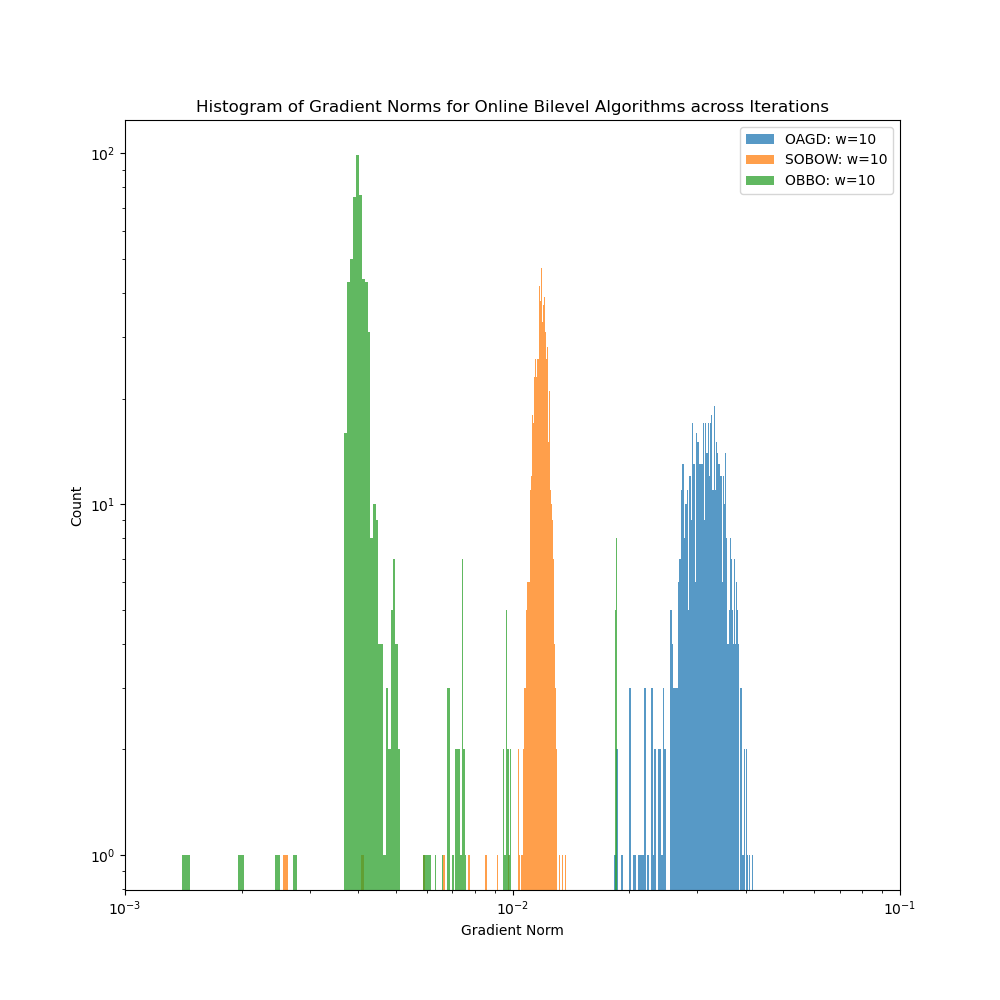} \captionof{figure}{\textbf{Left Panel}: Significant improvement with OBBO on in cumulative bilevel local regret.\\ \textbf{Right Panel}: Gradient norms across iterations are smaller for OBBO than  benchmarks with same initialization.}\label{fig:meta_learn_regret}
\end{minipage}
  \end{figure}

  \begin{figure}
      \centering
      \includegraphics[width=\textwidth,height=70mm]{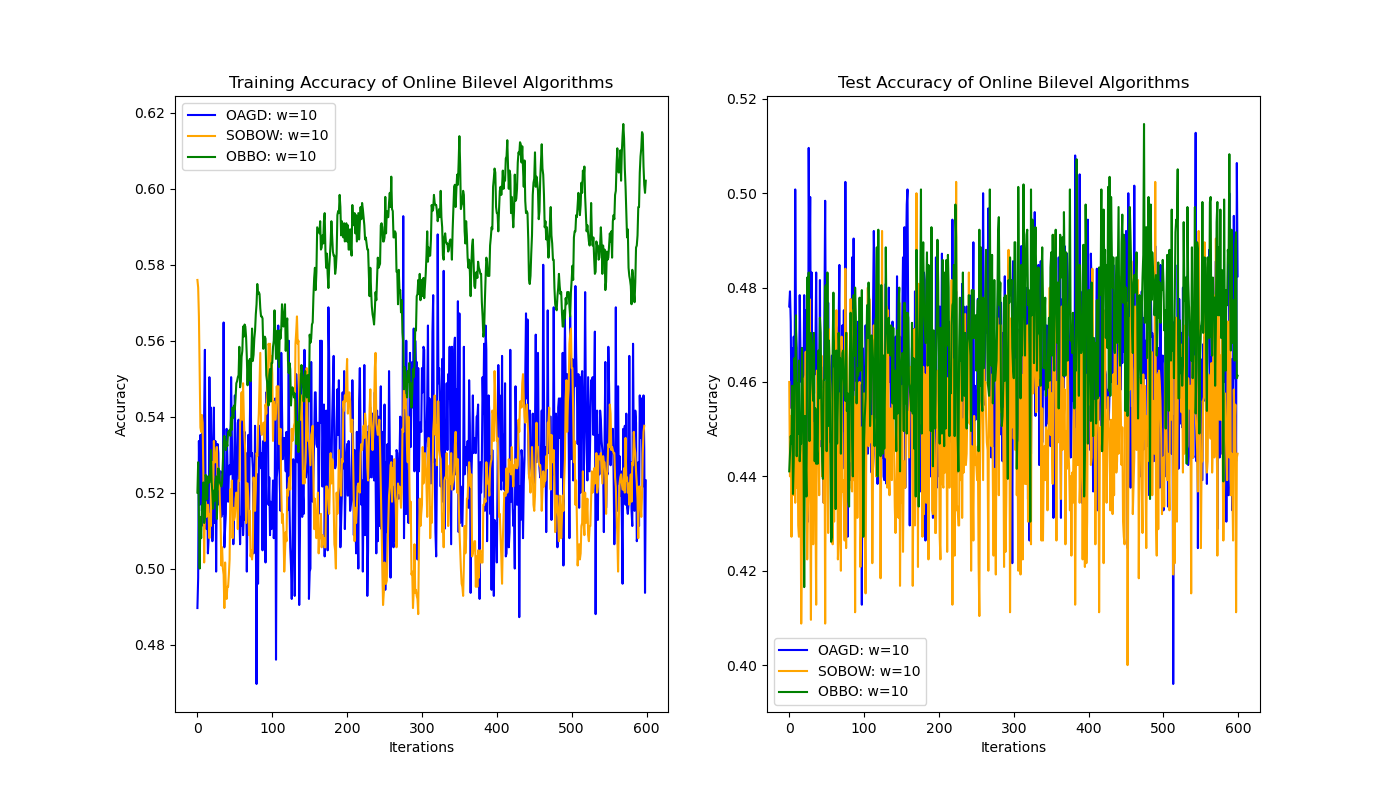}
      \caption{\textbf{Left Panel}: Higher training accuracy achieved with OBBO. \textbf{Right Panel}: Test accuracy: OBBO outperforms SOBOW while achieving OAGD performance with 10x ($w=10$) computationally cheaper update.}
      \label{fig:meta_learn_train_test}
  \end{figure}

\newpage

\end{document}